\documentclass[12pt]{amsart}
\usepackage[utf8]{inputenc}
\usepackage{mathrsfs} 
\usepackage{amssymb}
\usepackage{bm}
\usepackage{graphicx}
\usepackage[centertags]{amsmath}
\usepackage{amsfonts}
\usepackage{amsthm}
\usepackage{enumerate}
\usepackage{tocvsec2}
\usepackage[dvipsnames]{xcolor}
\usepackage[margin=1in]{geometry}

\newtheorem{theorem}{Theorem}[section]
\newtheorem*{theorem*}{Theorem}

\newtheorem{corollary}[theorem]{Corollary}
\newtheorem{lemma}[theorem]{Lemma}
\newtheorem{rem}[theorem]{Remark}

\newtheorem{proposition}[theorem]{Proposition}

\newtheorem{example}{Example}[section]
\newtheorem*{question*}{Question}

\theoremstyle{definition}

\newcommand{\rr}{\mathbb{R}}
\newcommand{\nn}{\mathbb{N}}
\newcommand{\ee}{\varepsilon}

\newcommand{\an}{\text{Ann}}

\begin{document}

\title{Szlenk index of $C(K)\widehat{\otimes}_\pi C(L)$}
\author{R.M. Causey, E. Galego, C. Samuel}

\begin{abstract} We compute the Szlenk index of the projective tensor product $C(K)\widehat{\otimes}_\pi C(L)$ of spaces $C(K), C(L)$ of continuous functions on arbitrary scattered, compact, Hausdorff spaces. In particular, we show that it is simply equal to the maximum of the Szlenk indices of the spaces $C(K), C(L)$.   We deduce several results regarding non-isomorphism of $C(K)\widehat{\otimes}_\pi C(L)$ and $C(M)$ or $C(M)\widehat{\otimes}_\pi C(N)$ for particular choices of $K,L,M,N$. 
    
\end{abstract}

\thanks{2010 \textit{Mathematics Subject Classification}. Primary: 46B03, 46B28.}
\thanks{\textit{Key words}: Asymptotic uniform smoothness, projective tensor products, Szlenk index.}

\maketitle

\section{Introduction}

Since  Grothendieck  established the theory of tensor products \cite{Grothendieck}, it has been clear that projective tensor products play a fundamental role in the geometry of Banach spaces.   However, due to the intractable nature of the projective tensor product, a number of elementary questions remain unanswered.  For example, although the isomorphism classes of $C(K)$ are well understood when $K$ is countable, compact, Hausdorff, the isomorphism classes of $C(K)\widehat{\otimes}_\pi C(L)$ are not known when $K,L$ are countable, compact, Hausdorff. 

A classical result in Banach space theory is that the isomorphism class of $C(K)$ when $K$ is an infinite,  countable, compact, Hausdorff spaces is determined by its Szlenk index, which in turn is fundamentally connected to the Cantor-Bendixson index of $K$. More precisely,  Bessaga and Pe\l czy\'{n}ski \cite{BP} showed that each such $K$ is isomorphic to $C(\omega^{\omega^\xi}+)$ for exactly one countable ordinal $\xi$. The third named author \cite{Samuel} showed that for a countable ordinal $\xi$, the Szlenk index $Sz(C(\omega^{\omega^\xi}+))$ of $C(\omega^{\omega^\xi}+)$ is equal to $\omega^{\xi+1}$.   Since the Szlenk index is an isomorphic invariant, the fact that the Szlenk index of $C(\omega^{\omega^\xi}+)$ is equal to $\omega^{\xi+1}$ implies, independently of the result of Bessaga and Pe\l czy\'{n}ski, that the spaces $C(\omega^{\omega^\xi}+)$, $\xi$ countable, are mutually non-isomorphic.  Moreover, since the Szlenk index of a Banach space is at least as large as the Szlenk index of any of its subspaces or quotients, this result actually implies that $C(\omega^{\omega^\xi}+)$ is not isomorphic to any subspace of any quotient of $C(\omega^{\omega^\zeta}+)$ when $\xi, \zeta$ are countable ordinals with $\zeta<\xi$.   One goal of the present work is to establish a partial solution to the problem of an isomorphic classification of the projective tensor products $C(K)\widehat{\otimes}_\pi C(L)$ for infinite, countable, compact, Hausdorff spaces $K,L$ using the Szlenk index as an isomorphic invariant.  Our main result regarding the Szlenk index is the following.

\begin{theorem} Let $K,L$ be compact, Hausdorff topological spaces.  Then \[Sz(C(K)\widehat{\otimes}_\pi C(L)) = \max\{Sz(C(K)), Sz(C(L))\}.\] 
\label{thefirsttheorem}
\end{theorem}

Asymptotic smoothness properties are generally stable under the formation of injective tensor products. However, the same is not true under the formation of projective tensor products.  In general, since the Szlenk index can be thought of as a quantification of the smoothness of a norm, and by classical duality the weak$^*$-convexity properties of the dual norm, a Banach space whose dual admits large, transfinite $\ell_\infty$ structures cannot have a small Szlenk index. We recall that for Banach spaces $X,Y$, $(X\widehat{\otimes}_\pi Y)^*$ is isometrically identifiable with $\mathfrak{L}(X,Y^*)$, the space of bounded, linear operators from $X$ into $Y^*$. Since $(X\widehat{\otimes}_\pi Y)^*$ is a space of operators,  it often contains large, transfinite $\ell_\infty$ structures, or even a fully copy of $\ell_\infty$, even when the factors spaces $X,Y$ separately have good smoothness properties. The most natural example is $\ell_2\widehat{\otimes}_\pi \ell_2$, in which $(e_i\otimes e_i)_{i=1}^\infty$ is isometrically equivalent to the canonical $\ell_1$ basis. Following Corollary \ref{ms}, we provide further examples of the delicate nature of asymptotic smoothness nature under the formation of projective tensor products by noting the existence of a subspace $X$ of $C(\omega^\omega+)$ such that $X\widehat{\otimes}_\pi X$ fails to be Asplund.  The stability of asymptotic smoothness properties under projective tensor products depends not only on factor spaces themselves considered separately, but on properties of the pair which are checked on $\mathfrak{L}(X,Y^*)$.  Our proof of Theorem \ref{thefirsttheorem} will use Grothendieck's theorem to obtain the required properties of $\mathfrak{L}(C(K), C(L)^*)$.

As a result of Theorem \ref{thefirsttheorem}, we establish the following. 

\begin{theorem} If $\xi, \zeta, \mu, \nu$ are countable ordinals such that $C(\omega^{\omega^\xi}+)\widehat{\otimes}_\pi C(\omega^{\omega^\zeta}+)$ is isomorphic to $C(\omega^{\omega^\mu}+)\widehat{\otimes}_\pi C(\omega^{\omega^\nu}+)$, then $\max\{\xi, \zeta\} = \max\{\mu, \nu\}.$
\label{suss}
\end{theorem}

We will actually reach the conclusion of Theorem \ref{suss} under the weaker conclusion that each of the spaces $C(\omega^{\omega^\xi}+)\widehat{\otimes}_\pi C(\omega^{\omega^\zeta}+)$,  $C(\omega^{\omega^\mu}+)\widehat{\otimes}_\pi C(\omega^{\omega^\nu}+)$ is isomorphic to a subspace of a quotient of the other.

Our result combines Grothendieck's theorem \cite{Gro} with the notion of $\xi$-asymptotic uniform flatness introduced in \cite{Causey}.   Grothendieck's theorem gives a fundamental inequality on the  $2$-weakly summing norm of a sequence $(f_i\otimes g_i)_{i=1}^\infty\subset C(K)\widehat{\otimes}_\pi C(L)$ in terms of the weakly $1$-summing norm of $(f_i)_{i=1}^\infty$ and the supremum norm of $(g_i)_{i=1}^\infty$.  Some recent results from  \cite{CD2} show that the full strength of Grothendieck's theorem may not be needed, although the basic ingredients used there are reminiscent of the ingredients of Grothendieck's theorem. More precisely, the key insight in \cite{CD2} regarded the $q$-weakly summing norm of a sequence $(x_i\otimes y_i)_{i=1}^\infty\subset X\widehat{\otimes}_\pi Y$, where $(x_i)_{i=1}^\infty$ is $1$-weakly summing, $(y_i)_{i=1}^\infty$ is bounded, and $Y^*$ has cotype $q$.   The notion of $\xi$-asymptotic uniform flatness, together with the fact that $Sz(C(K))=\omega^{\xi+1}$ if and only if $C(K)$ admits an equivalent $\xi$-asymptotically uniformly flat norm, will allow us to find weakly $1$-summing sequences among prescribed convex combinations of weakly null trees in the projective tensor product $C(K)\widehat{\otimes}_\pi C(L)$. Grothendieck's inequality will then allow us to deduce that certain convex combinations in the branches of weakly null trees in the projective tensor product are weakly $2$-summing, and therefore the projective tensor product admits an equivalent $\xi$-$2$-asymptotically uniformly smooth norm for an appropriately large $\xi$.  This result is a transfinite version of the result of Dilworth and Kutzarova from \cite{DK}, where it was shown that $c_0\widehat{\otimes}_\pi c_0$ admits an equivalent norm which is $2$-asymptotically uniformly smooth, which is the $\xi=0$ case of our result for separable $C(K)$ spaces.

Throughout, unless otherwise stated, we assume all Banach spaces are infinite dimensional and all subspaces are closed,  linear subspaces.  We close the introduction by  briefly describing  some non-standard notation and terminology which we will use, and where the corresponding definitions can be found. In Section $2$, we provide notation and terminology related to trees and rank. We also introduce the notation $T.D$, where $T$ is a tree and $D$ is a directed set. These will be convenient index sets, where the tree $T$ controls the rank of the structure $T.D$ and $D$ is a directed set. Section $3$ introduces notions related to the Szlenk index, higher order asymptotic smoothness, and weakly null trees. We will be concerned with weakly null trees, so most of our applications of $T.D$ will involve the choice of directed set $D=CD(X)$, where $X$ is a Banach space and $CD(X)$ is the set of finite codimensional subspaces of $X$. In order to separate the required combinatorial results from the specifics of our applications to weakly null trees, we build the weak nullity of the trees into the index set itself using the notion of normal weak nullity, defined in Section $3.2$.  Also in Section $3.2$, we recall the notions  of $\xi$-$p$-asymptotic uniform smoothness ($\xi$-$p$-AUS) and $\xi$-asymptotic uniform flatness ($\xi$-AUF).  Since we will be concerned with $C(K)$ spaces, we will also use the notion of \emph{normal pointwise nullity}, in which the pointwise nullity of a collection of $C(K)$ is built into the index set of the collection.  This notion is defined in Section $6$. There, for a compact, Hausdorff space $K$ and a finite subset $F$ of $K$, we also introduce the notion \[\an(F)=\{f\in C(K):f|_F\equiv 0\}.\]

\section{Trees and Games}

\subsection{Trees}

Given a set $\Omega$, $\Omega^{<\omega}$ will denote the set of finite sequences of elements of  $\Omega$. This includes the empty sequence, denoted $\varnothing$. We let $\Omega^\omega$ denote the set of infinite sequences whose members lie in $\Omega$ and let $\Omega^{\leqslant \omega}= \Omega^\omega\cup \Omega^{<\omega}$. For $t\in \Omega^{<\omega}$, we let $|t|$ denote the length of $t$. For $t\in \Omega^{\leqslant \omega}$ and $n< \omega$, we let $t|n$ be the initial segment $s$ of $t$ such that $|s|=n$.      For $s,t\in \Omega^{\leqslant \omega}$, we write $s< t$ if $s=t|n$ for some $n<\omega$.   If $t\in \Omega^{<\omega}$ satisfies $0<|t|<\omega$, we let $t^-=t|(|t|-1)$.    That is, $t^-$ is the immediate predecessor of $t$.  For $s\in \Omega^{<\omega}$ and $t\in \Omega^{\leqslant \omega}$, we let $s\smallfrown t$ denote the concatenation of $s$ with $t$. 

 A  subset $T$ of $\Omega^{<\omega}\setminus \{\varnothing\}$ is said to be a \emph{tree on} $\Omega$ (or simply a \emph{tree}) if  $\varnothing< s<t$ and $t\in T$  implies $s\in T$.  A subset $T$ of $\Omega^{<\omega}$   is said to be a \emph{rooted tree on} $\Omega$ (or simply a \emph{rooted tree}) if  $s<t$ and $t\in T$ implies $s\in T$.  A rooted tree is said to be \emph{hereditary} if it contains all subsequences of its members.   We let $MAX(T)$ denote the set consisting of all maximal members of $T$ with respect to the initial segment ordering.   A tree  $T\subset \Omega^{<\omega}$ is said to be \emph{pruned} if it is non-empty and $MAX(T)=\varnothing$.

Given a tree $T$, we define the \emph{derivative of } $T$, denoted by $T'$, by $T'=T\setminus MAX(T)$.   We define the transfinite derivatives of $T$ by $$T^0=T,$$ $$T^{\xi+1}=(T^\xi)',$$ and if $\xi$ is a limit ordinal, $$T^\xi=\bigcap_{\zeta<\xi}T^\zeta.$$   We say $T$ is \emph{well-founded} if there exists $\xi$ such that $T^\xi=\varnothing$, and in this case we define the \emph{rank} of $T$ by $$\text{rank}(T)=\min\{\xi: T^\xi=\varnothing\}.$$   A tree which is not well-founded is said to be \emph{ill-founded}. If $T$ is ill-founded, we agree to the convention that $\text{rank}(T)=\infty$.  Every pruned tree is ill-founded.  Given a tree $T$ on $\Omega$, we define the \emph{body} of $T$ by $$[T]=\{\tau\in \Omega^\omega:(\forall n\in\nn)(\tau|n\in T)\}.$$  Note that $T$ is well-founded if and only if $[T]=\varnothing$.    

Given a tree $T$ on $\Omega$ and $t\in \Omega^{<\omega}$, we let $$T_t=\{s\in \Omega^{<\omega}\setminus \{\varnothing\}: t\smallfrown s\in T\}.$$ Note that this is also a tree, which is well-founded (resp. pruned) if $T$ is.    An easy proof by induction shows that for any tree $T$, any $t\in \Omega^{<\omega}$, and any ordinal $\xi$, $$(T^\xi)_t=(T_t)^\xi.$$  Therefore the notation $T_t^\xi$ can be used unambiguously.

\begin{rem}\label{r1}
Let $\zeta$ an ordinal be given. There exists a tree of rank $\zeta.$
\end{rem}
\begin{proof}
Let 
$S_\zeta=\{\,(\nu_i)_{i=1}^n\,:\,0\leqslant\nu_n<\dots<\nu_1<\zeta\,\}.
$ Note that $S_\zeta=\varnothing$ if and only if $\zeta=0.$ Also,  $S_\zeta$ is a tree on the interval of ordinals $[0,\zeta)$, and $
 (\nu_i)_{i=1}^n\in \textrm{MAX}(S_\zeta)$ if and only if $\nu_n=0$, so \[S_\zeta'=\{\,(\nu_i)_{i=1}^n\,:\,1\leqslant\nu_n<\dots<\nu_1<\zeta\,\}.\]   More generally, an easy induction argument yields that for any ordinal $\delta$, 
 
 \[S_\zeta^\delta=\{\,(\nu_i)_{i=1}^n\,:\,\delta\leqslant\nu_n<\dots<\nu_1<\zeta\,\}.\] Of course, this means $S_\zeta^\delta=\varnothing$ if and only if $\delta\geqslant \zeta$, from which it follows that  $\textrm{rank}(S_\zeta)=\zeta.$
\end{proof}

Given a tree $T$ on $\Omega$ and a non-empty set $D$, we let $$T.D=\{(\zeta_i, u_i)_{i=1}^n \in (\Omega\times D)^{<\omega}: (\zeta_i)_{i=1}^n\in T\}.$$  We consider the members of $T.D$ as sequences of pairs. That is, $(\zeta_i, u_i)_{i=1}^n$ is treated as a sequence of length $n$ whose $i^{th}$ member is $(\zeta_i, u_i)$.  Therefore if $a=(\zeta_i,u_i)_{i=1}^n\in T.D$, then for $0\leqslant m\leqslant n$, $a|m=(\zeta_i, u_i)_{i=1}^m$.  Given a tree $T$, $t=(\zeta_i)_{i=1}^n\in T$ and $v=(u_i)_{i=1}^n\in D^{<\omega}$, we let $t.v=(\zeta_i, u_i)_{i=1}^n\in T.D$.  We also agree that $\varnothing.\varnothing=\varnothing$.     Note that if $T$ is a tree on $\Omega$, then $T.D$ is a tree on $\Omega\times D$.   Furthermore, for any ordinal $\xi$,  any $t\in \Omega^{<\omega}$, and $v\in D^{<\omega}$ with $|t|=|v|$, it holds that  $T^\xi.D=(T.D)^\xi$ and $T_t.D=(T.D)_{t.v}$. In particular, $T.D$ is pruned if $T$ is, and $\text{rank}(T)=\text{rank}(T.D)$. We also define $[T].D=\{(\zeta_i, u_i)_{i=1}^\infty\in (\Omega\times D)^\omega: (\zeta_i)_{i=1}^\infty\in [T]\}$.  For $\tau=(\zeta_i)_{i=1}^\infty\in \Omega^\omega$ and $\upsilon=(u_i)_{i=1}^\infty\in D^\omega$, we let $\tau. \upsilon=(\zeta_i, u_i)_{i=1}^\infty$, which is treated as the member of $(\Omega\times D)^\omega$ whose $i^{th}$ member is $(\zeta_i, u_i)$.

\subsection{Games}

Let $B$ be a set and let $D$   be a non empty subset of the power set of $B$ Let $\mathfrak{g}$   be a  subset of the power set of $B$ such that   for every $u\in D$, there exists $G\in \mathfrak{g}$ such that $G\subset u$.

Let $T$ be a pruned tree and let $\mathcal{E}$ be a subset of $[T].\mathfrak{g}$.   We define a two player game. Player $S$ chooses $u_1\in D$ and  $\zeta_1\in \Omega$ such that $(\zeta_1)\in T$. Player $V$ chooses $G_1\in \mathfrak{g}$ such that $G_1\subset u_1$. Player $S$ chooses $u_2\in D$ and $\zeta_2\in \Omega$ such that $(\zeta_1, \zeta_2)\in T$. Player $V$ chooses $G_2\in \mathfrak{g}$ such that $G_2\subset u_2$. Play continues in this way until $\tau=(\zeta_i)_{i=1}^\infty\in [T]$,  $\gamma=(G_i)_{i=1}^\infty \in \mathfrak{g}^\omega$, and $(u_i)_{i=1}^\infty \in D^\omega$ are chosen. Player $S$ wins if $\tau. \gamma\in \mathcal{E}$, and Player $V$ wins otherwise. We refer to $\mathcal{E}$ as the \emph{target set}. We refer to this game as the  $\mathcal{E}$ \emph{game on} $\mathfrak{g},T,D$, or simply the $\mathcal{E}$ game if $\mathfrak{g}$, $T$, and $D$ are understood.    

 If $n$  turns of the game have been played ($n= 1, \ldots$), resulting in a choice $t\in T$ for Player $S$ and $g\in \mathfrak{g}^{<\omega}$ for Player $V$, the remainder of the game is equivalent to a new game with target set $$\mathcal{E}_{t.g}=\{\alpha\in [T_t].\mathfrak{g}: t.g\smallfrown \alpha\in \mathcal{E}\}$$  played on the tree $T_t$.

A \emph{strategy for Player} $S$ in the $(\mathcal{E}, \mathfrak{g}, T)$ game (sometimes just called a \emph{strategy for Player} $S$) is a function $\chi:\{\varnothing\}\cup T.\mathfrak{g}\to \Omega\times D$ such that if $\chi(\varnothing)=(\zeta, u)$, then $(\zeta)\in T$, and if $\chi(t.g)=(\zeta, u)$, then $t\smallfrown(\zeta)\in T$.    A \emph{strategy for Player} $V$ in the $(\mathcal{E}, \mathfrak{g}, T)$ game (sometimes just called a \emph{strategy for Player} $V$) is a function $\psi: T.D\to \mathfrak{g}$ such that if $\psi((\zeta_i, u_i)_{i=1}^n)=G$, then $G\subset u_n$.

Given a strategy $\chi$ for Player $S$, $t\in T$, and $g\in \mathfrak{g}^{<\omega}$ with $|t|=|g|$, we say $t.g$ is $\chi$-\emph{admissible} if $t=g=\varnothing$ or if for each $0\leqslant i<|t|$, if $\chi((t.g)|i)=(\zeta, u)$, then $G_{i+1}\subset u$, where $G_{i+1}$ is the $i+1^{st}$ member of the sequence $g$.  For $\tau\in [T]$ and $\gamma\in \mathfrak{g}^\omega$, we say $\tau.\gamma$ is $\chi$-admissible if $\tau.\gamma |n$ is $\chi$-admissible for all $n<\omega$.

A strategy $\psi$ for Player $V$ in the $(\mathcal{E}, \mathfrak{g}, T)$ game is called a \emph{winning strategy for Player} $V$ if for any $\tau.\upsilon\in [T].D$, if $G_n=\psi(\tau.\upsilon|n)$ for $n=1, 2, \ldots$, and if $\gamma=(G_n)_{n=1}^\infty$, then $\tau.\gamma\notin \mathcal{E}$.   A strategy $\chi$ for Player $S$ in the $(\mathcal{E}, \mathfrak{g}, T)$ game is called a \emph{winning strategy for Player} $S$ if whenever $\alpha\in [T].\mathfrak{g}$  is $\chi$-admissible, then $\alpha\in \mathcal{E}$.     A strategy $\chi$ for Player $S$ in the $(\mathcal{E}, \mathfrak{g}, T)$ game is called a \emph{defensive strategy for Player} $S$ if whenever $t\in T$ and $g\in \mathfrak{g}^{<\omega}$ are such that $t.g$ is $\chi$-admissible, then Player $V$ does not have a winning strategy in the  $(\mathcal{E}_{t.g}, \mathfrak{g}, T_t)$ game.     Informally, if Player $S$ plays according to the strategy $\chi$, then Player $V$ will not have a winning strategy at any point during the progress of the game. 

We say that a target set $\mathcal{E}\subset [T].\mathfrak{g}$ is \emph{closed} if whenever  $\alpha\in[T].\mathfrak{g}\setminus \mathcal{E}$, there exists $n\in\nn$ such that $$\{\beta\in [T].\mathfrak{g}: \beta|n=\alpha|n\}\cap \mathcal{E}=\varnothing.$$ The concepts behind the following result are standard, but due to the specificity of our construction, we give the proof. 

\begin{proposition} Let $T$ be a pruned tree on a set $\Omega$. Let $B$ be a non-empty set and let $D, \mathfrak{g}$ be  non empty  subsets of the power set of $B.$ Suppose that ${\varnothing}\notin D,$ ${\varnothing}\notin\mathfrak{g}$ and any $u\in D$, there exists $G\in \mathfrak{g}$ such that $G\subset u$.  The following hold. 

\begin{enumerate}[(i)]\item If for every $(\zeta)\in T$ and $u\in D$, there exists $G=G(\zeta,u)\in \mathfrak{g}$ such that $G\subset u$ and Player $V$ has a winning strategy in the  $(\mathcal{E}_{(\zeta,G)},\mathfrak{g}, T_{(\zeta)})$ game, then Player $V$ has a winning strategy in the $(\mathcal{E},\mathfrak{g}, T)$ game. \item For any target set $\mathcal{E}\subset [T].\mathfrak{g}$, either Player $S$ has a defensive strategy or Player $V$ has a winning strategy in the $(\mathcal{E}, \mathfrak{g}, T)$ game. \item If $\mathcal{E}$ is closed, then any defensive strategy for Player $S$ in the $(\mathcal{E}, \mathfrak{g}, T)$ game is a winning strategy for Player $V$ in the $(\mathcal{E},\mathfrak{g}, T)$ game. \item If $\mathcal{E}$ is closed, exactly one of Player $S$ and Player $V$ has a winning strategy in the $(\mathcal{E}, \mathfrak{g}, T)$ game. \item If there exists a subset $\mathcal{F}$ of $T.\mathfrak{g}$ such that $\mathcal{E}=[\mathcal{F}]$, then $\mathcal{E}$ is closed. \end{enumerate}

\label{steph}
\end{proposition}

\begin{proof}$(i)$ For each $(\zeta)\in T$ and $u\in D$, let $G_{(\zeta,u)}\subset u$ be such that Player $V$ has a winning strategy  $\psi_{\zeta,u}:T_{(\zeta)}.D\to \mathfrak{g}$   in the $(\mathcal{E}_{(\zeta, G)}, \mathfrak{g}, T_{(\zeta)})$ game. Define $\psi:T.D\to \mathfrak{g}$ by letting $\psi((\zeta,u))=G_{(\zeta,u)}$ and for $a\in T.D$ with $|a|>1$, write $a=(\zeta,u)\smallfrown a_1$ and let $\psi(a)=\psi_{\zeta,u}(a_1)$. It is straightforward to verify that $\psi$ is a winning strategy in the $(\mathcal{E}, \mathfrak{g}, T)$ game.

$(ii)$ Assume Player $V$ has no winning strategy in the $(\mathcal{E}, \mathfrak{g}, T)$ game. We define a defensive strategy $\chi$ for Player $S$ in the game. More precisely, we define $\chi(a)$ for $a\in \{\varnothing\}\cup T.\mathfrak{g}$ by induction on $|a|$.   Since Player $V$ does not have a winning strategy in the $(\mathcal{E}, \mathfrak{g}, T)$ game, by negating the conditions in $(i)$,  there must exist  $\zeta_0\in \Omega$  and $u_0\in D$  such that $(\zeta)\in T$ and for any $G\in 2^u\cap   \mathfrak{g}$,  Player $V$ does not have a winning strategy in the $(\mathcal{E}_{(\zeta, G)}, \mathfrak{g}, T_{(\zeta)})$ game. Define $\chi(\varnothing)=(\zeta_0,u_0)$.

Now assume that for some $a=t.\gamma\in T.\mathfrak{g}$, $\chi(a)=(\zeta,u)$ has been defined, which means $t\smallfrown (\zeta)\in T$.  Assume further that if $a=(\zeta_i, G_i)_{i=1}^n$ and for any $0\leqslant i<n$, if $\chi(a|i)=(\zeta', u')$ and $G_i\subset u'$, then for any  $G\in \mathfrak{g}$  with $G\subset u$, Player $V$ has no winning strategy in the $(\mathcal{E}_{a\smallfrown (\zeta, G)}, \mathfrak{g}, T_{t\smallfrown (\zeta)})$ game.        For $G\in \mathfrak{g}$, we define $\chi(a\smallfrown (\zeta, G))$ in cases.    If $G\not\subset u$, define $\chi(a\smallfrown(\zeta, G))$ arbitrarily.   For the remaining cases, assume $G\subset u$. If $a=\varnothing$, then since $G\subset u$, by $(i)$ and our assumption that Player $V$ has no winning strategy in the $(\mathcal{E}_{(\zeta,G)},\mathfrak{g}, T_{(\zeta)})$ game, there exists $(\zeta')\in T_t$ and $u'\in D$ such that for any $G'\in \mathfrak{g}$ such that $G'\subset u'$, Player $V$ does not have a winning strategy in the $(\mathcal{E}_{(\zeta,u)\smallfrown(\zeta',u')}, \mathfrak{g}, T_{(\zeta, \zeta')})$ game. In this case, define $\chi(a\smallfrown(\zeta, G))=(\zeta',u')$.   If $a=(\zeta_i, G_i)_{i=1}^n$ and for any $0\leqslant i<n$, if $\chi(a|i)=(\zeta', u')$ and $G_i\subset u'$, then since $G\subset u$, our assumptions together with $(i)$ yield the existence of some $(\zeta')\in T_t$ and $u'\in D$ such that for any $G'\in \mathfrak{g}$ such that $G'\subset u'$, Player $V$ does not have a winning strategy in the   $(\mathcal{E}_{a\smallfrown (\zeta, G)\smallfrown (\zeta',G')}, \mathfrak{g}, T_{t\smallfrown (\zeta,\zeta')})$ game.    Define $\chi(a\smallfrown(\zeta, G))=(\zeta',u')$.   In the remaining case that $a=(\zeta_i, G_i)_{i=1}^n$ and for some $0\leqslant i<n$, $\chi(a|i)=(\zeta',u')$ and $G_i\not\subset '$, define $\chi(a\smallfrown (\zeta,u))$ arbitrarily.   This completes the recursive process.    It is routine to verify that $\chi$ is a defensive strategy for Player $S$ in the $(\mathcal{E}, \mathfrak{g}, T)$ game. 

$(iii)$ Assume $\mathcal{E}$ is closed. Assume $\chi$ is a defensive strategy for Player $S$ in the $(\mathcal{E}, \mathfrak{g}, T)$ game and $\alpha=\tau.\phi\in [T].\mathfrak{g}$ is $\chi$-admissible. Seeking a contradiction, assume $\alpha\in [T].\mathfrak{g}\setminus \mathcal{E}$. Then there exists $n\in\nn$ such that $$\{\beta\in [T].\mathfrak{g}: \beta|n=\alpha|n\}\cap \mathcal{E}=\varnothing.$$    Then $\alpha|n$ is $\chi$-admissible, but Player $V$ has a winning strategy in the $(\mathcal{E}_{\alpha|n}, \mathfrak{g}, T_{\tau|n})$ game. In fact, any strategy for Player $V$ is a winning strategy in this game, since $\mathcal{E}_{\alpha|n}=\varnothing$.

$(iv)$ If $\mathcal{E}$ is closed, then either Player $S$ has a defensive (and therefore winning) strategy in the $(\mathcal{E}, \mathfrak{g}, T)$ game or Player $V$ has a winning strategy in the $(\mathcal{E}, \mathfrak{g}, T)$ game. Therefore at least one of the two players has a winning strategy. Of course, at most one of the two players can have a wining strategy. 

$(v)$ Assume $\mathcal{E}=[\mathcal{F}]$ for some $\mathcal{F}\subset T.\mathfrak{g}$.   Fix $\alpha\in [T].\mathfrak{g}\setminus \mathcal{E}$. Since $\alpha\notin \mathcal{E}=[\mathcal{F}]$, there exists $n\in\nn$ such that $\alpha|n\notin \mathcal{F}$. Then for any $\beta\in [T].\mathfrak{g}$ such that $\beta|n=\alpha|n$, it follows that $\beta|n=\alpha|n\notin \mathcal{F}$, and $\beta\in [T].\mathfrak{g}\setminus [\mathcal{F}]=[T].\mathfrak{g}\setminus \mathcal{E}$.

\end{proof}

Item $(iii)$ above states that closed games, in the narrower sense that we have defined games,  are determined. A remarkable theorem of Martin \cite{Martin} yields in the  very general setting of Gale-Stewart games that games with Borel target sets are determined. The games above are special cases of such Gale-Stewart games, and we could have cited Martin's theorem rather than providing a direct proof. However, we find the direct proof in the case of closed games to be simple and illustrative.

\begin{lemma} Let $T$ be a pruned tree on a set $\Omega$. Let $B$ be a non-empty set and let $D, \mathfrak{g}$ be subsets of the power set of $B$ such that for any $u\in D$, there exists $G\in \mathfrak{g}$ such that $G\subset u$. Then for a target set $\mathcal{E}$, Player $V$ has a winning strategy in the $(\mathcal{E}, \mathfrak{g}, T)$ game if and only if there exists a collection $(G_a)_{a\in T.D}\subset \mathfrak{g}$  such that 
\begin{enumerate}[(i)]\item for each $a=(\zeta_i, u_i)_{i=1}^n\in T.D$, $G_a\subset u_n$, \item for each $\alpha=\tau.\upsilon\in [T].D$, $\tau.(G_{\alpha|n})_{n=1}^\infty\in [T].\mathfrak{g}\setminus  \mathcal{E}$. \end{enumerate}

\label{tree}
\end{lemma}

\begin{proof} If $\psi$ is a winning strategy for Player $V$ in the $(\mathcal{E}, \mathfrak{g}, T)$ game, then the collection given by $G_a=\psi(a)$ satisfies the conclusions of the lemma. If $(G_a)_{a\in T.D}$ is as in the lemma, then $\psi(a)=G_a$ defines a winning strategy $\psi$ for Player $V$ in the $(\mathcal{E}, \mathfrak{g}, T)$ game.

\end{proof}

Negating one direction of the previous result yields the following. 

\begin{corollary} Let $T$ be a pruned tree on a set $\Omega$. Let $B$ be a non-empty set and let $D, \mathfrak{g}$ be subsets of the power set of $B$ such that for any $u\in D$, there exists $G\in \mathfrak{g}$ such that $G\subset u$. Let $\mathcal{E}$ be a target set and assume Player $S$ has a winning strategy in the $(\mathcal{E}, \mathfrak{g}, T)$ game. Then  for any  collection $(G_a)_{a\in T.D}\subset  \mathfrak{g}$ such that for each $a=(\zeta_i, u_i)_{i=1}^n\in T.D$, $G_a\subset u_n$, there exists $\alpha=\tau.\upsilon\in [T].D$ such that $\tau.(G_{\alpha|n})_{n=1}^\infty \in \mathcal{E}$. 

\label{tree2}
\end{corollary}

\section{Szlenk index, weakly null trees,  and moduli}

A real Banach space $X$ is said to be \emph{Asplund} if every continuous, convex function defined on a convex,  open subset $U$ of $X$ is Fr\'{e}chet differentiable on a dense $G_\delta$ subset of $X$ \cite[Definition 5.1]{DGZ}. A complex Banach space is said to be Asplund if it is Asplund as a real Banach space.   In particular, the norm of an Asplund space is Fr\'{e}chet differentiable on a dense $G_\delta$ subset of $X$.  On the other hand, every non-Asplund space admits an equivalent rough norm \cite[Theorem 5.3]{DGZ}. Rough norms can be thought of as ``uniformly non-Fr\'{e}chet differentiable.''  Therefore Asplundness is fundamentally connected to smoothness of the norm. Another characterization of the Asplund property is weak$^*$-fragmentability of the dual ball, $B_{X^*}$. This means that for any non-empty $K\subset B_{X^*}$ and any  $\ee>0$, there exists a weak$^*$ open set $U$ such that $K\cap U\neq \varnothing$ and the norm diameter of $K\cap U$ is less than $\ee$.  It is this characterization of the Asplund property which the Szlenk index, defined below, will characterize.  Therefore smallness of the Szlenk index is a smoothness condition. More precise information is encoded not just in the Szlenk index of a space, but in the growth rate of the $\ee$-Szlenk indicies $Sz(X,\ee)$ as $\ee$ decreases to $0$. For example, although we will not use this direct argument, one can generally prove that for a scattered, compact, Hausdorff space $K$ whose Cantor-Bendixson index lies in $(\omega^\xi, \omega^{\xi+1})$, $C(K)$ is not isomorphic to  $C(K)\widehat{\otimes}_\pi C(K)$ because \[\min \{n\in\nn: Sz(C(K),\ee)\leqslant \omega^\xi n\}\] grows on the order of $1/\ee$, while \[\min \{n\in\nn: Sz(C(K)\widehat{\otimes}_\pi C(K))\leqslant \omega^\xi n\}\] grows on the order of $1/\ee^2$. The degree of smoothness can be encoded in the modulus of $\xi$-asymptotic uniform smoothness, which we also define below, and which are the subject of our renorming theorems.

\subsection{Szlenk index}

Throughout, we let $\mathbb{K}$ denote the scalar field, which is either $\rr$ or $\mathbb{C}$.  For a Banach space $X$, we let $B_X$ denote the closed unit ball of $X$ and we let $S_X$ denote the unit sphere of $X$.  In particular, $B_\mathbb{K}$ denotes the set of scalars with modulus not exceeding $1$.

Given a Banach space $X$, a weak$^*$-compact subset $K$ of $X^*$, and $\ee>0$, we define the $\ee$-\emph{Szlenk derivation} $s_\ee(K)$ of $K$ to be the set of those $x^*\in K$ such that for any weak$^*$-neighborhood $V$ of $x^*$, $\text{diam}(V\cap K)>\ee$.  We define the transfinite derivations by $$s_\ee^0(K)=K,$$ $$s^{\xi+1}_\ee(K)= s_\ee(s^\xi_\ee(K)),$$ and for a limit ordinal $\xi$, $$s_\ee^\xi(K)=\bigcap_{\zeta<\xi}s_\ee^\zeta(K).$$  It is easy to see that each derivation $s^\xi_\ee(K)$ is also weak$^*$-compact.  If there exists an ordinal $\xi$ such that $s_\ee^\xi(K)=\varnothing$, we let $Sz(K,\ee)$ denote the minimum such $\xi$. If no such $\xi$ exists, we agree to the convention that $Sz(K, \ee)=\infty$.  We also agree to the convention that $\xi<\infty$ for each ordinal $\xi$, so that $Sz(K,\ee)<\infty$ means that an ordinal $\xi$ exists such that $s^\xi_\ee(K)=\varnothing$.       We define $Sz(K)=\sup_{\ee>0}Sz(K, \ee)$, where $\sup_{\ee>0}Sz(K, \ee)=\infty$ if $Sz(K, \ee)=\infty$ for some $\ee>0$.     If $X$ is a Banach space, we let $Sz(X, \ee)=Sz(B_{X^*}, \ee)$ and $Sz(X)=Sz(B_{X^*})$. The condition that $Sz(X)<\infty$ is equivalent to $X$ being an Asplund space  
 \cite[Theorem 3.10]{Lancien}.

We recall that for a topological space $K$ and a subset $L$ of $K$, the \emph{Cantor-Bendixson derivative} of $K$ is the subset of $L$ consisting of those points in $L$ which are not isolated in $L$.    We denote the Cantor-Bendixson derivative of $L$ by $L'$.    We define the transfinite derivatives by $$L^0=L,$$ $$L^{\xi+1}=(L^\xi)',$$ and if $\xi$ is a limit ordinal, $$L^\xi=\bigcap_{\zeta<\xi}L^\zeta.$$   We say that $K$ is \emph{scattered} if there exists an ordinal $\xi$ such that $K^\xi=\varnothing$, and in this case, we define the \emph{Cantor-Bendixson index} $CB(K)$ of $K$ to be the minimum $\xi$ such that $K^\xi=\varnothing$.    We note that $K$ is scattered if and only if any non-empty subset of $K$ has an isolated point.       If $K$ is compact, then $CB(K)$ cannot be a limit ordinal, and since by convention we exclude the empty set from our definitions of topological space, $CB(K)$ cannot be zero. Therefore for a compact, Hausdorff topological space, $CB(K)$ must be a successor ordinal.

We note that if $K=[0,\omega^\xi]$, the ordinal interval with its usual compact order topology, then for any ordinal $\zeta\leqslant \xi$, $$K^\zeta=\{\omega^\xi\}\cup \{\omega^{\ee_1}+\ldots +\omega^{\ee_n}: \xi>\ee_1\geqslant \ldots \geqslant \ee_n\geqslant \zeta\} =\{\omega^\zeta \theta: \theta\leqslant \omega^\mu\},$$ where $\zeta+\mu=\xi$.   Therefore $K^\xi=\{\omega^\xi\}$ and $CB(K)=\xi+1$.    

The following Szlenk index computations were shown in \cite{Samuel} for countable $K$, in \cite{Brooker} for $K=[0,\xi]$, $\xi$ an arbitrary ordinal,  and in \cite{Causey} for general $K$.

\begin{theorem} Let $K$ be a compact, Hausdorff topological space. Then $C(K)$ is Asplund if and only if $K$ is scattered. Moreover, if $K$ is infinite and scattered and $\xi$ is the unique ordinal such that $\omega^\xi< CB(K)<\omega^{\xi+1}$, then $Sz(C(K))=\omega^{\xi+1}$. 

In particular, for any ordinal $\xi$, $Sz(C[0, \omega^{\omega^\xi}])=\omega^{\xi+1}$.

\label{Sz1}
\end{theorem}

Of course, if $K$ is finite, then $C(K)\approx \ell_\infty^n$ for some $n$, and $Sz(C(K))=1$, since $B_{C(K)^*}$ is norm compact in this case.

\subsection{Weakly null trees and weak derivatives}

For convenience, whenever $T$ is a tree and $(G_a)_{a\in T}$ is a collection indexed by $T$, we use the notation $(G_b)_{b\leqslant a}$ to refer to  $(G_b)_{\varnothing<b\leqslant a}$ rather than  $(G_b)_{\varnothing\leqslant  b\leqslant a}$. 

Suppose that $T$ is a tree  and $X$ is a Banach space.   We say a collection $(x_t)_{t\in T}$ of $ X$ is \emph{weakly null} provided that for each for all $\xi$ and all $t\in (T\cup \{\varnothing\})^{\xi+1}$, \[ 0 \in \overline{\{x_s: s\in T^\xi, s^-=t\}}^\text{weak}.\] 

It will often be useful to assume the weakly null trees $(x_t)_{t\in T}$ are indexed by a tree $T$ which has some specific form, namely $\Gamma.D$ for some directed set $D$.   There are two natural choices for such a $D$. We can let $D$ be any neighborhood basis at $0$ with respect to the relative weak topology on $B_X$. Alternatively, we can let $$D=CD(X):=\{B_Z: Z\textrm{  subspace of } X \textrm{ and }  \dim (X/Z)<\infty\}.$$

For  $D=CD(X)$ and a tree $\Gamma$, we say that a collection  $(x_a)_{a\in \Gamma.D}$ of $X$ is \emph{normally weakly null} if whenever $a=(\zeta_i, u_i)_{i=1}^n\in \Gamma.D$, $x_a\in u_n$.  Such a collection is weakly null according to the definition given above. To see this, let  $a=t.v\in (\{\varnothing\}\cup \Gamma.D)^{\xi+1})$, there exists $\zeta$ such that $t\smallfrown (\zeta)\in \Gamma^\xi$. Then for any $u\in D,$   $a\smallfrown(\zeta,u)\in (\Gamma.D)^\xi,$ and $a=(a\smallfrown(\zeta,u))^-.$ 
Since $x_{a\smallfrown(\zeta, u)}\in u$, $$0\in \overline{\{x_{a\smallfrown(\zeta, u)}:u\in D \}}^\text{weak}\subset \overline{\{x_b:b\in (\Gamma.D)^\xi,\ b^-=a \}}^\text{weak}.$$  
 
We will also be interested in the more general construction.     We say a collection of subsets $(G_a)_{a\in \Gamma.D}$ of $B_X$ is \emph{normally weakly null} if for each $a=(\zeta_i, u_i)_{i=1}^n\in \Gamma.D$, $G_a\subset u_n$.

We recall that 
$$S_{\omega^\xi}=\{(\zeta_i)_{1\leqslant i\leqslant n}\,;\,0\leqslant\zeta_n<...<\zeta_1<\omega^\xi\}
$$
is a tree of rank $\omega^\xi.$

\begin{rem}\label{r2}Let $T$ be a tree with $\mathrm{rank}(T)=\omega^\xi$ and let $(x_t)_{t\in T}\subset B_X$ be a weakly null collection.  Then for every $r>0$ there exist a map $\phi : S_{\omega^\xi}.D\to T$ such that for all $a\in S_{\omega^\xi}.D,$ $\Vert y_a-x_{\phi(a)}\Vert<r,$ $\phi(a\vert1)<\phi(a\vert2)<...<\phi(a).$ Moreover, if $a=(\zeta_i,B_{Z_i})_{1\leqslant i\leqslant n}, $ then $\phi(a)\in T^{\zeta_n}.$
\end{rem}

\begin{proof}
We construct $\phi$ by induction on the length of the element $a\in S_{\omega^\xi}.D.$ Let $a=(\zeta,B_Z).$ Fix $t\in T^\zeta.$ It is obvious that for every $s\in T$ such that $s^-=t^-$ we have $s\in T^\zeta.$ Let $u$ be a weak neighborhood of $0$ in $X$ such that for every $x\in u\cap B_X$ there exists$z\in B_Z$ satisfying $\Vert x-z\Vert<r.$ The collection $(x_t)_{t\in T}$ a weakly null so there exists $s\in T^\zeta$ such that $s^-=t^-$ and $ x_s\in u.$  Let $y_{(\zeta,B_Z)}\in B_Z$ such that $\Vert x_s-y_{(\zeta,B_Z)}\in B_Z\Vert<r.$ We define $\phi(\zeta,B_Z)=s.$

Next, assume that $n>1$ and $\phi$ has been defined on the elements of $S_{\omega^\xi}.D$ of length $n-1.$ Let $a=((\zeta_i,B_{Z_i}))_{i=1}^n\in S_{\omega^\xi}.D.$ Assume also that $\phi(a^-)\in T^{\zeta_{n-1}}.$ We have $\zeta_n<\zeta_{n-1}$ so $T^{\zeta_{n-1}}\subset T^{\zeta_n}$ and there exists $t\in T^{\zeta_n}$ such that $t^-=\phi(a^-).$ In order to define $\phi(a)$ we proceed as in the definition of the image of an element of length $1.$
\end{proof}

For an ordinal $\xi$, a number $\sigma\geqslant 0$, an infinite dimensional Banach space $X$,  and $y\in X$, we define  the \emph{modulus of } $\xi$-\emph{asymptotic smoothness} by   $\varrho^\xi_X(\sigma, y)$ by \begin{align*} \varrho^\xi_X(\sigma, y)=\sup\Bigl\{\inf \{\|y+\sigma x\|-1: t\in T, x\in \text{co}(x_s)_{s\leqslant t}\}:& (x_t)_{t\in T}\subset B_X\text{\ weakly null}, \\ & \text{rank}(T)=\omega^\xi\Bigr\}.\end{align*}   That is, $\varrho^\xi_X(\sigma, y)\leqslant C$ if and only if for any $C_1>C$, any tree $T$ of rank $\omega^\xi$, and any weakly null collection $(x_t)_{t\in T}\subset B_X$, there exists a convex combination $x$ of some branch $(x_s)_{s\leqslant t}$ of the collection such that $\|y+\sigma x\|\leqslant 1+C_1$.    We isolate here the important special case $\xi=0$, in which case $\varrho_X^0(\sigma, y)=\rho_X(\sigma,y)$ is the familiar modulus of asymptotic smoothness given by \[\varrho^0_X(\sigma,y)=\underset{\dim(X/Z)<\infty}{\inf}\underset{x\in B_Z}{\sup} \|y+\sigma x\|-1.\]

It is easy to see that in the definitions of $\varrho^\xi_X(\sigma, y)$ and $\varrho^\xi_X(\sigma)$, it is sufficient to take the supremum over weakly null collections of the form $(x_t)_{t\in \Gamma.D}\subset B_X$, where $\Gamma$ is any fixed tree of rank $\omega^\xi$, and $x_{(\zeta_i, u_i)_{i=1}^n}\in u_n$ as above.    This is because  collection $(x_t)_{t\in T}\subset B_X$, if $\text{rank}(T)=\omega^\xi$, then for a tree $\Gamma$ with $\text{rank}(\Gamma)$, we can find a normally weakly null collection $(y_a)_{a\in \Gamma.D}\subset B_X$ whose branches are small perturbations of the branches of $(x_t)_{t\in T}$. See \cite[Proposition $2.1$]{beta} for more details. 

 It is clear that for a fixed $\sigma$, $\varrho^\xi_X(\sigma,\cdot)$ is $1$-Lipschitz on $X$, and for a fixed $y\in X$, $\varrho^\xi_X(\cdot, y)$ is $1$-Lipschitz as a function of $\sigma$.   We define $$\varrho^\xi_X(\sigma)=\underset{y\in B_X}{\sup} \varrho^\xi_X(\sigma, y).$$  We say $X$ is $\xi$-\emph{asymptotically uniformly smooth} (in short, $\xi$-\emph{AUS}) if $\inf_{\sigma>0}\varrho^\xi_X(\sigma)/\sigma=0$.    For $1<p\leqslant \infty$, we say $X$ is $\xi$-$p$-\emph{asympotically uniformly smooth} (in short, $\xi$-$p$-\emph{AUS}) if $\sup_{\sigma>0} \varrho^\xi_X(\sigma)/\sigma^p<\infty$.  We say $X$ is $\xi$-\emph{asymptotically uniformly flat} (in short, $\xi$-\emph{AUF}) if there exists some $\sigma_0>0$ such that $\varrho^\xi_X(\sigma_0)=0$.   We say that $X$ is $\xi$-\emph{asymptotically uniformly smoothable} (resp. $\xi$-$p$-\emph{asymptotically uniformly smoothable}, $\xi$-\emph{asymptotically uniformly flattenable}) if there exists an equivalent norm $|\cdot|$ on $X$ such that $(X, |\cdot|)$ is $\xi$-AUS (resp. $\xi$-$p$-AUS, $\xi$-AUF).  

\begin{rem}\upshape If $X$ is $\xi$-$p$-AUS and $1/p+1/q=1$, then by \cite[Proposition $3.2$]{beta} and standard Young duality, there exists a constant $c>0$ such that for any $\ee>0$, $s_\ee^{\omega^\xi}(B_{X^*})\subset (1-c \ee^q)B_{X^*}$. From here, an easy homogeneity argument yields that $Sz(X, \ee)<\omega^{\xi+1}$.  Therefore if a Banach space $X$ is $\xi$-$p$-AUS, then $Sz(X)\leqslant \omega^{\xi+1}$.   Since the Szlenk index is an isomorphic invariant, if $X$ is $\xi$-$p$-AUS-able, then $Sz(X)\leqslant \omega^{\xi+1}$. 
\label{accordion}
\end{rem}

Let $X$ be a Banach space. If $\mathfrak{G}$ is a rooted, hereditary tree on  on a subset $\mathfrak{g}$ of $2^{B_X}$, the power set of $B_X$, we define the \emph{weak derivative of} $\mathfrak{G}$, denoted by $(\mathfrak{G})_w'$, to be the set of all sequences $g\in \mathfrak{G}$ such that for any $u\in D$, there exists $G\in \mathfrak{g}$ such that $G\subset u$ and $g\smallfrown (G) \in \mathfrak{G}$. 

 We define the transfinite weak derivatives by $$(\mathfrak{G})^0_w=\mathfrak{G},$$ $$(\mathfrak{G})_w^{\xi+1} = ((\mathfrak{G})_w^\xi)_w',$$ and if $\xi$ is a limit ordinal, $$(\mathfrak{G})_w^\xi=\bigcap_{\zeta<\xi}(\mathfrak{G})_w^\zeta.$$  If there exists an ordinal $\xi$ such that $(\mathfrak{G})_w^\xi=\varnothing$, then we define $\mathfrak{w}(\mathfrak{G})$ to be the minimum such $\xi$. If no such $\xi$ exists,  we use the notation $\mathfrak{w}(\mathfrak{G})=\infty$.  For consistency, if $\mathfrak{G}=\varnothing$, $(\mathfrak{G})_w^\xi=\varnothing$ and $\mathfrak{w}(\varnothing)=0$. 

\begin{rem}\upshape Above, the weak derivative was defined using $D=CD(X)$.  However, we could have defined $(\mathfrak{G})_w'$ in the same way, except with $D$ being a fixed weak neighborhood basis at $0$ in $B_X$.    For all of our applications, an easy perturbation argument yields that these two distinct definitions will lead to the same results.

\end{rem}

We note that if $\mathfrak{G}$ is a rooted, hereditary tree and $g\in \mathfrak{g}^{<\omega}$, then $$\mathfrak{G}(g):= \{g_1\in \mathfrak{g}^{<\omega}: g\smallfrown g_1\in \mathfrak{G}\}$$ is empty if and only if $g\in \mathfrak{g}^{<\omega}\setminus \mathfrak{G}$, and otherwise $\mathfrak{G}(g)$ is a rooted, hereditary tree. Note that we use the notation $\mathfrak{G}(g)$, which denotes a rooted tree,  rather than the previously used $\mathfrak{G}_g$, which was not rooted.   An easy proof by induction yields that for any ordinal $\xi$, $$(\mathfrak{G}(g))_w^\xi= (\mathfrak{G})_w^\xi(g).$$

We establish some consequences of the definitions above.   In what follows, for a fixed (understood) Banach space $X$, let $\mathfrak{s}$ denote the set of singleton subsets of $B_X$, $\mathfrak{f}$ the set of finite, non-empty subsets of $B_X$, and $\mathfrak{c}$ the set of non-empty, norm compact subsets of $B_X$.

\begin{lemma} Let $X$ be a Banach space and let $D=CD(X)$.   Let $\mathfrak{g}$ be a set of subsets of $B_X$ and let $\mathfrak{G}$ be a tree on $\mathfrak{g}$.  For any ordinal $\xi$, the  following are equivalent. \begin{enumerate}[(i)]\item For every tree $T$ with $\text{\emph{rank}}(T)=\xi$, there exists a normally weakly null collection $(G_a)_{a\in T.D}\subset \mathfrak{g}$ such that for every $a\in T.D$, $(G_{a|i})_{i=1}^{|a|}\in \mathfrak{G}$. \item There exist a tree $T$ with $\text{\emph{rank}}(T)=\xi$ and   a normally weakly null collection $(G_a)_{a\in T.D}\subset \mathfrak{g}$ such that for every $a\in T.D$, $(G_{a|i})_{i=1}^{|a|}\in \mathfrak{G}$.\item $\mathfrak{w}(\mathfrak{G})> \xi$. \end{enumerate}

\label{prop5}
\end{lemma}

\begin{proof} 

$(i)\Rightarrow (ii)$ In light of the previous paragraph, $(i)\Rightarrow (ii)$ is trivial, since for any $\xi$, at least one tree of rank $\xi$ exists.

$(ii)\Rightarrow (iii)$ Assume $\text{rank}(T)=\xi$ and $(G_a)_{a\in T.D}\subset \mathfrak{G}$ is such that $(G_{a|i})_{i=1}^{|a|}\in \mathfrak{G}$ for all $a\in T.D$. We will prove by  induction   that for any $\zeta<\xi$, if $a\in (\{\varnothing\}\cup T.D)^\zeta$, $(G_{a|i})_{i=1}^{|a|}\in (\mathfrak{G})_w^\zeta$.    For the base case, the result holds for $a\in T.D$ by hypothesis, and for $a=\varnothing$ because $(G_{\varnothing|i})_{i=1}^0=\varnothing\in \mathfrak{G}$, since $\mathfrak{G}$ is a rooted tree.   The  limit ordinal case of the induction is clear.  Assume $\zeta+1<\xi$ and the $\zeta$ case of the induction holds. Fix  $a=t.v\in (\{\varnothing\}\cup T.D)^{\zeta+1}$ (where $t=v=\varnothing$ if $a=\varnothing$).   Then there exists $\lambda$ such that $t\smallfrown (\lambda)\in T^\zeta$, and for all $B_Z\in D$, $a\smallfrown (\lambda,B_Z)\in T^\zeta.D$.  Therefore for any finite codimensional subspace $Z$ of $X$, because the collection is normally weakly null, $G_{a\smallfrown (\lambda, B_Z)}\subset B_Z$, and by the inductive hypothesis, $(G_{a|i})_{i=1}^{|a|}\smallfrown (G_{a\smallfrown (\lambda, B_Z)})\in (\mathfrak{G})_w^\zeta$. This yields that $(G_{a|i})_{i=1}^{|a|}\in (\mathfrak{G})_w^{\zeta+1}$.   

Since $\text{rank}(T)=\xi$, $\varnothing\in (\{\varnothing\}\cup T.D)^\xi$, and the previous induction yields that $\varnothing\in (\mathfrak{G})_w^\xi$. Therefore $\mathfrak{w}(\mathfrak{G})>\xi$.

$(iii)\Rightarrow (i)$   By induction on $\xi$.   The $\xi=0$ case is vacuous.   Assume $\mathfrak{w}(\mathfrak{G})>\xi+1$, the result holds for $\xi$, and $T$ is a tree of rank $\xi+1$.   Since $\mathfrak{w}(\mathfrak{G})>\xi+1$, $\varnothing\in (\mathfrak{G})_w^{\xi+1}$.   This means that for any $B_Z\in D$, there exists $G^Z$ such that $G^Z\subset B_Z$ and $(G^Z)\in (\mathfrak{G})_w^\xi$.   For any length $1$ sequence $(\lambda)\in T$, let $G_{(\lambda, B_Z)}=G^Z$.  Since $(G^Z)\in (\mathfrak{G})_w^\xi$, $\varnothing\in (\mathfrak{G}(G^Z))_w^\xi$, and $\mathfrak{w}(\mathfrak{G}(G^Z))>\xi$.    By the inductive hypothesis,  there exists a  normally weakly null collection $(G^{\lambda,Z}_a)_{a\in T_{(\lambda)}.D}$ such that for each $a\in T_{(\lambda)}.D$, $(G^{\lambda,Z}_{a|i})_{i=1}^{|a|}\in \mathfrak{G}(G^Z)$.    Then for $a=(\lambda, B_Z)\smallfrown a_1\in T.D$, we define $G_a=G^{\lambda,Z}_{a_1}$.  This collection clearly satisfies the conclusions. 

Assume now that $\xi$ is a limit ordinal and the result holds for all smaller ordinals. Let $T$ be a tree with rank $\xi$. Let $R$ denote the set of all length $1$ sequences in $T$ and for each $t\in R$, let $T[t]=\{s\in T: t\leqslant s\}$. Then by standard properties of well-founded trees, $\text{rank}(T[t])<\xi$ for each $t\in R$, and $T=\cup_{t\in R}T[t]$ is a totally incomparable union.  From the latter fact, $T.D=\cup_{t\in R}T[t].D$ is a disjoint union.  By the inductive hypothesis, for each $t\in R$, there exists a collection $(G_a^t)_{a\in T[t].D}$ such that for all $a\in T[t].D$, $(G_{a|i}^t)_{i=1}^{|a|}\in \mathfrak{G}$. Then define $(G_a)_{a\in T.D}$ by letting $G_a=G^t_a$ whenever $a\in T[t].D$.  This collection clearly satisfies the conclusions.

\end{proof}

For $y\in X$,  $\sigma\geqslant 0$, and $\rho\in \rr$, we let $$\mathfrak{S}(\sigma, y, \rho)=\{\varnothing\}\cup \{(x_i)_{i=1}^n:(\forall x\in\text{co}(x_i)_{i=1}^n)(\|y+\sigma x\|\geqslant \rho+1)\}.$$  Note that $\mathfrak{S}(\sigma,y,\rho)$ is a rooted, hereditary tree.  Note also that $\|y+\sigma x\|\geqslant \rho+1$ for all $x\in \text{co}(x_i)_{i=1}^n$ if and only if $\|x\|\geqslant \rho+1$ for all $x\in \text{co}(y+\sigma x_i)_{i=1}^n$. By the geometric Hahn-Banach theorem, $(x_i)_{i=1}^n\in \mathfrak{S}(\sigma, y, \rho)$ if and only if there exists $x^*\in B_{X^*}$ such that for all $1\leqslant i\leqslant n$, $\text{Re\ }x^*(y+\sigma x_i)\geqslant \rho+1$.

For a set $\mathfrak{g}$ of subsets of $B_X$, we let $\mathfrak{S}_\mathfrak{g}(\sigma, y, \rho)$ denote the set consisting of the empty sequence together with the set of all sequences $(G_i)_{i=1}^n \in \mathfrak{g}^{\omega}$ such that  $\mathfrak{S}(\sigma, y, \rho)\cap \prod_{i=1}^n G_i\neq \varnothing$. That is, $(G_i)_{i=1}^n \in \mathfrak{S}_\mathfrak{g}(\sigma, y, \rho)$ if and only if $G_i\in \mathfrak{g}$ for all $1\leqslant i\leqslant n$ and there exist $(x_i)_{i=1}^n\in \prod_{i=1}^n G_i$ and $x^*\in B_{X^*}$ such that for all $1\leqslant i\leqslant n$, $\text{Re\ }x^*(y+\sigma x_i)\geqslant \rho+1$.    

Recall that $\mathfrak{s}$ denotes the set of singleton subsets of $B_X$, from which it follows that $$\mathfrak{S}_\mathfrak{s}(\sigma, y, \rho)=\{\varnothing\}\cup \{(\{x_i\})_{i=1}^n: (x_i)_{i=1}^n\in \mathfrak{S}(\sigma, y, \rho)\}.$$

\begin{proposition} Let $\xi$ be an ordinal and let $X$ be a Banach space.  For $\sigma\geqslant 0$ and $y\in X$, $$\varrho^\xi_X(\sigma, y)=\sup\{\rho: \mathfrak{w}(\mathfrak{S}_\mathfrak{s}(\sigma,y,\rho))>\omega^\xi\}.$$   

\label{mod}
\end{proposition}

\begin{proof} For the $\sigma=0$ case, $\varrho^\xi_X(\sigma,y)=\|y\|-1=\sup \{\rho: \mathfrak{w}(\mathfrak{G}_\mathfrak{s}(0,y,\rho))>\omega^\xi\}$.  For the remainder of the proof, we consider the case $\sigma>0$.

Fix $C<C_1<\varrho^\xi_X(\sigma, y)$.  Then there exists a tree $T$ with $\text{rank}(T)=\omega^\xi$ and a weakly null collection $(x_t)_{t\in T}$ of $B_X$ such that for every $t\in T$ and $x\in \text{co}(x_s)_{s\leqslant t}$, $\|y+\sigma x\|\geqslant 1+C_1$. Fix $0<\ee<C_1-C$.     . Recall also that $D=CD(X)$. We recall that by Remark~\ref{r2} there exists    $\phi:S_{\omega^\xi}.D\to T$ and $(y_a)_{a\in S_{\omega^\xi}.D}$ a normally weakly null family $(y_a)_{a\in S_{\omega^\xi}.D}$ of $B_X$   such that for all $a\in S_{\omega^\xi}.D$,    $\|y_a-x_{\phi(a)}\|<\ee/\sigma$   and  $\phi(a|1)<\ldots <\phi(a)$.   From the last two properties, it will follow that for any $a\in S_{\omega^\xi}.D$ and $x=\sum_{b\leqslant a}w_b y_b\in \text{co}(y_b:b\leqslant a)$, since $x'=\sum_{b\leqslant a}w_b x_{\phi(b)}\in \text{co}(x_b: b\leqslant \phi(a))$,  $$\|y+\sigma x\|\geqslant \|y+\sigma x'\|-\sigma\sum_{b\leqslant a}w_b \ee/\sigma \geqslant 1+C_1-\ee\geqslant 1+C.$$   From this it follows that $C\leqslant \sup \{\rho: \mathfrak{w}(\mathfrak{S}_\mathfrak{s}(\sigma,y,\rho))>\omega^\xi\}$.   For $a=(\zeta)\in S_{\omega^\xi}$ and $B_Z\in D$, pick a weak neighborhood $u$ of $0$ in $X$ such that for any $x\in u\cap B_X$, there exists $y\in B_Z$ such that $\|x-y\|<\ee/\sigma$.  Choose $t\in T^\zeta\neq \varnothing$ and let $s=t^-\in \{\varnothing\}\cup T$.    Since $(x_r)_{r\in T}$ is weakly null, there exists $r\in T^\zeta$ such that $r^-=s==t^-$  and $x_r\in u\cap B_X$. Let $y_{(\zeta, B_Z)}\in B_Z$ be such that $\|y_{(\zeta, B_Z)}-x_r\|<\ee/\sigma$ and define $\phi(\zeta, B_Z)=r$. Note that $\phi(\zeta, B_Z)\in T^\zeta$. 

Next, assume that for  some $n>1$ and  $a=(\zeta_i, B_{Z_i})_{i=1}^n\in S_{\omega^\xi}.D$, $\phi(a^-)$ has been chosen. Assume also that $\phi(a^-)\in T^{\zeta_{n-1}}$.   Fix a weak neighborhood $u$ of $0$ in $X$ such that for any $x\in u\cap B_X$, there exists $y\in B_{Z_n}$ such that $\|x-y\|<\ee/\sigma$.   Since $\phi(a^-)\in T^{\zeta_{n-1}}$ and since $\zeta_{n-1}<\zeta_n$, there must exist some $t\in T^{\zeta_n}$ such that $t^-=\phi(a^-)$.  Since $(x_r)_{r\in T}$ is weakly null, there exists $r\in T^{\zeta_n}$ such that $x_r\in u\cap B_X$. Choose $y_a\in B_{Z_n}$ such that $\|y_a-x_r\|<\ee/\sigma$ and let $\phi(a)=r$.  This completes the recursive construction. Clearly the conclusions are satisfied by this construction.

Next, fix $C<\sup \{\rho: \mathfrak{w}(\mathfrak{S}_\mathfrak{s}(\sigma,y,\rho))>\omega^\xi\}$.    By Lemma \ref{prop5}, there exist a tree $T$ with $\text{rank}(T)=\omega^\xi$ and  a normally weakly null collection $(x_a)_{a\in T.D}$ of $B_X$ such that for all $a\in T.D$ and $x\in \text{co}(x_b:b\leqslant a)$, $\|y+\sigma x\| \geqslant 1+C$. Then $\varrho^\xi_X(\sigma, y)\geqslant C$.  This yields that $$\sup \{\rho: \mathfrak{w}(\mathfrak{S}_\mathfrak{s}(\sigma,y,\rho))>\omega^\xi\} \leqslant \varrho^\xi_X(\sigma,y).$$

\end{proof}

We next prove that one can move from singletons to norm compact sets in the preceding proposition.

\begin{lemma} Let $X$ be a Banach space. For $g\in \mathfrak{f}^{<\omega}$, we let $\prod g=\prod_{i=1}^n G_i$ if $g=(G_i)_{i=1}^n$ and $\prod g=\{\varnothing\}$ if $g=\varnothing$. \begin{enumerate}[(i)]
\item Let $S_1, \ldots, S_n$ be hereditary, non-empty,  rooted trees on $B_X$. For each $1\leqslant k\leqslant n$, let $$F_k=\{\varnothing\}\cup \Bigl\{g\in \mathfrak{f}^{<\omega}: S_k\cap \prod g\neq \varnothing\Bigr\}.$$  Then for any ordinal $\zeta$, $$\Bigl(\bigcup_{k=1}^n F_k\Bigr)_w^\zeta \subset \bigcup_{k=1}^n \Bigl\{g\in \mathfrak{f}^{<\omega}: (S_k)_w^\zeta \cap \prod g\neq \varnothing\Bigr\}\Bigr).$$ 
 \item If $\xi$ is an ordinal, $\varnothing\neq H\subset [0,\infty)\times X$ is finite, and $\ee>0$, then $$\mathfrak{w}\Biggl(\bigcup_{(\sigma,y)\in H} \mathfrak{S}_\mathfrak{f}(\sigma,y,\ee+\varrho^\zeta_X(\sigma,y))\Biggr)\leqslant \omega^\xi.$$  
 \item If $\xi$ is an ordinal,  $\varnothing\neq H\subset [0,\infty)\times X$ is compact, and $\ee>0$, then $$\mathfrak{w}\Biggl(\bigcup_{(\sigma,y)\in H} \mathfrak{S}_\mathfrak{c}(\sigma,y,\ee+\varrho^\zeta_X(\sigma,y))\Biggr)\leqslant \omega^\xi.$$    \end{enumerate}
\label{mod2}
\end{lemma}

\begin{proof} 

$(i)$ We work by induction on $\zeta$.    The $\zeta=0$ case holds by the definition of  the weak derivative.

Assume the result holds for $\zeta$.  Fix $g\in \Bigl(\bigcup_{k=1}^n F_k\Bigr)_w^{\zeta+1}$.   Seeking a contradiction, suppose that for each $1\leqslant k\leqslant n$ and $t\in \prod g$, there exists $B_{Z_{k,t}}\in D$ such that for each $x\in B_{Z_{k,t}}$, $t\smallfrown (x)\notin (S_k)_w^\zeta$.    Then let $Z=\bigcap_{k=1}^n\bigcap_{t\in \prod g}  Z_{k,t}\in D$.   Since $g\in \Bigl(\bigcup_{k=1}^n F_k\Bigr)_w^{\zeta+1}$, there exists $G\subset B_Z$ such that $g\smallfrown (G)\in \Bigl(\bigcup_{k=1}^n F_k\Bigr)_w^\zeta$.   Then there exist $1\leqslant k\leqslant n$ and some $(x_i)_{i=1}^n\smallfrown (x)\in (S_k)_w^\zeta\cap\prod(g\smallfrown (G))$ contradicting our choice of $B_Z\subset B_{Z_{k,t}}$. The existence of such a $k$ and $(x_i)_{i=1}^n\smallfrown (x)$ follows from the definition of the sets $F_i$ in the $\zeta=0$ case, and from the inductive hypothesis in the $\zeta>0$ case.  Therefore there must exist some $1\leqslant k\leqslant n$ and $t\in \prod g$ such that for any $B_Z\in D$, there exists $x\in B_Z$ such that $t\smallfrown (x)\in (S_k)_w^\zeta$.   This yields that $$g\in \Bigl\{h\in \mathfrak{f}^{<\omega}: (S_k)_w^\xi \cap \prod h\neq \varnothing\Bigr\},$$ finishing the successor case.

Assume $\zeta$ is a limit ordinal and the result holds for all ordinals less than $\zeta$.  Fix $g\in \Bigl(\bigcup_{k=1}^n F_k\Bigr)_w^\zeta$.  Seeking a contradiction, suppose that for each $t\in \prod g$ and $1\leqslant k\leqslant n$, there exists $\zeta_{k,t}<\zeta$ such that $t\notin (S_k)_w^{\zeta_{k,t}}$.  Let $$\zeta_0=\max\{\zeta_{k,t}: 1\leqslant k\leqslant n, t\in \prod g\}+1<\zeta.$$ Since $g\in \Bigl(\bigcup_{k=1}^n F_k\Bigr)_w^{\zeta_0}$, the inductive hypothesis yields the existence of $t\in \prod g$ and $1\leqslant k\leqslant n$ such that $t\in (S_k)_w^{\zeta_0}$, contradicting our choice of $\zeta_0$.    Therefore there exist some $t\in \prod g$ and $1\leqslant k\leqslant n$ such that for all $\zeta_0<\zeta$, $t\in (S_k)_w^{\zeta_0}$. From this it follows that $t\in (S_k)_w^\zeta$, and $$g\in \Bigl\{h\in \mathfrak{f}^{<\omega}: (S_k)_w^\xi \cap \prod h\neq \varnothing\Bigr\},$$ finishing the limit ordinal case.

$(ii)$ We prove  the $\xi=0$ case using the characterization of $\varrho^0_X(\sigma,y)$ given by
$$\varrho^0_X(\sigma,y)=\inf\{\sup\{\|y+\sigma x\|-1 : x\in B_Z\}: \dim(X/Z)<\infty\}
$$ 
 For each $(\sigma, y)\in H$, there exists $B_{Z_{\sigma, y}}\in D$ such that for all $x\in B_{Z_{\sigma, y}}$, $\|y+\sigma x\|<1+\ee+\varrho^0_X(\sigma, y)$.    Let $Z=\bigcap_{(\sigma, y)\in H}Z$. Then for any $(\sigma, y)\in H$, any  $G\subset B_Z$, and any $x\in G$, $\|y+\sigma x\|<1+\ee+\varrho^\xi_X(\sigma, y)$. This shows that $$\varnothing\notin \bigcup_{(\sigma, y)\in H}\mathfrak{S}_\mathfrak{f}(\sigma, y, 1+\ee+\varrho^\xi_X(\sigma, y)),$$ from which it follows that   $$\mathfrak{w} \Bigl(\bigcup_{(\sigma, y)\in H}\mathfrak{S}_\mathfrak{f}(\sigma, y, \ee+\varrho^\xi_X(\sigma, y))\Bigr)\leqslant 1=\omega^0.$$

Next we complete the $\xi>0$ case. Note that in this case, $\omega^\xi$ is a limit ordinal. Note also that since a rooted tree is either empty or contains $\varnothing$, it is not possible for $\mathfrak{w}(\mathfrak{G})=\omega^\xi$ whenever $\mathfrak{G}$ is a rooted, hereditary tree.    Therefore by Proposition \ref{mod}, for any $\sigma\geqslant 0$ and $y\in X$, $\mathfrak{w}(\mathfrak{S}(\sigma, y, \ee+\varrho^\xi_X(\sigma,y)))<\omega^\xi$.  Let $$\zeta= \max\{\mathfrak{w}(\mathfrak{S}(\sigma, y, \ee+\varrho^\xi_X(\sigma,y))):(\sigma, y)\in H\}<\omega^\xi.$$  Then by $(i)$, \begin{align*}\Bigl(\bigcup_{(\sigma, y)\in H} \mathfrak{S}_\mathfrak{f}(\sigma, y, \ee+\varrho^\xi_X(\sigma, y))\Bigr)^\zeta_w & \subset \bigcup_{(\sigma, y)\in H}\Bigl\{g\in \mathfrak{f}^{<\omega}:(\mathfrak{S}(\sigma, y, \ee+\varrho^\xi_X(\sigma,y)))_w^\zeta\cap \prod g\neq \varnothing\Bigr\} \\ & = \bigcup_{(\sigma,y)\in H} \Bigl\{g\in\mathfrak{f}^{<\omega}: \varnothing\cap \prod g \neq \varnothing \Bigr\} \\ & = \varnothing.\end{align*}  This completes the $\xi>0$ case. 

$(iii)$ First we prove the $\xi=0$ case. Fix $\ee>0$ and a compact subset $H\subset [0,\infty)\times X$.  Let $H_0\subset H$ be a finite set such that for any $(\sigma,y)\in H$, there exists $(\sigma_0, y_0)\in H_0$ such that $|\sigma-\sigma_0|+\|y-y_0\|<\ee/3$.  Note that for any $\sigma, \sigma_0\geqslant 0$ and $y, y_0\in X$, \[|\varrho^0_X(\sigma,y)-\varrho^0_X(\sigma_0, y_0)| \leqslant |\sigma-\sigma_0|+\|y-y_0\|.\] For each $(\sigma,y)\in H_0$, fix a subspace $Z_{\sigma,y}$  of $X$  such that $\dim (X/Z_{\sigma,y})<\infty$ and \[\sup \{\|y+\sigma x\| : x\in B_{Z_{\sigma,y}}\} <1+\ee/3+\varrho^0_X(\sigma,y).\]  Let $Z=\bigcap_{(\sigma, y)\in H_0} Z_{\sigma, y}$ and note that for each $(\sigma, y)\in H$ and $x\in B_Z$, if $(\sigma_0, y_0)\in H_0$ is such that $|\sigma-\sigma_0|+\|y-y_0\|<\ee/3$, \begin{align*} \|y+\sigma x\|  & \leqslant |\sigma-\sigma_0| + \|y-y_0\|+\|y_0+\sigma x\|  \\ & \leqslant 1+2\ee/3+\varrho^0_X(\sigma_0, y_0) \leqslant 1+2\ee/3 +\varrho^0_X(\sigma, y) + |\sigma-\sigma_0|+\|y-y_0\| \\ &  < 1+\ee+ \varrho^0_X(\sigma,y). \end{align*} From this it follows that $\varnothing\notin \bigcup_{(\sigma,y)\in H} \mathfrak{S}_\mathfrak{c}(\sigma,y,\ee+\varrho^0_X(\sigma,y))$. 

 Fix $R>0$ such that for any $(\sigma, y)\in H$, $|\sigma|\leqslant R$. Fix $\ee>0$ and let $\varphi:\mathfrak{c}\to \mathfrak{f}$ be such that for each $G\in \mathfrak{c}$, $\varphi(G)$ is a finite subset of $G$ such that for each $ y\in G$, there exists $ y_0\in \varphi(G)$  such that $\|y-y_0\|<\ee/4$.  Choose a finite subset $H_0$ of $H$ such that for any $(\sigma, y)\in H$, there exists $(\sigma_0, y_0)\in H_0$ such that $|\sigma-\sigma_0|+\|y-y_0\|<\ee/4R$.   To obtain a contradiction, assume that $\mathfrak{w}\Bigl(\bigcup_{(\sigma, y)\in H}\mathfrak{G}_\mathfrak{c}(\sigma, y, \ee+\varrho^\xi_X(\sigma, y))\Bigr)\geqslant\omega^\xi$.   Since $\omega^\xi$ is a limit ordinal, the inequality here must be strict.  By Lemma \ref{prop5}, there exist a tree $T$ with $\text{rank}(T)=\omega^\xi$ and a normally weakly null collection $(G_a)_{a\in T.D}\subset \mathfrak{c}$ such that for each $a\in T.D$, $(G_b)_{b\leqslant a}\in \bigcup_{(\sigma,y)\in H} \mathfrak{S}_\mathfrak{c}(\sigma, y, \ee+\varrho^\xi_X(\sigma,y))$.  Since $\varphi(G_a)\subset G_a$ for each $a\in T.D$, the collection $(\varphi(G_a))_{a\in T.D}$ is normally weakly null.  We claim that for each $a\in T.D$, $(\varphi(G_b))_{b\leqslant a}\in \bigcup_{(\sigma, y)\in H_0} \mathfrak{S}_\mathfrak{f}(\sigma, y, \ee/4+\varrho^\xi_X(\sigma, y))$, which, combined with Lemma \ref{prop5}, will yield that \[\mathfrak{w}\Bigl(\bigcup_{(\sigma,y)\in H_0} \mathfrak{S}_\mathfrak{f}(\sigma, y, \ee/4+\varrho^\xi_X(\sigma,y))\Bigr)>\omega^\xi.\]  This  inequality will contradict $(ii)$ and finish the proof.   Fix $a\in T.D$ and note that, since $(G_b)_{b\leqslant a}\in \bigcup_{(\sigma, y)\in H}\mathfrak{S}_\mathfrak{c}(\sigma, y, \ee+\varrho^\xi_X(\sigma, y))$, there exist $(\sigma,y)\in H$ and $(x_b)_{b\leqslant a}\in \prod_{b\leqslant a}G_b$ such that for any $x\in \text{co}(x_b: b\leqslant a)$, \[\|y+\sigma x\| \geqslant 1+\ee+\varrho^\xi_X(\sigma, y).\] Fix $(\sigma_0, y_0)\in H_0$ such that $|\sigma-\sigma_0|+\|y-y_0\|<\ee/4R$ and $(x_b^0)_{b\leqslant a}\in \prod_{b\leqslant a}\varphi(G_b)$ such that for each $b\leqslant a$, $\|x_b-x^0_b\|<\ee/4$.  For any $x_0=\sum_{b\leqslant a}w_bx_b^0\in \text{co}(x_b^0: b\leqslant a)$, since $x:=\sum_{b\leqslant a}w_bx_b\in \text{co}(x_b:b\leqslant a)$, \begin{align*} \|y_0+\sigma_0 x_0\| & \geqslant \|y+\sigma x_0\| -|\sigma-\sigma_0|-\|y-y_0\| \geqslant \|y+\sigma x\| - \sigma \sum_{b\leqslant a}w_b\|x_b-x^0_b\| - \ee/4 \\ & \geqslant  1+\ee+\varrho^\xi_X(\sigma,y) - R(\ee/4R) - \ee/4 \\ & \geqslant 1+\ee+\varrho^\xi_X(\sigma_0, y_0)-|\sigma-\sigma_0|-\|y-y_0\| - \ee/3 \\ & \geqslant 1+\ee/4 + \varrho^\xi_X(\sigma_0, y_0).\end{align*} Therefore \begin{align*} (\varphi(G_b))_{b\leqslant a} & \in \mathfrak{S}_\mathfrak{f}(\sigma_0, y_0, 1+\ee/4+\varrho^\xi_X(\sigma_0, y_0)) \\ & \subset \bigcup_{(\sigma',y')\in H_0} \mathfrak{S}_\mathfrak{f}(\sigma',y', 1+\ee/4+\varrho^\xi_X(\sigma',y')),\end{align*} as claimed.

\end{proof}

We are now ready to prove the following simultaneity result. Recall that $CD(X)$ denotes the set of finite codimensional subspaces of $X$. 

\begin{corollary} Let $\xi$ be an ordinal, $X$ a Banach space, let $D=CD(X)$, and let $T$ be any tree with rank $\omega^\xi$. Let $G_\varnothing \subset X$ be norm compact and let $(G_a)_{a\in T.D}\subset \mathfrak{c}$ be a normally weakly null collection.   Then for any $\ee>0$ and $\sigma_1>0$, there exists $a\in \text{\rm MAX}(T.D)$ such that for any $y\in G_\varnothing$, any  $(x_b)_{b\leqslant a}\in \prod_{b\leqslant a} G_b$, and any $\sigma\in \sigma_1 B_\mathbb{K}$, $$\min \{\|y+\sigma x\|: x\in \text{\emph{co}}(x_{a|i}:1\leqslant i\leqslant |a|)\} <1+\ee+\varrho^\xi_X(\sigma, y).$$

\label{cor1}
\end{corollary}

\begin{proof}  By replacing $G_\varnothing$ with its balanced hull, we can assume that $G_\varnothing$ is balanced. We can also replace each set $G_a$, $a\in T.D$, with its balanced hull, noting that the resulting sets still form a normally weakly null collection.  

  Seeking a contradiction, assume $T$ is a tree with rank $\omega^\xi$, assume $(G_a)_{a\in T.D}\subset \mathfrak{c}$ is normally weakly null and $\ee>0$ are such that for every $a\in MAX(T.D)$, there exist $y\in G_\varnothing$, $\sigma\in [0, \sigma_1]$, and $(x_b)_{b\leqslant a}\in \prod_{b\leqslant a}G_b$	such that for every $x\in \text{co}(x_b: b\leqslant a)$, $\|y+\sigma x\| \geqslant 1+\ee+\varrho^\xi_X(\sigma, y)$.  Then \begin{align*}(G_b)_{b\leqslant a} & \in \bigcup_{(\sigma,y)\in [0,\sigma_1]\times G_\varnothing} \mathfrak{S}_\mathfrak{c}(\sigma,y,1+\ee+\varrho^\xi_X(\sigma,y)). \end{align*}  By Lemma \ref{prop5}, \[\mathfrak{w}\Bigl(\bigcup_{(\sigma,y)\in [0,\sigma_1]\times G_\varnothing} \mathfrak{S}_\mathfrak{c}(\sigma,y,1+\ee+\varrho^\xi_X(\sigma,y))\Bigr)>\omega^\xi,\] which contradicts Proposition \ref{mod}$(iii)$ with $H=[0,\sigma_1]\times G_\varnothing$. From this it follows that there exists $a\in T.D$ such that for every $y\in G_\varnothing$, every $(x_b)_{b\leqslant a}\in \prod_{b\leqslant a}G_b$, and every $\sigma\in [0,\sigma_1]$,  there exists $x\in \text{co}(x_b:b\leqslant a)$ such that $\|y+\sigma x\|< 1+\ee +\varrho^\xi_X(\sigma, y)$. Since each of the sets $G_b$, $\varnothing\leqslant b\leqslant a$, is balanced, the same inequality holds for any $\sigma\in \sigma_1 B_\mathbb{K}$  with $\varrho^\xi_X(\sigma, x_0)$ replaced by $\varrho^\xi_X(|\sigma|, x_0)$.

\end{proof}

\section{Some special trees}

Note that in the context of Corollary \ref{cor1}, given a compact $G_\varnothing\subset X$ and a normally weakly null collection $(G_a)_{a\in T.D}$, for any $\sigma_1>0$,  we obtained a single branch $(G_b)_{b\leqslant a}$ such that for any $y\in G_\varnothing$, any $(x_b)_{b\leqslant a}\in \prod_{b\leqslant a}G_b$, and any $\sigma$ with $|\sigma|\leqslant \sigma_1$, there is a convex combination $x$ of $(x_b)_{b\leqslant a}$ such that $\|y+\sigma x\|<1+\ee+\varrho^\xi_X(\sigma, y)$. However, based on this result, the convex coefficients of the  convex combination $x$ could be different for different choices of $y$, $\sigma$, and $(x_b)_{b\leqslant a}$. As we will see later, for  our purposes, we will want the convex coefficients not to depend on these choices. Achieving this result is the content of this section.

If $t=(\zeta_i)_{i=1}^n$ is a sequence of ordinals and if $\zeta$ is an ordinal, we define $\zeta+t=(\zeta+\zeta_i)_{i=1}^n$. We make the same convention for an infinite sequence of ordinals.    If $T$ is a tree on some set $[0, \gamma]$ of ordinals, we define $\zeta+T=\{\zeta+t: t\in T\}$.  We now define for each ordinal $\xi$ and $n\in\nn$ a particular tree $\Gamma_{\xi,n}$ which will play an important role in our later results.  We will also define for each such $\xi$ and $n$ a function $\mathbb{P}_{\xi,n}:\Gamma_{\xi,n}\to [0,1]$.  We will also define some special subsets of $\Gamma_{\xi,n}$, called the \emph{levels} of $\Gamma_{\xi,n}$, denoted by $\Lambda_{\xi,n,1}, \ldots, \Lambda_{\xi,n,n}$.    We will also define some ill-founded analogues of these trees, $\Gamma_{\xi,\infty}$, with levels $\Lambda_{\xi,\infty,1}, \Lambda_{\xi,\infty, 2}, \ldots$, and corresponding functions $\mathbb{P}_{\xi,\infty}:\Gamma_{\xi,\infty}\to [0,1]$.   

We define $$\Gamma_{0,1}=\{(0)\},$$ the tree consisting of a single node, $(0)$. We define $\mathbb{P}_{0,1}((0))=1$.

Next, assume that for some $n\in\nn$ and each $1\leqslant k\leqslant n$, $\Gamma_{\xi,k}$ and $\Lambda_{\xi, k, 1}, \ldots, \Lambda_{\xi,k,k}$ have been defined. Suppose also that $\Gamma_{\xi,k}$ is a tree on $[0, \omega^\xi k)$.   Let $$\Lambda_{\xi, n+1, 1}=\{\omega^\xi n+t: t\in \Gamma_{\xi,1}\}$$ and for $1\leqslant i\leqslant n$, let $$\Lambda_{\xi, n+1, i+1}=\{s\smallfrown t: s\in MAX(\Lambda_{\xi,n+1,1}), t\in \Lambda_{\xi,n,i}\}.$$  Note that $t\leftrightarrow \omega^\xi n+t$ is a bijection of $\Gamma_{\xi,1}$ with $\Lambda_{\xi,n+1,1}$, and $\Lambda_{\xi,n+1,1}$ is a tree on $[\omega^\xi n, \omega^\xi(n+1))$.     We define $\mathbb{P}_{\xi,n+1}$ on $\Lambda_{\xi,n+1,1}$ by letting $\mathbb{P}_{\xi,n+1}(\omega^\xi n+t)=\mathbb{P}_{\xi,1}(t)$.   Note also that for each $s\in MAX(\Lambda_{\xi,n+1, 1})$, $t\leftrightarrow s\smallfrown t$ is a bijection of $\Lambda_{\xi,n+1,i}$ with $\{t\in \Lambda_{\xi, n+1, i+1}: s<t\}$.    We define $\mathbb{P}_{\xi,n+1}$ on $\Lambda_{\xi,n+1, i+1}$ by letting $$\mathbb{P}_{\xi,n+1}(s\smallfrown t)=\mathbb{P}_{\xi,n}(t),$$ where $s\smallfrown t$ is the unique representation of a member of $\Lambda_{\xi, n+1, i+1}$ as a concatenation of a member $s$ of $MAX(\Lambda_{\xi,n+1, 1})$ and a member $t$ of $\Lambda_{\xi, n,i}$.    

Next assume that for some $\xi$ and each $n\in\nn$, $\Gamma_{\xi, n}$, $\mathbb{P}_{\xi,n}$ have been defined.  Let $\Lambda_{\xi+1, 1,1}=\Gamma_{\xi+1, 1}=\cup_{n=1}^\infty \Gamma_{\xi, n}$ and define $\mathbb{P}_{\xi+1, 1}|_{\Gamma_{\xi,n}}=\frac{1}{n} \mathbb{P}_{\xi,n}$. Note that the union $\Gamma_{\xi+1, 1}=\cup_{n=1}^\infty \Gamma_{\xi,n}$ is a totally incomparable union, since any sequence $(\zeta_i)_{i=1}^k\in\Gamma_{\xi,n}$ satisfies $\omega^\xi(n-1)\leqslant \zeta_1<\omega^\xi n$.

For a limit ordinal $\xi$, if $\Gamma_{\zeta+1, 1}$ has been defined for each $\zeta<\xi$, we let $$\Lambda_{\xi,1,1}=\Gamma_{\xi,1}=\bigcap_{\zeta<\xi} (\omega^\zeta +\Gamma_{\zeta+1, 1}).$$  This is a totally incomparable union, since for each $\zeta<\xi$ and $(\zeta_i)_{i=1}^k\in \omega^\zeta+\Gamma_{\zeta+1, 1}$, $\omega^\zeta \leqslant \zeta_1< \omega^{\zeta+1}$.   Moreover, we define $\mathbb{P}_{\xi,1}|_{\omega^\zeta+\Gamma_{\zeta+1,1}}$ by $$\mathbb{P}_{\xi,1}(\omega^\zeta+t)=\mathbb{P}_{\zeta+1,1}(t).$$

We also define $\Lambda_{\xi,\infty,1}, \Lambda_{\xi,\infty, 2}, \ldots$ by letting $$\Lambda_{\xi,\infty,1}=\Gamma_{\xi,1}$$ and, if $\Lambda_{\xi,\infty, 1}, \ldots, \Lambda_{\xi,\infty, i}$ have been defined, letting  \begin{align*} \Lambda_{\xi,\infty, i+1} & = \{s\smallfrown (\omega^\xi i+t): s\in MAX(\Lambda_{\xi,\infty, i}), t\in \Gamma_{\xi,1}\} \end{align*}. We note that $\Lambda_{\xi,\infty,i+1}$ admits the alternative description \begin{align*} \Lambda_{\xi,\infty,i+1} = \{s_1\smallfrown (\omega^\xi+s_2)\smallfrown \ldots \smallfrown (\omega^\xi(i-1)+s_i)\smallfrown (\omega^\xi i+s):& s\in \Gamma_{\xi,1}, \\ &   s_1, \ldots, s_i\in MAX(\Gamma_{\xi,1})\}.\end{align*} Note that such a representation $s_1\smallfrown(\omega^\xi+s_2)\smallfrown \ldots \smallfrown (\omega^\xi i +s)$, $s_1, \ldots, s_i\in MAX(\Gamma_{\xi,1})$, $s\in \Gamma_{\xi,1}$, is unique. We define $$\mathbb{P}_{\xi,\infty}(s_1\smallfrown(\omega^\xi+s_2)\smallfrown \ldots \smallfrown (\omega^\xi +s))=\mathbb{P}_{\xi,1}(s).$$   We let $\Gamma_{\xi,\infty}=\cup_{i=1}^\infty\Lambda_{\xi,\infty, i}$ and note that this is a disjoint union.    We note also that $\Gamma_{\xi,\infty}$ is a pruned tree, and a member $\tau$ of $[\Gamma_{\xi,\infty}]$ is uniquely representable as $$\tau=s_1\smallfrown (\omega^\xi+s_2)\smallfrown (\omega^\xi 2+s_3)\smallfrown \ldots,$$ where $s_1, s_2, \ldots\in MAX(\Gamma_{\xi,1})$.   Moreover, $\Lambda_{\xi,\infty,1}=\Gamma_{\xi,1}$ and for any $i\in\nn$ and any $s\in MAX(\Lambda_{\xi,\infty, i})$, $t\leftrightarrow s\smallfrown (\omega^\xi i+t)$ is a bijection of $\Gamma_{\xi,\infty}$ with $\{t\in \Gamma_{\xi,\infty}: s<t\}$ which identifies $\Lambda_{\xi,\infty,j}$ with $\{t\in \Lambda_{\xi,\infty, i+j}: s<t\}$ for each $j\in\nn$, and  which satisfies $\mathbb{P}_{\xi,\infty}(s\smallfrown (\omega^\xi i+t))=\mathbb{P}_{\xi,1}(t)$ for all $t\in \Gamma_{\xi,1}$.  Therefore for such an $s$, $\{t\in \Lambda_{\xi,\infty, i+1}: s<t\}$ is naturally identifiable with $\Gamma_{\xi,1}$ in a way which equates values of $\mathbb{P}_{\xi,\infty}$ and $\mathbb{P}_{\xi,1}$.  Similarly, for any $i\in\nn$, any $s\in MAX(\Lambda_{\xi,\infty,i})$, the map $\tau\leftrightarrow s\smallfrown(\omega^\xi i+\tau)$ is a bijection of $[\Gamma_{\xi,\infty}]$ with $\{\tau\in [\Gamma_{\xi,\infty}]: s<\tau\}$.      We will use these natural identifications often in the sequel.   

For $t\in \Lambda_{\xi,\infty,i}$, we define $\lambda_0(t), \ldots, \lambda_{i-1}(t)$ by letting $\lambda_0(t)=\varnothing$ and $\lambda_j(t)$ be the initial segment of $t$ such that $\lambda_j(t)\in MAX(\Lambda_{\xi,\infty, j})$.    For $\tau\in [\Gamma_{\xi,\infty}]$, we define $\lambda_0(\tau), \lambda_1(\tau), \ldots$ similarly.

We also use the notations and analogous identifications above for the trees $\Gamma_{\xi,1}.D$ and $\Gamma_{\xi,\infty}.D$. By an abuse of notation, $\mathbb{P}_{\xi,1}$ (resp. $\mathbb{P}_{\xi,\infty}$) will denote the function defined on $\Gamma_{\xi,1}$ (resp. $\Gamma_{\xi,\infty}$) as well as $\Gamma_{\xi,1}.D$ (resp. $\Gamma_{\xi,\infty}.D$) defined by $\mathbb{P}_{\xi,1}(t.v)=\mathbb{P}_{\xi,1}(t)$ for $t.v\in \Gamma_{\xi,1}.D$ (resp. $\mathbb{P}_{\xi,\infty}(t.v)=\mathbb{P}_{\xi,\infty}(t)$ for $t\in \Gamma_{\xi,\infty}.D$).    The functions $\lambda_i$ will also be defined on subsets of $\Gamma_{\xi,\infty}.D\cup [\Gamma_{\xi,\infty}].D$ in the analogous way.

We collect the following obvious facts regarding these constructions, which indicate the connection between our functions defined above and convex combinations. For a more thorough discussion of $\Gamma_{\xi,1}$ and $\Gamma_{\xi,\infty}$, see \cite{Causey.5} and \cite{Causey1}.

\begin{proposition} Let $\xi$ be an ordinal. \begin{enumerate}[(i)]\item For each $t\in MAX(\Gamma_{\xi,1})$, $\sum_{s\leqslant t}\mathbb{P}_{\xi,1}(s)=1$. \item For each $\tau\in [\Gamma_{\xi,\infty}]$ and each $n\in\nn$, $\sum_{\Lambda_{\xi,\infty,n}\ni s<\tau} \mathbb{P}_{\xi,\infty}(s)=\sum_{\lambda_{n-1}(\tau)<t\leqslant \lambda_n(\tau)}\mathbb{P}_{\xi,\infty}(t) =1$. \end{enumerate}

\label{co}
\end{proposition}

Next, we prove an improved simultaneity result. 

\begin{lemma} Let $\xi$ be an ordinal and  $X$ a Banach space. Let $G_\varnothing \subset X$ be norm compact and let $(G_a)_{a\in \Gamma_{\xi,1}.D}\subset \mathfrak{c}$ be normally weakly null.   Then for any $\ee>0$ and $\sigma_1>0$, there exists $a\in MAX(\Gamma_{\xi,1}.D)$ such that for all $y\in G_\varnothing$, all $\sigma$ with $|\sigma|\leqslant \sigma_1$, and all $(x_b)_{b\leqslant a}\in \prod_{b\leqslant a}G_b$,  $$\Bigl\|y+\sigma \sum_{b\leqslant a} \mathbb{P}_{\xi,1}(b)x_b\Bigr\|<1+\ee+\varrho^\xi_X(\sigma, y).$$  

\label{bost}

\end{lemma}

\begin{proof}  Toward a contradiction, suppose we have compact sets $G_\varnothing$, a normally weakly null collection  $(G_a)_{a\in \Gamma_{\xi,1}.D}\subset \mathfrak{c}$, and $\ee>0$ such that for each $a\in MAX(\Gamma_{\xi,1}.D)$, there exist $y_a\in G_\varnothing$, $\sigma_a$ with $|\sigma_a|\leqslant \sigma_1$, and  $(x^a_b)_{b\leqslant a}\in \prod_{b\leqslant a} G_b$ such that $$\Bigl\|y_a+\sigma_a \sum_{b\leqslant a} \mathbb{P}_{\xi,1}(b)x_b^a\Bigr\|\geqslant 1+\ee+\varrho^\xi_X(|\sigma_a|, y_a).$$ For each $a\in MAX(\Gamma_{\xi,1}.D)$, fix $x^*_a\in B_{X^*}$ such that $$\text{Re\ }x^*_a\Bigl(y_a+\sigma_a \sum_{b\leqslant a} \mathbb{P}_{\xi,1}(b)x_b^a\Bigr)= \Bigl\|y_a+\sigma_a \sum_{b\leqslant a} \mathbb{P}_{\xi,1}(b)x_b^a\Bigr\|.$$   

Let $$\Pi=\{(b,a)\in \Gamma_{\xi,1}.D\times \Gamma_{\xi,1}.D: b\leqslant a\in MAX(\Gamma_{\xi,1}.D)\}$$ Define $f:\Pi\to \rr$ by $$f(b,a)=\text{Re\ }x^*_a\bigl(y_a+\sigma_a x^a_b\bigr) - \varrho^\xi_X(|\sigma_a|, y_a).$$  Note that since $G_\varnothing$ is bounded, $|\sigma_a|\leqslant 1$,  $G_b\subset B_X$ for each $b\in \Gamma_{\xi,1}.D$, and $\varrho^\xi_X(\sigma, y)\leqslant \|y\|+\sigma$ for any $y\in X$ and $\sigma\in \mathbb{K}$,  it follows that $f$ is a bounded function. By hypothesis, for each $a\in MAX(\Gamma_{\xi,1}.D)$, \begin{align*} \sum_{b\leqslant a} \mathbb{P}_{\xi,1}(b)f(b,a) & = \sum_{b\leqslant a} \mathbb{P}_{\xi,1}(b) \Bigl[\text{Re\ }x^*_a\bigl(y_a+\sigma_a x^a_b\bigr)- \varrho^\xi_X(|\sigma_a|, y_a)\Bigr] \\ & = \text{Re\ }x^*_a\Bigl(y_a+\sigma_a \sum_{b\leqslant a} \mathbb{P}_{\xi,1}(b)x^a_b\Bigr) - \varrho^\xi_X(|\sigma_a|, y_a)\\ &  = \Bigl\|y_a+\sigma_a \sum_{b\leqslant a} \mathbb{P}_{\xi,1}(b)x_b^a\Bigr\| - \varrho^\xi_X(|\sigma_a|, y_a) \geqslant 1+\ee. \end{align*}  By \cite[Theorem $4.2$]{Causey.5} applied with $\ee$ replaced by $1+\ee$ and $\delta=\ee/2$, there exist  functions $d:\Gamma_{\xi,1}.D\to \Gamma_{\xi,1}.D$ and $e:MAX(\Gamma_{\xi,1}.D)\to MAX(\Gamma_{\xi,1}.D)$ such that \begin{enumerate}[(i)]\item for each  $b,a\in \Gamma_{\xi,1}.D$ such that $b<a$, it follows that  $d(b)<d(a)$, \item for each $a\in MAX(\Gamma_{\xi,1}.D)$, $d(a)\leqslant e(a)$, \item if $b=(\zeta_i, u_i)_{i=1}^m$ and $d(b)=(\nu_i, v_i)_{i=1}^n$, then $v_n\subset u_m$, \item for each $(b,a)\in \Pi$, either $f(d(b),e(a)) \geqslant 1+\ee-\ee/2=1+\ee/2$ or $$\sum_{b\leqslant e(a)} \mathbb{P}_{\xi,1}(b) f(b, e(a))<1+\ee.$$\end{enumerate} Above we showed that the inequality $\sum_{b\leqslant e(a)} \mathbb{P}_{\xi,1}(b) f(b,e(a))<1+\ee$ in (iv) is not possible, so $f(d(b), e(a))\geqslant 1+\ee/2$ for all $(b,a)\in \Pi$.  

Note that item (iii) implies that the collection $(G_{\phi(b)})_{b\in \Gamma_{\xi,1}.D}$ is also normally weakly null.  Define $F_\varnothing=G_\varnothing$ and $F_b=G_{\phi(b)}$ for each $b\in \Gamma_{\xi,1}.D$.    By Corollary \ref{cor1}, there exists $a\in MAX(\Gamma_{\xi,1}.D)$ such that for every $y\in F_\varnothing$, $\sigma$ with $|\sigma|\leqslant\sigma_1$,and $(x_b)_{b\leqslant a}\in \prod_{b\leqslant a}F_b$,  there exists $x\in \text{co}(x_b:b\leqslant a)$ such that  $$\|y+\sigma x\|-\varrho^\xi_X(|\sigma|, y)<1+\ee/2.$$    However, for each $a\in MAX(\Gamma_{\xi,1}.D)$, $y_{e(a)}\in F_\varnothing$, $(x^{e(a)}_{d(b)})_{b\leqslant a}\in \prod_{b\leqslant a}F_b$, and $|\sigma_{e(a)}|\leqslant \sigma_1$, but for each $x=\sum_{b\leqslant a} w_b x^{e(a)}_{d(b)}\in \text{co}(x^{e(a)}_{d(b)}: b\leqslant a)$, \begin{align*} \|y_{e(b)} + \sigma_{e(a)} x\| -\varrho^\xi_X(|\sigma_{e(a)}|, y_{e(a)}) & \geqslant \text{Re\ }x^*_{e(a)}(y_{e(a)}+\sigma_{e(a)} x) -\varrho^\xi_X(|\sigma_{e(a)}|, y_{e(a)}) \\ & = \sum_{b\leqslant a}w_b f(d(b),e(a)) \geqslant \sum_{b\leqslant a}w_b (1+\ee/2) = 1+\ee/2.\end{align*} This contradiction finishes the proof.

\end{proof}

\begin{corollary} Let $Y$ be a Banach space and let $B\subset B_Y$ be such that $\overline{\text{\emph{co}}}(B)=B_Y$.  Then for any $\sigma\geqslant 0$, $$\varrho_Y^\xi(\sigma)= \sup_{y\in B} \varrho_Y^\xi(\sigma,y).$$  

\label{vict}
\end{corollary}

\begin{proof} It clear that for any $\sigma\geqslant 0$, $$\varrho_Y^\xi(\sigma)=\underset{y\in B_Y}{\sup} \varrho^\xi_Y(\sigma, y) \geqslant \underset{y\in \text{co}(B)}{\sup} \varrho^\xi_Y(\sigma,y)\geqslant \underset{y\in B}{\sup}\text{\ } \varrho^\xi_Y(\sigma,y).$$  We will show the reverse inequalities. 

Recall that for each $\sigma\geqslant 0$, $\varrho^\xi_Y(\sigma,\cdot)$ is $1$-Lipschitz. Since $\text{co}(B)$ is dense in $B_Y$, $$\underset{y\in B_Y}{\sup} \varrho^\xi_Y(\sigma, y) \leqslant \underset{y\in \text{co}(B)}{\sup} \varrho^\xi_Y(\sigma,y).$$   

Fix $\sigma\geqslant 0$, $C>\sup_{y\in B}\varrho_Y^\xi(\sigma,y)$, $y_1, \ldots, y_n$, a normally weakly null collection $(x_a)_{a\in \Gamma_{\xi,1}.D}\subset B_Y$, and non-negative numbers $w_1, \ldots, w_n$ such that $1=\sum_{i=1}^n w_i$.   By Lemma \ref{bost} applied with $G_0=\{y_1, \ldots, y_n\}$ and $G_a=\{x_a\}$, there exists $a\in MAX(\Gamma_{\xi,1}.D)$ such that for each $1\leqslant i\leqslant n$, $$\Bigl\|y_i+\sigma \sum_{b\leqslant a} \mathbb{P}_{\xi,1}(b)x_b\Bigr\| \leqslant 1+C.$$   Then $$\Bigl\|\sum_{i=1}^n w_iy_i + \sigma\sum_{b\leqslant a} \mathbb{P}_{\xi,1}(b)x_b\Bigr\| \leqslant \sum_{i=1}^n w_i \Bigl\|y_i+\sigma \sum_{b\leqslant a}\mathbb{P}_{\xi,1}(b)x_b\Bigr\| \leqslant 1+C.$$   This shows that $$\varrho^\xi_Y\Bigl(\sigma, \sum_{i=1}^n w_iy_i\Bigr)\leqslant C.$$  Since $C>\sup_{y\in B}\varrho_Y^\xi(\sigma, y)$, $y_1, \ldots, y_n\in B$, and $w_1, \ldots, w_n$ were arbitrary, $$\underset{y\in \text{co}(B)}{\sup} \varrho^\xi_Y(\sigma,y)\leqslant \underset{y\in B}{\sup}\text{\ } \varrho^\xi_Y(\sigma,y).$$

\end{proof}

\begin{proposition} Fix an ordinal $\xi$ and $1<p<\infty$. For a Banach space $X$, the following are equivalent. \begin{enumerate}[(i)]\item $X$ is $\xi$-$p$-AUS. 

\item There exists a constant $c>0$ such that for any compact $G_\varnothing\subset B_X$, any normally weakly null collection $(G_a)_{a\in \Gamma_{\xi,1}.D}\subset \mathfrak{c}$, and any $\ee>0$, there exists $a\in MAX(\Gamma_{\xi,1}.D)$ such that for any $y\in G_\varnothing$, any $\sigma$ with $|\sigma|\leqslant 1$, and any $(x_b)_{b\leqslant a}\in \prod_{b\leqslant a} G_b$, $$\Bigl\|y+\sigma \sum_{b\leqslant a}\mathbb{P}_{\xi,1}(b)x_b\Bigr\|\leqslant 1+c^p|\sigma|^p+\ee.$$   

\item There exists a constant $c_1>0$ such that for any compact $G_\varnothing\subset X$, any normally weakly null collection $(G_a)_{a\in \Gamma_{\xi,1}.D}\subset \mathfrak{c}$, and any $\ee>0$, there exists $a\in MAX(\Gamma_{\xi,1}.D)$ such that for any $y\in G_\varnothing$,any $\sigma$ with $|\sigma|\leqslant 1$, and any $(x_b)_{b\leqslant a}\in \prod_{b\leqslant a}G_b$, $$\Bigl\|y+\sigma \sum_{b\leqslant a}\mathbb{P}_{\xi,1}(b)x_b\Bigr\|^p \leqslant \|y\|^p+c_1^p |\sigma|^p + \ee.$$   \end{enumerate}

\label{chuff}
\end{proposition}

\begin{proof}$(i)\Rightarrow (ii)$ If $X$ is $\xi$-$p$-AUS, then $c:= \sup_{\sigma>0}\varrho^\xi_X(\sigma)/\sigma^p<\infty$.   Then by Lemma \ref{bost}, $X$ satisfies $(ii)$ with this choice of $c$.  

$(ii)\Rightarrow (i)$ The property in $(ii)$ clearly implies that $\varrho_X^\xi(\sigma) \leqslant c\sigma^p$ for any $\sigma>0$.

$(ii)\Rightarrow (iii)$ Assume $(ii)$ holds with constant $c$. Fix $G_\varnothing\subset X$ norm compact, a normally weakly null collection $(G_a)_{a\in \Gamma_{\xi,1}.D}\subset \mathfrak{c}$, and $\ee>0$.   Fix $R>\sup_{x\in G_\varnothing} \|x\|$ and  $$0<\delta<\min \{\ee^{1/p}/2,\ee/R^p\}.$$    Define $F_\varnothing=\{x/\|x\|: x\in G_\varnothing, \|x\|\geqslant \delta\}$, which is a norm compact subset of $B_X$. Define $F_a=G_a$ for $a\in \Gamma_{\xi,1}.D$, which is normally weakly null.  By hypothesis, there exists $a\in \Gamma_{\xi,1}.D$ such that for any  $y\in F_\varnothing$, any $\sigma$ with $|\sigma|\leqslant 1$,and any $(x_b)_{b\leqslant a}\in \prod_{b\leqslant a}F_b$,  $$\Bigl\|y+\sigma \sum_{b\leqslant a}\mathbb{P}_{\xi,1}(b)x_b\Bigr\|\leqslant 1+c^p|\sigma|^p + \delta.$$   Now fix $y\in G_\varnothing$, a scalar $\sigma$ with $|\sigma|\leqslant 1$, and $(x_b)_{b\leqslant a}\in \prod_{b\leqslant a}G_b$.  Consider three cases.   

Case $1$: $\|y\|\leqslant |\sigma|$.  Then by the triangle inequality, \begin{align*}  \Bigl\|y+\sigma \sum_{b\leqslant a}\mathbb{P}_{\xi,1}(b)x_b\Bigr\|^p \leqslant (\|y\|+|\sigma|)^p \leqslant 2^p |\sigma|^p \leqslant \|y\|^p+(2^p+c^p)|\sigma|^p +\ee.\end{align*}

Case $2$: $|\sigma|<\|y\|\leqslant \delta$.    Then by the triangle inequality, \begin{align*} \Bigl\|y+\sigma \sum_{b\leqslant a}\mathbb{P}_{\xi,1}(b)x_b\Bigr\|^p & \leqslant [\delta+|\sigma|]^p \leqslant 2^p\delta^p <\ee \\ & \leqslant \|y\|^p+(2^p+c^p) |\sigma|^p + \ee. \end{align*}

Case $3$: $\|y\|\geqslant \max\{\delta, |\sigma|\}$.  Then $|\sigma/\|y\| |\leqslant 1$ and $y/\|y\|\in F_\varnothing$, so \begin{align*}\Bigl\|y+\sigma \sum_{b\leqslant a}\mathbb{P}_{\xi,1}(b)x_b\Bigr\|^p & = \|y\|^p \Bigl\|\frac{y}{\|y\|} + \frac{\sigma}{\|y\|} \sum_{b\leqslant a}\mathbb{P}_{\xi,1}(b)x_b\Bigr\|^p \\ & \leqslant \|y\|^p\Bigl(1+ c^p |\sigma|^p/\|y\|^p + \delta\Bigr) < \|y\|^p+(2^p+c^p) |\sigma|^p + \ee.   \end{align*}

Therefore $(iii)$ is satisfied with $c_1^p=2^p+c^p$.

$(iii)\Rightarrow (ii)$ Assume $(iii)$ holds with constant $c_1$. Fix  compact $G_\varnothing\subset B_X$, any normally weakly null collection $(G_a)_{a\in \Gamma_{\xi,1}.D}\subset \mathfrak{c}$, and $\ee>0$.  By uniform continuity of the function $f(x)=(1+x)^{1/p}$ on $[0, 2+c_1^p]$, there exists $\delta>0$ such that $(1+ c_1^p |\sigma|^p + \delta)^{1/p} \leqslant (1+c_1^p |\sigma|^p)^{1/p}+\ee$ for any $\sigma$ with $|\sigma|\leqslant 1$.     Note also that since the function $f$ is concave on $[0,\infty)$, for any $x>0$, $$\frac{(1+x)^{1/p}-1}{x} \leqslant f'(0)=1/p.$$    Therefore for any $\sigma$ with $|\sigma|\leqslant 1$, $$(1+c_1^p|\sigma|^p+\delta)^{1/p} \leqslant 1+\frac{c_1^p|\sigma|^p}{p}+\ee.$$   

  By our assumption that $(iii)$ holds with constant $c_1$, there exists $a\in MAX(\Gamma_{\xi,1}.D)$ such that for any $y\in G_\varnothing$, any $\sigma$ with $|\sigma|\leqslant 1$, and any $(x_b)_{b\leqslant a} \in \prod_{b\leqslant a}G_b$, $$\Bigl\|y+\sigma \sum_{b\leqslant a}\mathbb{P}_{\xi,1}(b)x_b\Bigr\|\leqslant \Bigl(\|y\|^p + c_1^p |\sigma|^p + \delta\Bigr)^{1/p}.$$   Since $G_\varnothing\subset B_X$, we deduce that for any $y\in G_\varnothing$, any $\sigma$ with $|\sigma|\leqslant 1$,  any $(x_b)_{b\leqslant a}\in \prod_{b\leqslant a}G_b$, $$\Bigl\|y+\sigma \sum_{b\leqslant a}\mathbb{P}_{\xi,1}(b)x_b\Bigr\|\leqslant \Bigl(\|y\|^p + c_1^p |\sigma|^p + \delta\Bigr)^{1/p} \leqslant 1+\frac{c_1^p|\sigma|^p}{p} +\ee.$$  Therefore $(ii)$ holds with $c^p=c_1^p/p$.

\end{proof}

\begin{proposition} Fix an ordinal $\xi$. For a Banach space $X$, the following are equivalent. \begin{enumerate}[(i)]\item $X$ is $\xi$-AUF. 

\item There exists a constant $\sigma_0>0$ such that for any compact $G_\varnothing\subset B_X$, any normally weakly null collection $(G_a)_{a\in \Gamma_{\xi,1}.D}\subset \mathfrak{c}$, and any $\ee>0$, there exists $a\in MAX(\Gamma_{\xi,1}.D)$ such that for any $y\in G_\varnothing$, any $\sigma$ with $|\sigma|\leqslant \sigma_0$, and any $(x_b)_{b\leqslant a}\in \prod_{b\leqslant a}G_b$, $$\Bigl\|y+\sigma \sum_{b\leqslant a}\mathbb{P}_{\xi,1}(b)x_b\Bigr\|\leqslant 1+\ee.$$   

\item There exists a constant $c_1>0$ such that for any compact $G_\varnothing\subset X$, any normally weakly null collection $(G_a)_{a\in \Gamma_{\xi,1}.D}\subset \mathfrak{c}$, and any $\ee>0$, there exists $a\in MAX(\Gamma_{\xi,1}.D)$ such that for any $y\in G_\varnothing$, any  scalar $\sigma$ with $|\sigma|\leqslant 1$, and any $(x_b)_{b\leqslant a}\in \prod_{b\leqslant a}G_b$, $$\Bigl\|y+\sigma \sum_{b\leqslant a}\mathbb{P}_{\xi,1}(b)x_b\Bigr\| \leqslant \max\{ \|y\|, c_1 |\sigma|\}+\ee.$$ \end{enumerate}

\label{chuff2}
\end{proposition}

\begin{proof}$(i)\Rightarrow (ii)$ If $X$ is $\xi$-AUF, then there exists $\sigma_0>0$ such that $\varrho_X^\xi(\sigma_0)=0$. By Lemma \ref{bost} applied with $\sigma_1=\sigma_0$, $X$ satisfies $(ii)$ with this choice of $\sigma_0$.  

$(ii)\Rightarrow (i)$ The property in $(ii)$ clearly implies that $\varrho_X^\xi(\sigma_0) =0$.

$(ii)\Rightarrow (iii)$ Assume $(ii)$ holds with constant $\sigma_0>0$.    Fix $G_\varnothing\subset X$ norm compact, $(G_a)_{a\in \Gamma_{\xi,1}.D}\subset \mathfrak{c}$ normally weakly null. Let $$F_\varnothing=\{x/\|x\|: x\in F_\varnothing, \|x\|\geqslant \ee\}$$ and for each $a\in \Gamma_{\xi,1}.D$, let $F_a=G_a$.  Fix $0<\delta$ such that $\delta G_\varnothing\subset \ee B_X$.   Since $(ii)$ holds with constant $\sigma_0$, there exists $a\in MAX(\Gamma_{\xi,1}.D)$ such that for any $y\in F_\varnothing$, any $\sigma$ with $|\sigma|\leqslant 1$, and any $(x_b)_{b\leqslant a}\in \prod_{b\leqslant a}F_b$, $$\Bigl\|y+\sigma \sum_{b\leqslant a}\mathbb{P}_{\xi,1}(b)x_b\Bigr\|\leqslant 1+\delta.$$

Case $1$: $\|y\|\leqslant \ee$.  Then \begin{align*}\Bigl\|y+\sigma \sum_{b\leqslant a}\mathbb{P}_{\xi,1}(b)x_b\Bigr\| & \leqslant \ee + |\sigma|  \leqslant \max\{\|y\|, (1+1/\sigma_0)|\sigma|\}+\ee. \end{align*}

Case $2$: $\|y\|\leqslant |\sigma|/\sigma_0$.  Then by the triangle inequality, \begin{align*}\Bigl\|y+\sigma \sum_{b\leqslant a}\mathbb{P}_{\xi,1}(b)x_b\Bigr\| & \leqslant \|y\|+|\sigma| \leqslant (1+1/\sigma_0)|\sigma| \\ & \leqslant \max\{\|y\|, (1+1/\sigma_0)|\sigma|\}+\ee. \end{align*}

Case $3$: $\|y\|\geqslant \max\{|\sigma|/\sigma_0,\ee\}$.  Then \begin{align*} \Bigl\|y+\sigma \sum_{b\leqslant a}\mathbb{P}_{\xi,1}(b)x_b\Bigr\| & = \|y\|\Bigl\|\frac{y}{\|y\|} + \frac{\sigma}{\|y\|}\sum_{b\leqslant a}\mathbb{P}_{\xi,1}(b)x_b\Bigr\| \leqslant \|y\|(1+\delta) \leqslant \|y\|+\ee \\ & \leqslant \max\{\|y\|, (1+1/\sigma_0)|\sigma|\}+\ee. \end{align*}

Therefore $(iii)$ is satisfied with constant $c_1=1+1/\sigma_0$. 

$(iii)\Rightarrow (ii)$  Fix $G_\varnothing \subset B_X$ compact and $(G_a)_{a\in \Gamma_{\xi,1}.D}\subset \mathfrak{c}$ normally weakly null. By hypothesis, there exists $a\in MAX(\Gamma_{\xi,1}.D)$ such that for any $y\in G_\varnothing$, $\sigma$ with $|\sigma|\leqslant 1$, and $(x_b)_{b\leqslant a}\in \prod_{b\leqslant a}G_b$, \[\Bigl\|y+\sigma \sum_{b\leqslant a}\mathbb{P}_{\xi,1}(b)x_b\Bigr\|\leqslant \max\{\|y\|,c_1|\sigma|\}+\ee.\]   It follows that with $\sigma_0=1/c_1$ that for any $y\in G_\varnothing$, any $\sigma$ with $|\sigma|\leqslant \sigma_0$, and any $(x_b)_{b\leqslant a}\in \prod_{b\leqslant a}G_b$, \begin{align*} \Bigl\|y+\sigma \sum_{b\leqslant a}\mathbb{P}_{\xi,1}(b)x_b\Bigr\| & \leqslant \max\{\|y\|, c_1|\sigma|\} + \ee \leqslant 1+\ee.\end{align*}

\end{proof}

\begin{lemma} Let $\xi$ be an ordinal and $1<p<\infty$. Let $(X, \|\cdot\|)$ be a Banach space and suppose $|\cdot|$ is an equivalent norm on $X$.  Assume $B\subset B_X^{|\cdot|}$ is such that $\overline{\text{\emph{co}}}(B)=B_X^{|\cdot|}$. \begin{enumerate}[(i)]\item If  for each $y\in B$, each $\sigma\geqslant 0$,  and each $(x_a)_{a\in \Gamma_{\xi,1}.D}\subset B_X^{\|\cdot\|}$ normally weakly null in $(X,\|\cdot\|)$, $$\underset{a\in MAX(\Gamma_{\xi,1}.D)}{\inf} \Bigl|y+\sigma \sum_{b\leqslant a}\mathbb{P}_{\xi,1}(b)x_b\Bigr|^p \leqslant 1+\sigma^p,$$ then $(X, |\cdot|)$ is $\xi$-$p$-AUS. \item If for each $y\in B$, each $\sigma>0$, and each $(x_a)_{a\in \Gamma_{\xi,1}.D}\subset B_X^{\|\cdot\|}$ normally weakly null in $(X,\|\cdot\|)$, , $$\underset{a\in MAX(\Gamma_{\xi,1}.D)}{\inf} \Bigl|y+\sigma \sum_{b\leqslant a}\mathbb{P}_{\xi,1}(b)x_b\Bigr| \leqslant 1,$$ then $(X, |\cdot|)$ is $\xi$-AUF.  \end{enumerate}

\label{13}
\end{lemma}

\begin{proof}$(i)$ Fix $C>0$ such that $\frac{1}{C}B_X^{|\cdot|}\subset B_X^{\|\cdot\|}$.  Fix $y\in B$,  $(x_a)_{a\in \Gamma_{\xi,1}.D}\subset B_X^{|\cdot|}$ normally weakly null in $(X, |\cdot|)$, and $\sigma\geqslant 0$.  Then $(C^{-1} x_a)_{a\in \Gamma_{\xi,1}.D}\subset B_X^{\|\cdot\|}$ is normally weakly null in $(X, \|\cdot\|)$. Therefore \begin{align*} \underset{a\in MAX(\Gamma_{\xi,1}.D)}{\inf} \Bigl|y+\sigma \sum_{b\leqslant a}\mathbb{P}_{\xi,1}(b) x_b\Bigr| & = \underset{a\in MAX(\Gamma_{\xi,1}.D)}{\inf} \Bigl|y+\sigma C \sum_{b\leqslant a}\mathbb{P}_{\xi,1}(b) C^{-1}x_b\Bigr| \\ & \leqslant (1+C^p|\sigma|^p)^{1/p}\leqslant 1+\frac{C^p|\sigma|^p}{p}. \end{align*} Here we used the concavity of $f(x)=(1+x)^{1/p}$ together with the fact that $f'(0)=1/p$.   This shows that $\varrho^\xi_{(X, |\cdot|)}(\sigma, y) \leqslant  C^p|\sigma|^p/p$ for any $y\in B$.   We then deduce that $\sup_{\sigma>0} \varrho^\xi_{(X, |\cdot|)}(\sigma)/\sigma^p\leqslant C^p/p$ by Corollary \ref{vict}.

$(ii)$  Fix $C>0$ such that $\frac{1}{C}B_X^{|\cdot|}\subset B_X^{\|\cdot\|}$.  Fix $y\in B$,  $(x_a)_{a\in \Gamma_{\xi,1}.D}\subset B_X^{|\cdot|}$ normally weakly null in $(X, |\cdot|)$, and $\sigma\geqslant 0$.  Then $(C^{-1} x_a)_{a\in \Gamma_{\xi,1}.D}\subset B_X^{\|\cdot\|}$ is normally weakly null in $(X, \|\cdot\|)$. Therefore with $\sigma_0=1/C$,  \begin{align*} \underset{a\in MAX(\Gamma_{\xi,1}.D)}{\inf} \Bigl|y+\sigma_0 \sum_{b\leqslant a}\mathbb{P}_{\xi,1}(b) x_b\Bigr| & = \underset{a\in MAX(\Gamma_{\xi,1}.D)}{\inf} \Bigl|y+ \sum_{b\leqslant a}\mathbb{P}_{\xi,1}(b) C^{-1}x_b\Bigr|  \leqslant 1.  \end{align*}  Therefore for any $y\in B$, $\varrho^\xi_{(X, |\cdot|)}(\sigma_0,y)=0$.  By Corollary \ref{vict}, $\varrho^\xi_{(X, |\cdot|)}(\sigma_0)=0$.   
\end{proof}

\section{Two renorming theorems}

In this section, we provide an isomorphic characterization of $\xi$-$p$-AUS renormability and of $\xi$-AUF renormability in terms of games.    Isomorphic characterizations of these properties were previously given in \cite{Causey} in terms of big and inevitable subsets of $\Gamma_{\xi,\infty}.D$, quite technical notions which we are able to avoid in the present paper.   

Recall that for a Banach space $X$, $1\leqslant q<\infty$,  and a sequence $(x_n)_{n=1}^\infty\subset X$, we define \[\|(x_n)_{n=1}^\infty\|_q^w= \sup\Bigl\{\Bigl(\sum_{n=1}^\infty |x^*(x_n)|^q\Bigr)^{1/q}: x^*\in B_{X^*}\Bigr\}.\] We also use this notation for finite sequences. That is, we denote \[\|(x_n)_{n=1}^m\|_q^w= \sup\Bigl\{\Bigl(\sum_{n=1}^m |x^*(x_n)|^q\Bigr)^{1/q}: x^*\in B_{X^*}\Bigr\}.\]   We note that with $1/p+1/q=1$, $\|(x_n)_{n=1}^\infty\|_q^w$ (resp. $\|(x_n)_{n=1}^m\|_q^w$) is the operator norm of the formal inclusion $I:(c_{00}, \ell_p)\to X$ given by $e_n\mapsto x_n$.    Also for a sequence $(x_n)_{n=1}^\infty$, we let $\|(x_n)_{n=1}^\infty\|_\infty=\sup_n \|x_n\|$, and we use the same notation for finite sequences $(x_n)_{n=1}^m$.

Recall that for a Banach space $X$, $CD(X)$ denotes the set of finite codimensional subspaces of $X$. In this section, let $\xi$ be an ordinal,   $X$  a Banach space, and  $D=CD(X)$. Recall that $\mathfrak{s}, \mathcal{F}, \mathfrak{c}$ denote the sets of singleton, non-empty finite, and non-empty compact subsets of $B_X$, respectively.     

Recall that for each $\tau\in [\Gamma_{\xi,\infty}]$, the sequence $\varnothing=\lambda_0(\tau)<\lambda_1(\tau)<\ldots$ are such that for each $n\in\nn$, $\lambda_n(\tau)$ is the initial segment of $\tau$ such that $\lambda_n(\tau)\in MAX(\Lambda_{\xi,\infty,n})$. We have a similar definition of $\lambda_0(t)<\ldots <\lambda_{n-1}(t)$ for each $t\in \Lambda_{\xi,\infty,n}$. Similarly, for $\alpha=\tau.\upsilon\in [\Gamma_{\xi,\infty}].D$, we define $\lambda_0(\alpha)=\varnothing$ and $\lambda_n(\alpha)\in MAX(\Lambda_{\xi,\infty,n}.D)$ for each $n\in\nn$. 

for a sequence $\gamma=(G_i)_{i=1}^\infty\in \mathfrak{c}^\omega$, we let $\prod \gamma=\prod_{i=1}^\infty G_i$.    Fix $1<p\leqslant \infty$ and let $1/p+1/q=1$.  For a constant $C$, we let $\mathcal{E}^p_C$ denote the set of all $\tau.\gamma=\in [\Gamma_{\xi,\infty}].\mathfrak{c}$ such that \[\sup\Bigl\{\Bigl\|\bigl(\sum_{\lambda_{n-1}(\tau)<t\leqslant \lambda_n(\tau)} \mathbb{P}_{\xi,\infty}(t)x_{|t|}\bigr)_{n=1}^\infty\Bigr\|_q^w: (x_i)_{i=1}^\infty\in \prod \gamma\Bigr\}\leqslant C.\]

If we let $\mathcal{F}^p_C$ denote the set of $t.g\in \Gamma_{\xi,\infty}.\mathfrak{c}$ such that, if $t\in \Lambda_{\xi,\infty,m}$, then \[\sup\Bigl\{\Bigl\|\bigl(\sum_{\lambda_{n-1}(\tau)<t\leqslant \lambda_n(\tau)} \mathbb{P}_{\xi,\infty}(t)x_{|t|}\bigr)_{n=1}^{m-1}\Bigr\|_q^w: (x_i)_{i=1}^\infty\in \prod \gamma\Bigr\}\leqslant C,\] then $\mathcal{E}^p_C=[\mathcal{F}^p_C]$. By Proposition \ref{steph}, $\mathcal{E}^p_C$ is closed, so either Player $S$ or Player $V$ has a winning strategy in the $(\mathcal{E}^p_C, \mathfrak{c}, \Gamma_{\xi,\infty})$ game.

Our next two results are the main renorming theorems of the paper.  

\begin{theorem} Fix $1<p<\infty$. The following are equivalent. \begin{enumerate}[(i)]\item $X$ admits an equivalent $\xi$-$p$-AUS norm. 

\item There exists $C>0$ such that Player $S$ has a winning strategy in the $(\mathcal{E}^p_C, \mathfrak{c}, \Gamma_{\xi,\infty})$ game.  

\item There exists $C>0$ such that Player $S$ has a winning strategy in the $(\mathcal{E}^p_C, \mathcal{F}, \Gamma_{\xi,\infty})$ game.

\item There exists $C>0$ such that Player $S$ has a winning strategy in the $(\mathcal{E}^p_C, \mathfrak{s}, \Gamma_{\xi,\infty})$ game.
\end{enumerate}
\label{renorm1}
\end{theorem}

\begin{proof}$(i)\Rightarrow (ii)$ Since the condition in $(ii)$ is an isomorphic invariant, we can assume $X$ is $\xi$-$p$-AUS.   By Theorem \ref{chuff}, there exists a constant $c>0$ such that for any compact $G_\varnothing\subset X$,  any normally weakly null $(G_a)_{a\in \Gamma_{\xi,1}.D}\subset B_X$, and any $\ee>0$, there exists $a\in MAX(\Gamma_{\xi,1}.D)$ such that for any $y\in G_\varnothing$, any scalar $\sigma$ with $|\sigma|\leqslant 1$, and any $(x_b)_{b\leqslant a}\in \prod_{b\leqslant a}G_b$, $$\Bigl\|y+\sigma \sum_{b\leqslant a} \mathbb{P}_{\xi,1}(b) x_b\Bigr\|^p \leqslant \|y\|+c^p|\sigma|^p + \ee.$$   By replacing $c$ with a larger value if necessary, we may assume $c\geqslant 1$.   We claim that for any $C>c$, Player $S$ has a winning strategy in the $(\mathcal{E}^p_C, \mathfrak{c}, \Gamma_{\xi,\infty})$ game.    We prove this by contradiction. Assume that $C>c$ is such that Player $S$ does not have a winning strategy in the $(\mathcal{E}^p_C, \mathfrak{c}, \Gamma_{\xi,\infty})$ game.

Since $\mathcal{E}^p_C$ is closed and we have assumed Player $S$ does not have a winning strategy in the $(\mathcal{E}^p_C, \mathfrak{c}, \Gamma_{\xi,\infty})$ game, Proposition \ref{steph} yields that Player $V$ has a winning strategy in this game.  By Lemma \ref{tree}, there exists a collection $(G_a)_{a\in \Gamma_{\xi,\infty}.D}\subset \mathfrak{c}$ such that \begin{enumerate}[(a)]\item for $a=(\zeta_i, u_i)_{i=1}^n\in \Gamma_{\xi,\infty}.D$, $G_a\subset u_n$, \item for each $\alpha=\tau.\upsilon\in [\Gamma_{\xi,\infty}].D$, $\tau.(G_{\alpha|n})_{n=1}^\infty\in [\Gamma_{\xi,\infty}].\mathfrak{c}\setminus \mathcal{E}^p_C$.  \end{enumerate}   We will recursively select $a_1<a_2<\ldots$ such that for all $n\in\nn$, $a_n\in MAX(\Lambda_{\xi,\infty,\infty}.D)$ and such that, if $\alpha=\tau.\upsilon\in [\Gamma_{\xi,\infty}].D$ is the sequence such that $\lambda_n(\alpha)=a_n$ for all $n\in\nn$, then $\tau.(G_{\alpha|n})_{n=1}^\infty \in \mathcal{E}^p_C$. This contradiction will finish the first implication.   

Fix $(\ee_n)_{n=2}^\infty \subset (0,1)$ such that $c^p+\sum_{n=2}^\infty \ee_n <C^p$.  Let $a_0=\varnothing$.    Fix $a_1\in MAX(\Lambda_{\xi,\infty,1}.D)$ arbitrary.  Now assume that $a_1<\ldots <a_n$ have been chosen.   Recall that $\Gamma_{\xi,1}.D$ is canonically identifiable with $$\{a\in \Lambda_{\xi,\infty,n+1}.D: a_n <a\}$$  via the map $a\mapsto a_n\smallfrown (\omega^\xi n + a)$.   Let $$F_\varnothing=\Bigl\{\sum_{l=1}^n \sum_{a_{l-1}<a\leqslant a_n} c_l \mathbb{P}_{\xi,\infty}(a) x_a: (c_l)_{l=1}^n\in B_{\ell_\infty^n}, (x_a)_{a\leqslant a_n}\in \prod_{a\leqslant a_n}G_a\Bigr\},$$ which is norm compact.   Define $F_a= G_{f(a)}$, where $f$ is the bijection above.  Then there exists $a\in MAX(\Gamma_{\xi,1}.D)$ such that for any $y\in F_\varnothing$, any $\sigma$ with $|\sigma|\leqslant 1$, and any $(x_b)_{b\leqslant a}\in \prod_{b\leqslant a}F_b$, $$\Bigl\|y+\sigma \sum_{b\leqslant a} \mathbb{P}_{\xi,1}(b) x_b\Bigr\|^p \leqslant \|y\|^p +c_1^p |\sigma|^p + \ee_{n+1}.$$   Let $a_{n+1}=f(a)$.  Then since $\mathbb{P}_{\xi,1}(b)=\mathbb{P}_{\xi,\infty}(f(a))$, it follows that for any $y\in G_\varnothing$,  any $(x_b)_{a_n<b\leqslant a_{n+1}}\in \prod_{a_n<b\leqslant a_{n+1}} G_b$, and any $\sigma$ with $|\sigma|\leqslant 1$, $$\Bigl\|y+\sigma \sum_{a_n<b\leqslant a_{n+1}}\mathbb{P}_{\xi,\infty}(b) x_b\Bigr\|^p \leqslant \|y\|^p +c_1^p |\sigma|^p + \ee_{n+1}.$$  This completes the recursive construction.   

Let $\alpha=\tau.\upsilon$ be the sequence which has $a_1, a_2, \ldots$ as initial segments. Let $\gamma=(G_{\alpha|n})_{n=1}^\infty$.    Fix $(x_a)_{a<\alpha}\in \prod \gamma$ and $(c_n)_{n=1}^\infty \in c_{00}$ such that $\sum_{n=1}^\infty |c_n|^p=1$.  Fix $m\in\nn$ such that $c_n=0$ for all $n>m$.  If $m=1$, then by the triangle inequality,  since $C\geqslant 1$, $$\Bigl\|\sum_{n=1}^\infty c_n \sum_{\lambda_{n-1}(\alpha)<a\leqslant \lambda_n(\alpha)} \mathbb{P}_{\xi,\infty}(a)x_a\Bigr\|^p=\Bigl\|c_1 \sum_{a\leqslant \lambda_1(\alpha)} \mathbb{P}_{\xi,\infty}(a)x_a\Bigr\|^p\leqslant C^p|c_1|^p \leqslant C^p\sum_{n=1}^\infty |c_n|^p.$$   If $m>1$, then by the preceding paragraph, letting $I_n=\{a\in \Gamma_{\xi,\infty}.D: a_{n-1}<a\leqslant a_n\}$, \begin{align*} \Bigl\|\sum_{n=1}^\infty c_n \sum_{\lambda_{n-1}(\alpha)<a\leqslant \lambda_n(\alpha)} \mathbb{P}_{\xi,\infty}(a)x_a\Bigr\|^p & =  \Bigl\|\sum_{n=1}^m c_n \sum_{a\in I_n } \mathbb{P}_{\xi,\infty}(a)x_a\Bigr\|^p \\ & = \Bigl\|\sum_{n=1}^{m-1} c_n \sum_{a\in I_n} \mathbb{P}_{\xi,\infty}(a)x_a + c_m \sum_{b\in I_m} \mathbb{P}_{\xi,\infty}(b) x_b\Bigr\|^p \\ & \leqslant \Bigl\|\sum_{n=1}^{m-1} c_n \sum_{a\in I_n} \mathbb{P}_{\xi,\infty}(a)x_a\Bigr\|^p + c^p |c_m|^p + \ee_m  \\ & \leqslant \Bigl\|\sum_{n=1}^{m-2} c_n \sum_{a\in I_n} \mathbb{P}_{\xi,\infty}(a)x_a  + c_{m-1}\sum_{b\in I_{m-1}} \mathbb{P}_{\xi,\infty}(b)x_b \Bigr\|^p \\ & + c^p |c_m|^p + \ee_m  \\ & \leqslant \Bigl\|\sum_{n=1}^{m-2} c_n \sum_{a\in I_n} \mathbb{P}_{\xi,\infty}(a)x_a \Bigr\|^p \\ & + c^p(|c_{m-1}|^p+ |c_m|^p) + (\ee_{m-1}+\ee_m) \\ & \leqslant \ldots \\ & \leqslant \Bigl\|c_1 \sum_{a\in I_1} \mathbb{P}_{\xi,\infty}(a)x_a\Bigr\|^p + c^p\sum_{n=2}^m |c_n|^p + \sum_{n=2}^\infty \ee_n \\ & \leqslant c^p\sum_{n=1}^\infty |c_n|^p + \sum_{n=2}^\infty \ee_n = c^p + \sum_{n=2}^\infty \ee_n < C^p.  \end{align*}  By homogeneity, $\Bigl\|\bigl(\sum_{\lambda_{n-1}(\alpha)<a\leqslant \lambda_n(\alpha)}\mathbb{P}_{\xi,\infty}(a)x_a\bigr)_{n=1}^\infty\Bigr\|_q^w\leqslant C$. Since $(x_a)_{a<\alpha}\in \prod \gamma$ was arbitrary, $\tau.\gamma\in \mathcal{E}^p_C$.  This is the necessary contradiction.

$(ii)\Rightarrow (iii)\Rightarrow (iv)$ These are clear, since any winning strategy for Player $S$ in the $(\mathcal{E}^p_C, \mathfrak{c}, \Gamma_{\xi,\infty})$ game is a winning strategy for Player $S$ in the $(\mathcal{E}^p_C, \mathcal{F}, \Gamma_{\xi,\infty})$ game, and any winning strategy for Player $(\mathcal{E}^p_C, \mathcal{F}, \Gamma_{\xi,\infty})$ game is a winning strategy for Player $S$ in the $(\mathcal{E}^p_C, \mathfrak{s}, \Gamma_{\xi,\infty})$ game. 

$(iv)\Rightarrow (i)$ Our proof is a modification of Pisier's celebrated renorming theorem for $p$-smooth Banach spaces \cite{Pisier2}. In this problem, for ease of notation, we will write $x$ in place of $\{x\}$ for members of $\mathfrak{s}$. We also refer to the games by their target set rather than by the usual triple, since $\mathfrak{s}$ and $\Gamma_{\xi,\infty}$ are understood.  Assume $C \geqslant 2$ is such that Player $S$ has a winning strategy in the $\mathcal{E}^p_{C/2}$ game.  Fix such a winning strategy $\chi_0$.   For each $y\in X$ and $\lambda>0$, let $\mathcal{F}_{y,\lambda}$ be the set of all $\tau.\gamma=(\zeta_t,x_t)_{t<\tau}$ such that for any $(c_n)_{n=1}^\infty\in c_{00}$, \[\frac{1}{C^p}\Bigl\|y+\sum_{n=1}^\infty c_n\sum_{t\in \Lambda_n(\tau)} \mathbb{P}_{\xi,\infty}(t)x_t\Bigr\|^p - \sum_{n=1}^\infty |c_n|^p \leqslant \lambda^p.\]

For each $y\in X$, let $[y]$ denote the infimum of $\lambda>0$ such that Player $S$ has a winning strategy in the $\mathcal{F}_{y,\lambda}$ game.   Let \[|y|= \inf\Bigl\{\sum_{i=1}^n [y_i]: n\in\nn, y=\sum_{i=1}^n y_i\Bigr\}.\]  We will prove that $|\cdot|$ is an equivalent $\xi$-$p$-AUS norm on $X$. 

Obviously $[y], |y|\geqslant 0$ for all $y\in X$.  Note that for any $y\in X$ and $\lambda>0$, if Player $S$ plays the game according to the strategy $\chi_0$ fixed above, and the game results in choices $\tau.\gamma=(\zeta_t, x_t)_{t<\tau}$, then for any $(c_n)_{n=1}^\infty$, \begin{align*} \frac{1}{C^p}\Bigl\|y+\sum_{n=1}^\infty c_n\sum_{t\in \Lambda_n(\tau)} \mathbb{P}_{\xi,\infty}(t)x_t\Bigr\|^p - \sum_{n=1}^\infty |c_n|^p &\leqslant \frac{2^p\|y\|^p}{C^p}+\frac{2^p}{C^p}\Bigl\|\sum_{n=1}^\infty c_n\sum_{t\in \Lambda_n(\tau)} \mathbb{P}_{\xi,\infty}(t)x_t\Bigr\|^p \\ & - \sum_{n=1}^\infty |c_n|^p \\ & \leqslant \|y\|^p+\sum_{n=1}^\infty |c_n|^p-\sum_{n=1}^\infty |c_n|^p \\ & =\|y\|^p. \end{align*}  Therefore for any $y\in X$,  $\chi_0$ is a winning strategy for Player $S$ in $\mathcal{F}_{y, \|y\|}$ game.    From this it follows that for any $y\in X$, $|y|\leqslant [y]\leqslant \|y\|$. Moreover, for a given $y\in X$, Player $S$ cannot have a winning strategy in the game $\mathcal{F}_{y, \lambda}$ game for any $\lambda<\|y\|/C$, which can be seen by playing any strategy of Player $S$ against the strategy for Player $V$ which consists of choosing the zero vector on every turn.  Therefore for any $y\in Y$, \[\|y\|/C \leqslant |y|\leqslant [y] \leqslant \|y\|.\]

Next, note that if for some $y\in X$ and $\lambda>0$, $\chi$ is a winning strategy for Player $S$ in the $\mathcal{F}_{y,\lambda}$ game, then for any non-zero scalar $c$,  $\chi$ is also a winning strategy for Player $S$ in the $\mathcal{F}_{cy, |c|\lambda}$ game. Indeed, if Player $S$ plays according to $\chi$, resulting in $\tau.\gamma=(\zeta_t, x_t)_{t<\tau}$, then for any $(c_n)_{n=1}^\infty\in c_{00}$,  \begin{align*} \Bigl\|cy +\sum_{n=1}^\infty c_n \sum_{t\in \Lambda_n(\tau)}\mathbb{P}_{\xi,\infty}(t)x_t\Bigr\|^p -\sum_{n=1}^\infty |c_n|^p & = |c|^p\Biggl[\Bigl\|y +\sum_{n=1}^\infty \frac{c_n}{c} \sum_{t\in \Lambda_n(\tau)}\mathbb{P}_{\xi,\infty}(t)x_t\Bigr\|^p -\sum_{n=1}^\infty \Bigl|\frac{c_n}{c}\Bigr|^p\Biggr] \\ & \leqslant |c|^p\lambda^p. \end{align*}   From this it easily follows that $|cy|=|c||y|$ for any scalar $c$ and $y\in X$. 

It is obvious from construction that $|\cdot|$ satisfies the triangle inequality.  Therefore $|\cdot|$ is an equivalent norm on $X$.  

Next note that for any $y\in X$ and $\lambda>0$, the set $\mathcal{F}_{y, \lambda}$ is closed, and therefore the $\mathcal{F}_{y,\lambda}$ game is determined. In order to see this, note that for some $\alpha\tau.\gamma=(\zeta_t, x_t)_{t<\tau}\in [\Gamma_{\xi,\infty}].\mathfrak{s}\setminus \mathcal{F}_{y,\lambda}$, there exists $(c_n)_{n=1}^\infty\in c_{00}$ such that \[\frac{1}{C^p}\Bigl\|y+\sum_{n=1}^\infty c_n\sum_{t\in \Lambda_n(\tau)} \mathbb{P}_{\xi,\infty}(t)x_t\Bigr\|^p - \sum_{n=1}^\infty |c_n|^p > \lambda^p.\]  Choose $m\in\nn$ such that $c_n=0$ for all $n>0$ and fix $r\in\nn$ such that $\cup_{k=1}^m \Lambda_n(\tau)=\{(\zeta_i)_{i=1}^j: j\leqslant r\}$.  Then if $\beta=\tau'.\gamma'=(\nu_t,y_t)_{t<\tau'} \in [\Gamma_{\xi,\infty}].\mathfrak{s}$ is such that  $\beta|r=\alpha|r$, it follows that $\nu_t=\zeta_t$ and $y_t=x_t$ for all $t\in \cup_{k=1}^r \Lambda_n(\tau')$, and \begin{align*} \frac{1}{C^p}\Bigl\|y+\sum_{n=1}^\infty c_n\sum_{t\in \Lambda_n(\tau')} \mathbb{P}_{\xi,\infty}(t)y_t\Bigr\|^p - \sum_{n=1}^\infty |c_n|^p & = \frac{1}{C^p}\Bigl\|y+\sum_{n=1}^m c_n\sum_{t\in \Lambda_n(\tau')} \mathbb{P}_{\xi,\infty}(t)y_t\Bigr\|^p - \sum_{n=1}^\infty |c_n|^p  \\ & = \frac{1}{C^p}\Bigl\|y+\sum_{n=1}^m c_n\sum_{t\in \Lambda_n(\tau)} \mathbb{P}_{\xi,\infty}(t)x_t\Bigr\|^p - \sum_{n=1}^\infty |c_n|^p \\ & = \frac{1}{C^p}\Bigl\|y+\sum_{n=1}^\infty c_n\sum_{t\in \Lambda_n(\tau)} \mathbb{P}_{\xi,\infty}(t)x_t\Bigr\|^p - \sum_{n=1}^\infty |c_n|^p \\ & > \lambda^p, \end{align*} and $\beta\in [\Gamma_{\xi,\infty}].\mathfrak{s}\setminus \mathcal{F}_{y,\lambda}$.

  Let $B=\{y\in X: [y]<1\}$ and note that $B_X^{|\cdot|}$ is the closed, convex hull of $B$.  Fix  $y\in B$,  $\sigma\geqslant 0$, and any  collection $(x_a)_{a\in \Gamma_{\xi,\infty}}\subset B_X^{\|\cdot\|}$ normally weakly null in  $(X, \|\cdot\|)$. Fix a real number $\mu$ such that  \[\inf_{a\in MAX(\Gamma_{\xi,\infty})} \Bigl|y+\sigma\sum_{b\leqslant a} \mathbb{P}_{\xi,\infty}(b) x_b\Bigr|^p > 1 + \mu.\]    Then \begin{align*} \inf_{a\in MAX(\Gamma_{\xi,1})} \Bigl[y+\sigma\sum_{b\leqslant a} \mathbb{P}_{\xi,\infty}(b) x_b\Bigr]^p  & \geqslant \inf_{a\in MAX(\Gamma_{\xi,1})} \Bigl|y+\sigma\sum_{b\leqslant a} \mathbb{P}_{\xi,\infty}(b) x_b\Bigr|^p  \\ & > 1 + \mu. \end{align*} From this and the previous paragraph, it follows that for each $a\in MAX(\Gamma_{\xi,1}.D)$, Player $V$ has a winning strategy in the $\mathcal{F}_{y+\sum_{b\leqslant a}\mathbb{P}_{\xi,\infty}(b)x_b, (1+\mu)^{1/p}}$ game.    By Lemma \ref{tree}, for each $a\in MAX(\Gamma_{\xi,1}.D)$, there exists a collection $(x^a_b)_{b\in \Gamma_{\xi,\infty}.D}\subset B_X$ normally weakly null in $(X, \|\cdot\|)$ such that for each $\beta=(\zeta_b, u_b)_{b<\beta}\in [\Gamma_{\xi,\infty}].D$, $(\zeta_b, x^a_b)_{b<\beta}\in [\Gamma_{\xi,\infty}].\mathfrak{s}\setminus \mathcal{F}_{y+\sum_{b\leqslant a}\mathbb{P}_{\xi,\infty}x_b, (1+\mu)^{1/p}}$.  This means that for each $\beta=(\zeta_b, u_b)_{b<\beta}\in [\Gamma_{\xi,\infty}].D$, there exists $(c_n)_{n=1}^\infty\in c_{00}$ such that \begin{align*} \frac{1}{C^p}\Bigl\|y+\sigma \sum_{b\leqslant a} \mathbb{P}_{\xi,\infty}(b)x_b+\sum_{n=1}^\infty c_n\sum_{t\in \Lambda_n(\beta)} \mathbb{P}_{\xi,\infty}(b)x_b^a\Bigr\|^p - \sum_{n=1}^\infty |c_n|^p & > 1+\mu. \end{align*}  Extend the collection $(x_b)_{b\in \Gamma_{\xi,1}.D}$ to a collection $(x_b)_{b\in \Gamma_{\xi,\infty}.D}$ by letting $x^a_b=x_{a\smallfrown (\omega^\xi + beta)}$. Here we recall that  $\Lambda_{\xi,\infty,1}.D=\Gamma_{\xi,1}.D$ and for each $a\in MAX(\Gamma_{\xi,1}.D)$, $b\mapsto a\smallfrown (\omega^\xi+b)$ is a bijection from $\Gamma_{\xi,\infty}.D$ onto $\{b\in \Gamma_{\xi,\infty}.D: a<b\}$ which identifies $\Lambda_{\xi,\infty,n}.D$ with $\{b\in \Lambda_{\xi,\infty, n+1}.D: a<b\}$ and such that $\mathbb{P}_{\xi,\infty}(b)=\mathbb{P}_{xi,\infty}(a\smallfrown (\omega^\xi +b))$ for all $b\in \Gamma_{\xi,\infty}.D$.   From this it follows that for each $\beta=(\zeta_b, u_b)_{b<\beta}\in [\Gamma_{\xi,\infty}].D$, if $a<\beta$ is such that $a\in MAX(\Gamma_{\xi,\infty}.D)$ and if $a\smallfrown (\omega^\xi+\alpha)=\beta$, then there exists $(c_n)_{n=2}^\infty\in c_{00}$ such that, with $c_1=\sigma$,  \begin{align*}\Bigl\|y+\sum_{n=1}^\infty c_n\sum_{b\in \Lambda_n(\beta)} &\mathbb{P}_{\xi,\infty}(b)x_b\Bigr\|^p -\sum_{n=2}^\infty |c_n|^p \\ &  = \frac{1}{C^p}\Bigl\|y+\sigma\sum_{b\in \Lambda_1(\beta)}\mathbb{P}_{\xi,\infty}(b)x_b + \sum_{n=2}^\infty c_n \sum_{b\in \Lambda_n(b)} \mathbb{P}_{\xi,\infty}(b)x_b\Bigr\|^p -\sum_{n=2}^\infty |c_n|^p\\ & = \frac{1}{C^p}\Bigl\|y+\sigma \sum_{b\leqslant a} \mathbb{P}_{\xi,\infty}(b)x_b+\sum_{n=2}^\infty c_n\sum_{t\in \Lambda_n(\alpha)} \mathbb{P}_{\xi,\infty}(b)x_b^a\Bigr\|^p - \sum_{n=2}^\infty |c_n|^p \\ & > 1+\mu  \end{align*}  But since $y\in B$, $[y]<1$, so there exists a winning strategy for Player $S$ in the $\mathcal{F}_{y,1}$ game. If $\beta=(\zeta_b, u_b)_{b<\beta}$ is chosen according to such a winning strategy, it follows that for any $(c_n)_{n=2}^\infty\in c_{00}$, with $c_1=\sigma$,  \[\Bigl\|y+\sum_{n=1}^\infty c_n\sum_{b\in \Lambda_n(\beta)} \mathbb{P}_{\xi,\infty}(b)x_b\Bigr\|^p -\sigma^p-\sum_{n=2}^\infty |c_n|^p \leqslant 1.\]   Combining this inequality with the previous inequality yields that $1+\mu\leqslant 1+\sigma^p$, and $\mu\leqslant \sigma^p$.   By Lemma \ref{13}, $(X, |\cdot|)$ is $\xi$-$p$-AUS.

\end{proof}

\begin{theorem}The following are equivalent. \begin{enumerate}[(i)]\item $X$ admits an equivalent $\xi$-AUF norm. 

\item There exists $C>0$ such that Player $S$ has a winning strategy in the $(\mathcal{E}^\infty_C, \mathfrak{c}, \Gamma_{\xi,\infty})$ game.  

\item There exists $C>0$ such that Player $S$ has a winning strategy in the $(\mathcal{E}^\infty_C, \mathcal{F}, \Gamma_{\xi,\infty})$ game.

\item There exists $C>0$ such that Player $S$ has a winning strategy in the $(\mathcal{E}^\infty_C, \mathfrak{s}, \Gamma_{\xi,\infty})$ game.
\end{enumerate}
\label{renorm2}
\end{theorem}

\begin{proof}$(i)\Rightarrow (ii)$ This is similar to the implication $(i)\Rightarrow (ii)$ of Theorem \ref{renorm1}.   More precisely, we select a constant $c$ from Proposition \ref{chuff2}, $C>c$, and positive numbers $(\ee_n)_{n=2}^\infty$ such that $\sum_{n=2}^\infty \ee_n <C-c$. Then  in the recursive construction of the $a_n$, the sequences are chosen according to so that for each \[y\in G_\varnothing=\Bigl\{\sum_{m=1}^n c_m\sum_{a_{m-1}<a\leqslant a_m} \mathbb{P}_{\xi,\infty}(a)x_a: (c_m)_{m=1}^n\in B_{\ell_\infty^n}, (x_a)_{a\leqslant a_n}\in \prod_{a\leqslant a_n}G_a\Bigr\},\] each $\sigma$ with $\sigma\leqslant 1$, and each $(x_a)_{a_n<a\leqslant a_{n+1}}\in \prod_{a_n<a\leqslant a_{n+1}}G_a$, \[\Bigl\|y+\sigma\sum_{a_n<a\leqslant a_{n+1}}\mathbb{P}_{\xi,\infty}(a)x_a\Bigr\|\leqslant \max\{\|y\|, c |\sigma|\}+\ee_{n+1}.\] After completing the recursive construction, we define $\alpha$ and $\gamma$ as in the proof of $(i)\Rightarrow (ii)$ of Theorem \ref{renorm1} and compute for any $(x_a)_{a<\alpha}\in \prod\gamma$ and $(c_n)_{n=1}^\infty\in B_{c_0}$ such that $c_n=0$ for all $n>m$ that, with $I_n=\{a\in\Gamma_{\xi,\infty}.D: a_{n-1}<a\leqslant a_n\}$, \begin{align*} \Bigl\|\sum_{n=1}^\infty c_n \sum_{a\in I_n} \mathbb{P}_{\xi,\infty}(a)x_a\Bigr\| & =  \Bigl\|\sum_{n=1}^m c_n \sum_{a\in I_n} \mathbb{P}_{\xi,\infty}(a)x_a\Bigr\| \\ & = \Bigl\|\sum_{n=1}^{m-1} c_n \sum_{a\in I_n} \mathbb{P}_{\xi,\infty}(a)x_a + c_m \sum_{b\in I_m} \mathbb{P}_{\xi,\infty}(b) x_b\Bigr\| \\ & \leqslant \max\Bigl\{\Bigl\|\sum_{n=1}^{m-1} c_n \sum_{a\in I_n} \mathbb{P}_{\xi,\infty}(a)x_a\Bigr\|,  c |c_m|\Bigr\} + \ee_m \\ & \leqslant \max\Bigl\{\max\Bigl\{\Bigl\|\sum_{n=1}^{m-2} \sum_{a\in I_n} \mathbb{P}_{\xi,\infty}(a)x_a\Bigr\|, c |c_{m-1}|\Bigr\}+\ee_{m-1}, c|c_m|\Bigr\} +\ee_m \\ & \leqslant \max\Bigl\{\Bigl\|\sum_{a\in I_n} \mathbb{P}_{\xi,\infty}(a)x_a\Bigr\|, c|c_{m-1}|, c|c_m|\Bigr\}+\ee_{m-1}+\ee_m \\ & \leqslant \ldots \\ & \leqslant c\max_{1\leqslant n\leqslant m}|c_m| + \sum_{n=1}^\infty \ee_m <C. \end{align*}

$(ii)\Rightarrow (iii)\Rightarrow (iv)$ These are clear, since any winning strategy for Player $S$ in the $(\mathcal{E}^\infty_C, \mathfrak{c}, \Gamma_{\xi,\infty})$ game is a winning strategy for Player $S$ in the $(\mathcal{E}^\infty_C, \mathcal{F}, \Gamma_{\xi,\infty})$ game, and any winning strategy for Player $(\mathcal{E}^\infty_C, \mathcal{F}, \Gamma_{\xi,\infty})$ game is a winning strategy for Player $S$ in the $(\mathcal{E}^\infty_C, \mathfrak{s}, \Gamma_{\xi,\infty})$ game. 

$(iv)\Rightarrow (i)$ As in the proof of Theorem \ref{renorm1}$(iv)\Rightarrow (i)$, we refer to games simply by their target sets.   Assume that $C>0$ is such that Player $S$ has a winning strategy in the $\mathcal{E}^\infty_C$ game.  Let $\chi_0$ be such a strategy.  For $y\in X$ and $\lambda>0$, let $\mathcal{G}_{y,\lambda}$ be the set of all $\tau.\gamma=(\zeta_t,x_t)_{t<\tau}$ such that  \[ \sup_m \Bigl\|y+\sum_{n=1}^m\sum_{t\in \Lambda_n(\tau)} \mathbb{P}_{\xi,\infty}(t)x_t \Bigr\|\leqslant \lambda.\]  For $y\in X$, let $g(y)$ be the infimum of $\lambda>0$ such that Player $S$ has a winning strategy in the $\mathcal{G}_{y, \lambda}$ game.  It is obvious from the triangle inequality that for any $y\in X$, the strategy $\chi_0$ is a winning strategy for Player $S$ in the $\mathcal{G}_{y, \|y\|+C}$ game.  Therefore $g(y)\leqslant \|y\|+C$.   Obviously $g(y)\geqslant \|y\|$ for any $y\in X$, which can be seen by considering the strategy for Player $V$ consisting of choosing the zero vector on each turn.    Let $G=\{y\in X: g(y)<1+C\}$. By the preceding remarks, \[\text{int}(B_X)\subset G\subset (1+C)B_X.\]  Let $|\cdot|$ be the Minkowski functional of the closed, convex hull of $G$.  By the preceding remarks, $|\cdot|$ is an equivalent norm on $X$.  We will show that $(X, |\cdot|)$ is $\xi$-AUF.  Fix $y\in G$ and a collection $(x_a)_{a\in \Gamma_{\xi,1}.D}\subset B_X^{\|\cdot\|}$ which is normally weakly null in $(X, \|\cdot\|)$.   Seeking a contradiction, assume that \[\inf_{a\in MAX(\Gamma_{\xi,1}.D)} \Bigl|y+\sum_{b\leqslant a}\mathbb{P}_{\xi,1}(b) x_b\Bigr| > 1.\]  Since $G\subset B_X^{|\cdot|}$, it follows that for each $a\in MAX(\Gamma_{\xi,1}.D)$, $y+\sum_{b\leqslant a}\mathbb{P}_{\xi,1}(b)x_b\notin G$, which means $g\Bigl(y+\sum_{b\leqslant a}\mathbb{P}_{\xi,1}(b)x_b\Bigr)\geqslant 1+C$.  Therefore \[\inf_{a\in MAX(\Gamma_{\xi,1}.D)} g\Bigl(y+\sum_{b\leqslant a}\mathbb{P}_{\xi,1}(b) x_b\Bigr)\geqslant 1+C.\]   Since $y\in G$, $g(y)<1+C$. Fix $g(y)< \phi<1+C$. The remainder of the proof is similar to Theorem \ref{renorm1}$(iv)\Rightarrow (i)$. For each $a\in MAX(\Gamma_{\xi,1}.D)$, we fix a collection $(x^a_b)_{b\in \Gamma_{\xi,\infty}.D}$ and then use these to extend the collection $(x_a)_{a\in \Gamma_{\xi,\infty}.D}$. Moreover, this tree is constructed from the collections $(x^a_b)_{b\in \Gamma_{\xi,\infty}.D}$ such that for any $\alpha\in [\Gamma_{\xi,\infty}].D$, if $a<\alpha$ is such that $a\in MAX(\Lambda_{\xi,\infty,1}.D)=MAX(\Gamma_{\xi,1}.D)$,   \[\sup_m \Bigl\|y+\sum_{n=1}^m \sum_{b\in \Lambda_n(\alpha)} \mathbb{P}_{\xi,\infty}(b)x_b \Bigr\|\geqslant  \sup_m\Bigl\|y+\sum_{b\leqslant a}\mathbb{P}_{\xi,1}(b)x_b + \sum_{n=1}^m \sum_{b\in \Lambda_n(\beta)} \mathbb{P}_{\xi,\infty}(b)x^a_b\Bigr\|>\phi,\] where $\alpha= a\smallfrown (\omega^\xi+\beta)$.   The existence of such a collection contradicts the fact that $g(y)<\phi$, according to Lemma \ref{tree}.   This shows that for any $y\in G$ and any collection $(x_a)_{a\in \Gamma_{\xi,1}.D}\subset B_X$ which is normally weakly null in $(X, \|\cdot\|)$,  \[\inf_{a\in MAX(\Gamma_{\xi,1}.D)} \Bigl|y+\sum_{b\leqslant a} \mathbb{P}_{\xi,1}(b)x_b\Bigr|\leqslant 1.\]    Therefore $(X, |\cdot|)$ is $\xi$-AUF by Lemma \ref{13}.

\end{proof}

\begin{rem}\upshape In \cite{GKL}, Godefroy, Kalton, and Lancien showed that AUF-renormability is a Lipschitz invariant. The notion of $\xi$-AUF-renormability was developed to be a candidate for higher ordinal versions of their theorem. However, there remain obstructions to the proof of the analogous theorems, namely the need for a strengthening of the Gorelik principle which can be applied multiple times during games of the type studied in this work, and which approximately preserves the values of functionals chosen during the game. 

\end{rem}

\section{Games on $C(K)$ spaces}

In this section, we discuss projective tensor products.  We recall that for Banach spaces $X,Y$, the projective tensor norm defined on $X\otimes Y$ is defined by \[\Bigl\|\sum_{i=1}^n x_i\otimes y_i\Bigr\|=\inf\Bigl\{\sum_{i=1}^m \|u_i\|\|v_i\|: \sum_{i=1}^n x_i\otimes y_i=\sum_{i=1}^m u_i\otimes v_i\Bigr\}.\]  The completion of $X\otimes Y$ with respect to this norm is denoted by $X\widehat{\otimes}_\pi Y$.  We state now the two main properties of the projective tensor product which we will need in the sequel.  First, $B_{X\widehat{\otimes}_\pi Y}$ is the closed, convex hull of $\{x\otimes y: x\in B_X, y\in B_Y\}$. The second fact is that if $S:X\to X$ and $T:Y\to Y$ are operators, then there is a bounded, linear operator from  $S\otimes T:X\widehat{\otimes}_\pi Y\to X\widehat{\otimes}_\pi Y$ such that $(S\otimes T)(x\otimes y)=Sx\otimes Ty$ for all $x\in X$ and $y\in Y$, and $\|S\otimes T\|=\|S\|\|T\|$.

For a compact, Hausdorff space $K$, and a subset $M$ of $K$, we let $\text{iso}(M)$ denote the set of relatively isolated points in $M$. For $M\subset K$, we define the \emph{Cantor-Bendixson derivative of} $M$ by $M'=M\setminus \text{iso}(M)$.  We note that $M'$ is closed in the relative topology of $M$. Therefore if $M$ is closed in $K$, so is $M'$.  We define the transfinite Cantor-Bendixson derivatives by \[M^0=M,\] \[M^{\xi+1}=(M^\xi)',\] and if $\xi$ is a limit ordinal, \[M^\xi=\bigcap_{\zeta<\xi}M^\zeta.\]   We say that $K$ is \emph{scattered} if any non-empty subset of $K$ has an isolated point. It is obvious this is equivalent to the condition that there exists an ordinal $\xi$ such that $K^\xi=\varnothing$. If $K$ is scattered, we let $CB(K)$ denote the minimum ordinal $\nu$ such that $K^\nu=\varnothing$.   The value $CB(K)$ is the \emph{Cantor-Bendixson index} of $K$.    If $K$ is compact, then by our preceding remarks, $K^\xi$ is either compact or empty for each ordinal $\xi$.  From this it follows that for a scattered,  compact, Hausdorff topological space, $CB(K)$ must be a successor ordinal. Moreover, $CB(K)=1$ if and only if $K$ is finite. 

\begin{example}\upshape For any ordinal $\xi$, it is easy to see that if ordinal intervals are endowed with their order topology, for $\zeta\leqslant \xi$, \[[0, \omega^xi]^\zeta = \{\omega^\xi\}\cup \{\omega^{\ee_1}+\ldots +\omega^{\ee_n}: \ee_1\geqslant \ldots \geqslant \ee_n\geqslant \zeta\}.\]  From this it follows that $[0,\omega^\xi]^\xi=\{\omega^\xi\}$ and $CB([0,\omega^\xi])=\xi+1$.

\end{example}

Let us isolate the following standard facts about the Cantor-Bendixson index. 

\begin{proposition}   Let $K$ be compact, Hausdorff space and let $\xi, \zeta$ be ordinals.   \begin{enumerate}[(i)]\item It holds that $(K^\xi)^\zeta=K^{\xi+\zeta}$. \item If $CB(K)=\xi+\zeta$, then $CB(K^\xi)=\zeta$. \item If $n$ is a positive integer and $F\subset K^{\omega^\xi(n-1)}\setminus K^{\omega^\xi n}$, then $CB(F)\leqslant \omega^\xi$.  \item For $T_0, \ldots, T_m\subset K$, $\Bigl(\cup_{j=0}^m T_j\Bigr)^\xi=\cup_{j=0}^m T_m^\xi$. \item For $T_0, \ldots, T_m\subset K$, $CB\Bigl(\cup_{j=0}^m T_j\Bigr)=\max_{0\leqslant j\leqslant m} CB(T_j)$.  \end{enumerate}

\label{cbf}
\end{proposition}

\begin{proof} Item $(i)$ follows from an easy induction argument and item $(ii)$ follows from $(i)$. For $(iii)$, $F^{\omega^\xi}\subset F\cap  (K^{\omega^\xi(n-1)})^{\omega^\xi}=F\cap K^{\omega^\xi n}=\varnothing$. 

Item $(iv)$ is an easy induction, and item $(v)$ follows from $(iv)$. 

\end{proof}

We next recall the following formulation of Grothendieck's inequality from \cite[Theorem $5.5$, page 55]{Pisier}. In what follows, $k_G$ is Grothendieck's constant and a scalar-valued, bounded, bilinear form $\varphi:C(K)\times C(L)\to \mathbb{K}$ is endowed with the norm \[\|\varphi\|=\sup\{|\varphi(f,g)|: f\in B_{C(K)}, g\in B_{C(L)}.\]

\begin{theorem} Let $K,L$ be compact, Hausdorff sets.  For any bounded, bilinear form $\varphi:C(K)\times C(L)\to \mathbb{K}$, there exist Borel probability measures $\mu, \nu$ on $K,L$, respectively, such that for any $f\in C(K)$ and $g\in C(L)$, \[|\varphi(f,g)|\leqslant k_G \|\varphi\|\Bigl(\int_K |f|^2d\mu\Bigr)^{1/2}\Bigl(\int_L |g|^2 d\nu\Bigr)^{1/2}.\]
\label{pi}
\end{theorem}

For the following proof, we recall the definition of the $q$-\emph{weakly summing} norms.  For a Banach space $X$,  a sequence $(x_n)_{n=1}^\infty \subset X$, and $1\leqslant q<\infty$, we define \[\|(x_n)_{n=1}^\infty\|_q^w = \sup \{\|(x^*(x_n))_{n=1}^\infty\|_{\ell_q}: x^*\in B_{X^*}.\]   For convenience, we will assume the value $\|(x_n)_{n=1}^\infty\|_q^w$ to be defined for any sequence in $X$, even if the value is finite.  For $(x_n)_{n=1}^\infty\subset X$ such that $\|(x_n)_{n=1}^\infty\|_q^w<\infty$, we refer to the value $\|(x_n)_{n=1}^\infty\|_q^w$ as the $q$-\emph{weakly summing norm} of $(x_n)_{n=1}^\infty$, and we say $(x_n)_{n=1}^\infty$ is $q$-\emph{weakly summing}.  

\begin{corollary} Let $K$ be  a compact, Hausdorff space.   Assume $C>0$, $0=r_0<r_1<\ldots$ are integers,  $(w_j)_{j=1}^\infty$ is a sequence of positive numbers such that $1=\sum_{j=r_{n-1}+1}^{r_n} w_j=1$ for each $n\in\nn$, and $(F_j)_{j=1}^\infty \subset B_{C(K)}$ is a sequence of sets such that for any $m\in\nn$ and any $(f_j)_{j=1}^\infty \in \prod_{j=1}^\infty F_j$, $\Bigl\|\sum_{n=1}^m \sum_{j=r_{n-1}+1}^{r_n} w_j|f_j|\Bigr\| \leqslant C^{1/2}$. 

\begin{enumerate}[(i)]\item For any compact, Hausdorff space $L$, any $(f_j)_{j=1}^\infty\in \prod_{j=1}^\infty F_j$, and any $(g_j)_{j=1}^\infty \subset B_{C(L)}$, $\Bigl\|\Bigl(\sum_{j=r_{n-1}+1}^{r_n} w_jf_j\otimes g_j\Bigr)_{n=1}^\infty\Bigr\|_2^w \leqslant k_GC$. 

\item If $(z_j)_{j=1}^\infty\subset C(K)\widehat{\otimes}_\pi C(L)$ is such that for each $j\in \nn$, $z_j\in \text{co}\{f\otimes g: f\in F_j, g\in B_{C(L)}\}$, then $\Bigl\|\Bigl(\sum_{j=r_{n-1}+1}^{r_n} w_jz_j\Bigr)_{n=1}^\infty\Bigr\|_2^w \leqslant k_GC^{1/2}$. \end{enumerate}

\label{kane}
\end{corollary}

\begin{proof}$(i)$  Fix a compact, Hausdorff space $L$, $(f_j)_{j=1}^\infty\in \prod_{j=1}^\infty F_j$, and $(g_j)_{j=1}^\infty\in B_{C(L)}$.  To finish $(i)$, it is sufficient to show that for any $m\in\nn$ and scalars $(c_n)_{n=1}^m$ such that $\sum_{n=1}^m |c_n|^2=1$, it holds that $\Bigl\|\sum_{n=1}^m c_n\sum_{j=r_{n-1}+1}^{r_n} w_j f_j\otimes g_j\Bigr\|\leqslant k_GC$.  To that end, fix $m\in\nn$ and scalars $(c_n)_{n=1}^m$ such that $\sum_{n=1}^m |c_n|^2=1$.    Recall that $(C(K)\widehat{\otimes}_\pi C(L))^*$ is the space of all bounded, bilinear forms on $C(K)\times C(L)$. Fix a bounded, bilinear form $\varphi$ on $C(K)\times C(L)$ such that $\|\varphi\|=1$.  By Theorem \ref{pi}, there exist Borel probability measures $\mu, \nu$ on $K,L$, respectively, such that for any $f\in C(K)$ and $g\in B(L)$, $|\varphi(f,g)|\leqslant k_G\Bigl(\int_K |f|^2d\mu\Bigr)^{1/2}\Bigl(\int_L |g|^2 d\nu\Bigr)^{1/2}$.   We note that since $1=\sum_{j=r_{n-1}+1}^{r_n} w_j$ for each $n\in\nn$, it follows from Jensen's inequality that for each $n\in\nn$ together with the fact that $\|f_j\|\leqslant 1$ for all $j\in\nn$ that \begin{align*} \sum_{j=r_{n-1}+1}^{r_n} w_j\Bigl(\int_K |f_j|^2 d\mu\Bigr)^{1/2} & \leqslant \Bigl(\sum_{j=r_{n-1}+1}^{r_n} w_j\int_K |f_j|^2 d\mu\Bigr)^{1/2}  \leqslant \Bigl(\sum_{j=r_{n-1}+1}^{r_n} w_j\int_K |f_j| d\mu\Bigr)^{1/2}. \end{align*}

  Therefore \begin{align*} \Bigl|\Bigl\langle \varphi, \sum_{n=1}^m c_n\sum_{j=r_{n-1}+1}^{r_n} w_j f_j\otimes g_j\Bigr\rangle\Bigr| & = \Bigl|\sum_{n=1}^nc_n\sum_{j=r_{n-1}+1}^{r_n} w_j \varphi(f_j, g_j) \Bigr| \\ & \leqslant \sum_{n=1}^m |c_n|\sum_{j=r_{n-1}+1}^{r_n} w_j|\varphi(f_j, g_j)| \\ &  \leqslant k_G\sum_{n=1}^m |c_n|\sum_{j=r_{n-1}+1}^{r_n} w_j\Bigl(\int_K |f_j|^2d\mu\Bigr)^{1/2}\Bigl(\int_L |g_j|^2d\nu\Bigr)^{1/2} \\ & \leqslant  k_G\sum_{n=1}^m |c_n|\sum_{j=r_{n-1}+1}^{r_n} w_j\Bigl(\int_K |f_j|^2d\mu\Bigr)^{1/2}\|g_j\| \\ & \leqslant  k_G\sum_{n=1}^m |c_n|\sum_{j=r_{n-1}+1}^{r_n} w_j\Bigl(\int_K |f_j|^2d\mu\Bigr)^{1/2} \\ & \leqslant k_G\sum_{n=1}^m |c_n|\Bigl(\sum_{j=r_{n-1}+1}^{r_n} w_j\int_K |f_j| d\mu\Bigr)^{1/2} \\ & = k_G\sum_{n=1}^m |c_n|\Bigl(\int_K\sum_{j=r_{n-1}+1}^{r_n} w_j |f_j| d\mu\Bigr)^{1/2}\\ & \leqslant k_G\Bigl(\sum_{n=1}^m |c_n|^2\Bigr)^{1/2}\Bigl(\sum_{n=1}^m\int_K \sum_{j=r_{n-1}+1}^{r_n} w_j|f_j| d\mu\Bigr)^{1/2} \\ & = k_G\Bigl(\int_K \sum_{n=1}^m \sum_{j=r_{j-1}+1}^{r_n} w_j|f_j|d\mu \Bigr)^{1/2} \\ & \leqslant k_G \Bigl\|\sum_{n=1}^m \sum_{j=r_{n-1}+1}^{r_n} w_j |f_j|\Bigr\|^{1/2} \leqslant k_GC^{1/2}.  \end{align*}

$(ii)$ Let $(z_j)_{j=1}^\infty$ be as in the statement of $(ii)$.  Fix $m\in\nn$ and $(c_n)_{n=1}^m$ be such that $\sum_{n=1}^m |c_n|=1$. Note that since $z_j\in \text{co}\{f\otimes g: f\in F_j, g\in B_{C(L)}\}$, it follows that \begin{align*}\sum_{n=1}^m c_n\sum_{j=r_{n-1}+1}^{r_n}  w_jz_j & \in \text{co}\Bigl\{\sum_{n=1}^m c_n\sum_{j=r_{n-1}+1}^{r_n} c_nw_j f_j\otimes g_j: (f_j)_{j=1}^\infty\in \prod_{j=1}^\infty F_j, g_j\in B_{C(L)}\Bigr\} \\ & \subset kC^{1/2}B_{C(K)\widehat{\otimes}_\pi C(L)},\end{align*} where the last containment follows from $(i)$. Since $m$ and $(c_n)_{n=1}^m$ were arbitrary, we are done.

\end{proof}

The preceding result involves sequences in $C(K)$ sequence which are $1$-absolutely summing.  Our next goal is to show that we may always find such collections. In what follows, for a compact, Hausdorff space $K$ and a finite, non-empty subset $N$ of $K$, we let $\an(N)=\{f\in C(K): f|_N\equiv 0\}$.   We let $A_K$ denote the set of all finite, non-empty subsets of $K$. Given a tree $T$ and a collection $(F_b)_{b\in T.A_K}$ of subsets of $C(K)$, we say $(F_b)_{b\in T.A_K}$ is \emph{normally pointwise null} provided that for any $a=(\zeta_i, N_i)_{i=1}^n\in T.A_K$, $F_a\subset \an(N_n)$ (that is, if $f|_{N_n}\equiv 0$ for every $f\in F_a$). 

We next define the \emph{Grasberg norm} for a scattered, compact, Hausdorff space. If $K$ is compact, Hausdorff, scattered, then there exist a unique ordinal $\xi$ and positive integer $k$ such that $K^{\omega^\xi(k-1)}\neq \varnothing$ and $K^{\omega^\xi k}=\varnothing$.   We then define the equivalent norm $[\cdot]$ on $C(K)$ by \[ [f]=\max_{1\leqslant j\leqslant k} 2^j \|f|_{K^{\omega^\xi(j-1)}}\|.\] Of course, for any $f\in C(K)$, \[\|f\|\leqslant [f]\leqslant 2^{k-1}\|f\|.\]    The Grasberg norm is a $\xi$-AUF norm on $C(K)$, which will be shown as a consequence of the following result.  In what follows, when $C(K)$ is written without a specific reference to the norm, it will be understood that the norm in question is the usual norm. When we wish to refer to the Grasberg norm, we will make it explicit. 

For the following result, we also establish the following notation. For $h\in C(K)$ and $c>0$, we let \[M_j(h,c)=\Bigl\{\kappa \in K^{\omega^\xi(j-1)}: 2^j|h(\kappa)|\geqslant c\Bigr\}\] and \[M(h,c)=\bigcup_{j=0}^{k-1} M_j(h,c).\] For a finite subset $H$ of $C(K)$, we let $M_j(H,c)=\cup_{h\in H}M_j(h,c)$ and $M(H,c)=\cup_{h\in H}M(h,c)$. Of course, $M_j(h,c)$, $M_j(H,c)$, $M(h,c)$, $M(H,c)$ are compact for any $h\in C(K)$, $H\subset K$ finite, and any $c>0$. 

\begin{lemma} Let $K$ be scattered, compact, Hausdorff and let $\xi$ be an ordinal. Assume that $K^{\omega^\xi(k-1)}\neq \varnothing$ and $K^{\omega^\xi k}=\varnothing$.    In what follows, $[\cdot]$ denotes the Grasberg norm on $C(K)$.  \begin{enumerate}[(i)]\item If $h\in C(K)$, then for any $c>[h]$,  $CB(M(h,c)) \leqslant \omega^\xi$. \item If $H$ is a finite subset of $C(K)$ and $c>\max_{h\in H}[h]$, then $CB(M(H,c)) \leqslant \omega^\xi$.  \item If $M\subset K$ is compact and $CB(M)\leqslant \omega^\xi$, then for any normally pointwise null collection $(F_b)_{b\in \Gamma_{\xi,1}.A_K}\subset \mathfrak{f}_{C(K)}$ and any $\ee>0$, there exists $a\in MAX(\Gamma_{\xi,1}.A_K)$ such that for all $(f_b)_{b\leqslant a}\in \prod_{b\leqslant a}F_b$, \[\Bigl\|\Bigl(\sum_{b\leqslant a} \mathbb{P}_{\xi,1}(b) f_b\Bigr)\Bigr|_M\Bigr\|\leqslant \ee.\]   \item Assume $1\leqslant C<C'$, $H,F$ are finite subsets of $C(K)$ such that for each $h\in H$ and $f\in F$, $[h]\leqslant C$, $[f]\leqslant 1/2$, and $\|f|_{M(H, C')}\|\leqslant \frac{C'-C}{2^k}$.    Then for any $(h, f)\in H\times F$, $[h+f]\leqslant C'$.   \end{enumerate}

\label{so}
\end{lemma}

\begin{proof}$(i)$ Let $h,c$ be as in $(i)$.   Since for any $h\in C(K)$ and $c>0$, \[CB(M(h,c))=CB\Bigl(\bigcup_{j=0}^{k-1} M_j(h,c)\Bigr)=\max_{0\leqslant j<k} CB(M_j(h,c)),\] in order to prove $(i)$, it is sufficient to show that $CB(M_j(h,c))\leqslant \omega^\xi$ for each $0\leqslant j<k$. 

First consider the case $\xi=0$. In this case, $CB(M_j(h,c)) \leqslant \omega^0=1$ is equivalent to the condition that $M_j(h,c)$ is finite.    If $M_j(h,c)$ were infinite, then there would exist some accumulation point $\kappa$ of $M_j(h,c)\subset K^{\omega^0(j-1)}=K^{j-1}$. Since an accumulation point in $K^{j-1}$ cannot be isolated in $K^{j-1}$, it follows that $\kappa\in K^j$. If $j=k$, this is the necessary contradiction, since $K^k=\varnothing$, so assume $j<k$. Since $2^j|h(\kappa_0)|\geqslant c$ for all $\kappa_0\in M_j(h,c)$,  since $\kappa$ is an accumulation point of $M_j(h,c)$, and since $h$ is continuous, $2^j|h(\kappa)|\geqslant c$.  Then \[c> [h] \geqslant 2^j |h(\kappa)|\geqslant c,\] which is the necessary contradiction in the $j<k$ case. This concludes the $\xi=0$ case.

Now consider the $\xi>0$ case. In order to show that $CB(M_j(h,c))\leqslant \omega^\xi$, it is sufficient to show that there exists $\zeta<\omega^\xi$ such that $M_j(h,c)\cap K^{\omega^\xi(j-1)+\zeta}=\varnothing$. Indeed, if $M_j(h,c)\cap K^{\omega^\xi(j-1)+\zeta}=\varnothing$, then $M_j(h,c)\subset K^{\omega^\xi(j-1)}\setminus K^{\omega^\xi(j-1)+\zeta}$ and \[CB(M_j(h,c)) \leqslant CB(K^{\omega^\xi(j-1)}\setminus K^{\omega^\xi(j-1)+\zeta}) = \zeta<\omega^\xi.\]    In order to obtain a contradiction, assume that for every $\zeta<\omega^\xi$, $M_j(h,c)\cap K^{\omega^\xi(j-1)+\zeta}\neq \varnothing$.  Then since $(M_j(h,c)\cap K^{\omega^\xi(j-1)+\zeta})_{\zeta<\omega^\xi}$ is a decreasing chain of compact, non-empty sets, which therefore have the finite intersection property, it follows that \[M_j(h,c)\cap K^{\omega^\xi j}=M_j(h,c)\cap \bigcap_{\zeta<\omega^\xi} K^{\omega^\xi(j-1)+\zeta} = \bigcap_{\zeta<\omega^\xi} (M_j(h,c)\cap K^{\omega^\xi(j-1)+\zeta})\neq \varnothing.\] As in the previous paragraph, if $j=k$, we reach a contradiction by noting that $K^{\omega^\xi j}=K^{\omega^\xi k}=\varnothing$, and if $j<k$ we reach a contradiction by noting that if $\kappa\in M_j(h,c)\cap K^{\omega^\xi j}$, then \[c >[h] \geqslant 2^j|h(\kappa)| \geqslant c.\]

$(ii)$ By $(i)$, $CB(M(h,c))\leqslant \omega^\xi$ for each $h\in H$. Since $H$ is finite, \[CB(M(H,c)) = CB\Bigl(\bigcup_{h\in H}M(h,c)\Bigr)=\max_{h\in H} CB(M(h,c))\leqslant \omega^\xi.\]

$(iii)$ We also work by contradiction. If $\xi=0$, then $M$ is finite. Therefore if $(F_b)_{b\in \Gamma_{\xi,1}.A_K}= (F_b)_{b\in \Gamma_{0,1}.A_K}$ is normally pointwise null, we can choose $a=(0, M)\in MAX(\Gamma_{0,1}.A_K)$ and note that for any $f_a\in F_a$, by the definition of normal pointwise nullity, $f_a|_M\equiv 0$. Of course this implies that \[\Bigl\|\Bigl(\sum_{b\leqslant a} f_b\Bigr)\Bigr|_M\Bigr\| = \|f_a|_M\|=0.\]

Now consider the $\xi>0$ case. Assume that $(F_b)_{b\in \Gamma_{\xi,1}.A_K}\subset \mathfrak{f}_{C(K)}$ is normally pointwise null and for each $a\in MAX(\Gamma_{\xi,1}.A_K)$, there exist $\kappa_a\in M$ and $(f^a_b)_{b\leqslant a}\in \prod_{b\leqslant a}F_b$ such that \[\ee \leqslant \sum_{b\leqslant a} \mathbb{P}_{\xi,1}(b) |f^a_b|(\kappa_a).\]   We argue as in Lemma \ref{bost}. 

Let $$\Pi=\{(b,a)\in \Gamma_{\xi,1}.A_K\times \Gamma_{\xi,1}.A_K: b\leqslant a\in MAX(\Gamma_{\xi,1}.A_K)\}$$ Define $\varphi:\Pi\to \rr$ by $$\varphi(b,a)=|f^a_b|(\kappa_a).$$  Note that $\varphi$ maps into $[0,1]$, since $F_b\subset B_{C(K)}$ for all $b\in \Gamma_{\xi,1}.A_K$. 

By hypothesis, for each $a\in MAX(\Gamma_{\xi,1}.A_K)$, \begin{align*} \sum_{b\leqslant a} \mathbb{P}_{\xi,1}(b)\varphi(b,a) & = \sum_{b\leqslant a} \mathbb{P}_{\xi,1}(b)|f^a_b|(\kappa_a) \geqslant \ee. \end{align*}  

By \cite[Theorem $4.2$]{Causey.5}, there exist  functions $d:\Gamma_{\xi,1}.A_K\to \Gamma_{\xi,1}.A_K$ and $e:MAX(\Gamma_{\xi,1}.A_K)\to MAX(\Gamma_{\xi,1}.A_K)$ such that \begin{enumerate}[(i)]\item for each  $b,a\in \Gamma_{\xi,1}.A_K$ such that $b<a$, it follows that  $d(b)<d(a)$, \item for each $a\in MAX(\Gamma_{\xi,1}.A_K)$, $d(a)\leqslant e(a)$, \item if $b=(\zeta_i, N_i)_{i=1}^m$ and $d(b)=(\nu_i, P_i)_{i=1}^n$, then $P_n\subset N_m$, \item for each $(b,a)\in \Pi$, either $f(d(b),e(a)) \geqslant \ee-\ee/2=\ee/2$ or $$\sum_{b\leqslant e(a)} \mathbb{P}_{\xi,1}(b) f(b, e(a))<\ee.$$\end{enumerate} Above we showed that the inequality $\sum_{b\leqslant e(a)} \mathbb{P}_{\xi,1}(b)\varphi(b,e(a))<\ee$ in (iv) is not possible, so $|f^{e(a)}_{d(b)}|(\kappa_{e(a)})=\varphi(d(b), e(a))\geqslant \ee/2$ for all $(b,a)\in \Pi$.  

We next claim that for any $\eta<\omega^\xi$ and $a\in \Gamma_{\xi,1}^\eta.A_K$, there exist $\varpi_a\in M^\eta$ and $(h_b)_{b\leqslant a}\in \prod_{b\leqslant a}F_{d(b)}$ such that for all $b\leqslant a$, $|h_b|(\varpi_a)\geqslant \ee/2$.  In particular, this will imply that for each $\eta<\omega^\xi$, $M^\eta\neq \varnothing$. By compactness of $M$ and the finite intersection property, this will imply that $M^{\omega^\xi}=\bigcap_{\eta<\omega^\xi}M^\eta \neq \varnothing$, contradicting the fact that $CB(M)\leqslant \omega^\xi$.    We prove the claim from the first sentence of the paragraph by induction on $\eta$.  

Base case, $\eta=0$.   In this case, for $a\in \Gamma_{\xi,1}^0.A_K= \Gamma_{\xi,1}.A_K$ and $b\leqslant a$, let $h_b=f^{e(a)}_{d(b)}\in F_{d(b)}$ and $\varpi_a=\kappa_{e(a)}\in M=M^0$. It follows from the properties of $d$ and $e$ that $|h_b|(\varpi_a)=|f^{e(a)}_{d(b)}|(\kappa_{e(a)}) \geqslant \ee/2$. 

Limit case: Assume $\eta<\omega^\xi$ is a limit ordinal and the result holds for every $\upsilon<\eta$.  Fix $a\in \Gamma_{\xi,1}^\eta.A_K=\bigcap_{\upsilon<\eta}\Gamma_{\xi,1}^\upsilon.A_K$. By the inductive hypothesis, for each $\upsilon<\eta$, there exists $(\varpi_\upsilon, (h^\upsilon_b)_{b\leqslant a})\in M\times \prod_{b\leqslant a}F_{d(b)}$ such that for each $\upsilon<\eta$, $\varpi_\upsilon\in M^\upsilon$ and for each $b\leqslant a$, $|h^\upsilon_b|(\varpi_\upsilon) \geqslant \ee/2$. Since $M\times \prod_{b\leqslant a} F_{d(b)}$ is compact, \[\bigcap_{\upsilon<\eta} \overline{\{(\varpi_\gamma, (h^\gamma_b)_{b\leqslant a}): \gamma \geqslant \upsilon\}} \neq \varnothing.\] Then if $(\varpi, (h_b)_{b\leqslant a})\in \bigcap_{\upsilon<\eta} \overline{\{(\varpi_\gamma, (h^\gamma_b)_{b\leqslant a}): \gamma \geqslant \upsilon\}}$, it holds that  \[\varpi\in \bigcap_{\upsilon<\eta}M^\upsilon=M^\eta,\]  $(h_b)_{b\leqslant a}\in \prod_{b\leqslant a}F_{d(b)}$, and by continuity of the map $(\tau, (g_b)_{b\leqslant a})\mapsto (|g_b|(\tau))_{b\leqslant a}$ on $M\times \prod_{b\leqslant a}F_{d(b)}$, it follows that $|h_b|(\varpi)\geqslant \ee/2$ for each $b\leqslant a$. 

Successor case: Assume that the claim holds for some $\eta<\omega^\xi$. Fix $a=t.f\in \Gamma_{\xi,1}^{\eta+1}.A_K$. Note that since $t\in \Gamma_{\xi,1}^{\eta+1}$, there exists $\zeta$ such that $t\smallfrown (\zeta)\in \Gamma_{\xi,1}^\eta$. Then for each finite, non-empty subset $N$ of $K$, $a\smallfrown (\zeta, N)\in \Gamma_{\xi,1}^\eta.A_K$. Let $D$ denote the set of finite, non-empty subsets of $K$ and direct $D$ by inclusion. By the inductive hypothesis, for each $N\in D$, there exist $\varpi_N\in M^\eta$, $(h^N_b)_{b\leqslant a}\in \prod_{b\leqslant a}F_{d(b)}$, and $g_N\in F_{d(a\smallfrown (\zeta, N))}$ such that for each $b\leqslant a$, $|h^N_b|(\varpi_N)\geqslant \ee/2$, and such that $|g_N|(\varpi_N)\geqslant \ee/2$.  By compactness of $M^\eta\times \prod_{b\leqslant a}F_{d(b)}$, we can select $(\varpi, (h_b)_{b\leqslant a})\in M^\eta\times \prod_{b\leqslant a}F_{d(b)}$ which is the limit of a subnet of $((\varpi_N, (h^N_b)_{b\leqslant a}))_{N\in D}$.  As in the previous paragraph, we deduce that $|h_b|(\varpi)\geqslant \ee/2$ for each $b\leqslant a$.   We claim that $\varpi\in M^{\eta+1}$. Since $\varpi$ is the limit of a net in $M^\eta$, it follows that $\varpi\in M^\eta$. To show that $\eta\in M^{\eta+1}$, it is sufficient to show that $\varpi$ is not isolated in $M^\eta$. To that end, let $U$ be any open set in $K$ containing $\varpi$.   Since $(\varpi, (h_b)_{b\leqslant a})$ is the limit of a subset of $((\varpi_N, (h^N_b)_{b\leqslant a}))_{N\in D}$, there exists $N\in D$ such that $\{\varpi\}\subset N$ and $\varpi_N\in U$.    Let $d(a\smallfrown (\zeta, N))=(\zeta_i, E_i)_{i=1}^m$ and note that by property $(iii)$ above, $N\subset E_m$. Also, by the definition of normally pointwise null, it follows that for each $g\in F_{d(a\smallfrown (\zeta, N))}$ and each $\tau\in E_m$, $h(\tau)=0$. Since $\varpi\in N\subset E_m$, it follows that for each $g\in F_{d(a\smallfrown(\zeta, N))}$, $g(\varpi)=0$.  Recall that  $g_N \in F_{d(a\smallfrown (\zeta, N))}$ has the property that $|g_N|(\varpi_N) \geqslant \ee/2$. Since $g_N\in F_{d(a\smallfrown (\zeta, N)}$, $g_N(\varpi)=0$, from which it follows that $\varpi_N\neq \varpi$.  Therefore $\varpi\neq \varpi_N\in U\cap M^\eta$. Since $U$ was an arbitrary neighborhood of $\varpi$, it follows that $\varpi$ is not isolated in $M^\eta$, and $\varpi\in M^{\eta+1}$. This completes the proof of $(iii)$.

$(iv)$ Fix $h\in H$ and $f\in F$. Fix $0\leqslant j<k$ and $\kappa\in K^{\omega^\xi(j-1)}$.   If $\kappa\in M_j(H,C')\subset M(H,C')$, then $|f(\kappa)|\leqslant \frac{C'-C}{2^k}$ and \[2^{j-1}|h(\kappa)+f(\kappa)| \leqslant 2^{j-1}|h(\kappa)|+\frac{2^{j-1}}{2^k}(C'-C) \leqslant [h]+(C'-C)\leqslant C+C'-C=C'.\]   If $\kappa\in K^{\omega^\xi(j-1)}\setminus M_j(H,C')$, then $2^{j-1}|h(\kappa)|\leqslant C'/2$, so \begin{align*} 2^{j-1}|h(\kappa)+f(\kappa)| & \leqslant 2^{j-1}|h(\kappa)|+2^{j-1}|f(\kappa)| \leqslant C'/2+ [f] \leqslant C'/2+1/2\leqslant C'.\end{align*} This shows that $2^{j-1}\|(h+f)|_{K^{\omega^\xi(j-1)}} \leqslant C'$. Since this holds for $0\leqslant j<k$, $[h+f]\leqslant C'$. 

\end{proof}

We now define a game closely related to the games associated with $\xi$-AUF-renormability. For a compact, Hausdorff space $\Omega$, $C>0$,  and an ordinal $\xi$, Player $S$ chooses $\zeta_1$ such that $(\zeta_1)\in \Gamma_{\xi,\infty}$ and a finite, non-empty subset $N_1$ of $\Omega$ (equivalently, Player $S$ chooses $(\zeta_1, Z_1)\in \Gamma_{\xi,\infty}A_\Omega$, where $Z_1=\an(N_1)$), and Player $V$ chooses a finite subset $F_1$ of $\an(N_1)$. Player $S$ then chooses $\zeta_2$ such that $(\zeta_i)_{i=1}^2\in \Gamma_{\xi,\infty}$ and a finite, non-empty subset $N_2$ of $\Omega$, and Player $V$ chooses a finite, non-empty subset $F_2$ of $\an(N_2)$.   Play continues in this way until $\alpha=(\zeta_a, \an(N_a))_{a<\alpha}\in [\Gamma_{\xi,\infty}].A_\Omega$ and $(F_a)_{a<\alpha}\in \mathfrak{f}_{C(K)}$ have been chosen so that $F_a\subset \an(N_a)$ for all $a<\alpha$.    Player $S$ wins if for all $(f_a)_{a<\alpha}$ and $m\in\nn$, \[\sup_m\Bigl\|\sum_{n=1}^m \sup_{a\in \Lambda_n(\alpha)} \mathbb{P}_{\xi,\infty}(a) |f_a|\Bigr\|\leqslant C.\]    More formally, if $\Omega$ and $\xi$ are understood, we let $\mathcal{I}_C$ denote the space of all $\alpha=(\zeta_a, F_a)_{a<\alpha}\in [\Gamma_{\xi,\infty}].\mathfrak{f}_{C(\Omega)}$ such that for all $m\in\nn$ and $(f_a)_{a<\alpha}\in \prod_{a<\alpha}F_a$, \[\sup_m\Bigl\|\sum_{n=1}^m \sum_{a\in \Lambda_n(\alpha)} \mathbb{P}_{\xi,\infty}(a) |f_a|\Bigr\|\leqslant C.\]   Then the game above is the game $(\mathcal{I}_C, \mathfrak{f}_{C(\Omega)}, \Gamma_{\xi,\infty})$, where $D=A_\Omega$.

\begin{lemma} Assume $\xi$ is an ordinal, $r$ is a positive integer, and $\Omega$ is a scattered, compact, Hausdorff space such that $\Omega^{\omega^\xi (r-1)}\neq \varnothing$ and $\Omega^{\omega^\xi r}=\varnothing$. Then for any $C>2^r$, Player $S$ has a winning strategy in the $(\mathcal{I}_C, \mathfrak{f}_{C(\Omega)}, \Gamma_{\xi,\infty})$ game. 

\label{breathing}
\end{lemma}

\begin{proof} In the proof, let $[\cdot]$ denote the Grasberg norm on $\Omega$.    Note that \[ \|f\|\leqslant [f]\leqslant 2^{r-1}[f]\] for all $f\in C(\Omega)$.  Therefore for $f\in B_{C(\Omega)}$, $[|f|/2^r] \leqslant 1/2$. 

Fix $C>2^r$. It is easy to see that $\mathcal{I}_C$ is closed, so the game is determined. Assume Player $S$ does not have a winning strategy in the $(\mathcal{I}_C, \mathfrak{f}_{C(\Omega)}, \Gamma_{\xi,\infty})$ game.  Since the game is determined, Player $V$ has a winning strategy. By Proposition \ref{tree}, there exists a normally pointwise null collection $(F_a)_{a\in \Gamma_{\xi,\infty}.A_\Omega}$ such that for any $\alpha\in [\Gamma_{\xi,\infty}].A_\Omega$, \[\sup_m \max\Bigl\{\Bigl\|\sum_{n=1}^m \sum_{a\in \Lambda_n(\alpha)} \mathbb{P}_{\xi,\infty}(a)|f_a|\Bigr\| : (f_a)_{a<\alpha}\in \prod_{a<\alpha}F_a\Bigr\} >C.\]   

$1=C_0'<C_1'<\ldots $ with $\sup_m C'_m = C/2^r$. We will recursively choose $a_1<a_2<\ldots$ with $a_n\in MAX(\Lambda_{\xi,\infty,n}.A_\Omega)$.    

First we perform the base step. Fix $a_1\in MAX(\Lambda_{\xi,\infty,1}.A_\Omega)$ arbitrary. Let \[H_1 = \Bigl\{\sum_{a\leqslant a_1} \mathbb{P}_{\xi,\infty}|f_a|/2^r : (f_a)_{a\leqslant a_1}\in \prod_{a\leqslant a_1} F_a\Bigr\}\] and note that $[h] \leqslant 1/2 \leqslant 1=C_0'$ for all $h\in H_1$.    

Now assume that $a_1<\ldots <a_m$ have been chosen. Assume also that with \[H_m=\Bigl\{\sum_{n=1}^m \sum_{a\in \Lambda_n(a_m)} \mathbb{P}_{\xi,\infty}(a) |f_a|/2^r: (f_a)_{a\leqslant a_m}\in \prod_{a\leqslant a_m}F_a\Bigr\},\] it holds that $[h]\leqslant C_{m-1}'$ for all $h\in H_m$. By lemma \ref{so}$(ii)$, $CB(M(H_m, C_m'))\leqslant \omega^\xi$.  Using the canonical identification of $\Gamma_{\xi,1}.A_\Omega$ with $\{a\in \Lambda_{\xi,\infty, m+1}.A_\Omega: a_m<a\}$ together with Lemma \ref{so}$(iii)$, there exists $a_{m+1}\in MAX(\Lambda_{\xi,\infty, m+1}.A_\Omega)$ such that $a_m<a_{m+1}$ and for any $(f_a)_{a\in \Lambda_{m+1}(a_{m+1})}\in \prod_{a\in \Lambda_{m+1}(a_{m+1})} F_a$, \[\Bigl\|\Bigl(\sum_{a\in \Lambda_{m+1}(a_{m+1})} \mathbb{P}_{\xi,\infty}(a)|f_a| \Bigr)|_{M(H_m, C'_m)}\Bigr\|\leqslant C'_m-C'_{m-1}.\]       Define \[F=\Bigl\{\sum_{a\in \Lambda_{m+1}(a_{m+1})} \mathbb{P}_{\xi,\infty}(a)|f_a|/2^r: (f_a)_{a\in \Lambda_{m+1}(a_{m+1})}\in \prod_{a\in \Lambda_{m+1}(a_{m+1})} F_a\Bigr\}\] and \[H_{m+1}=\Bigl\{\sum_{n=1}^{m+1} \sum_{a\in \Lambda_n(a_m)} \mathbb{P}_{\xi,\infty}(a) |f_a|/2^r: (f_a)_{a\leqslant a_{m+1}}\in \prod_{a\leqslant a_{m+1}}F_a\Bigr\}=\{h+f: h\in H_m, f\in F\}.\]  By the assumptions on $H_m$, $[h]\leqslant C'_{m-1}$ for each $h\in H_m$. By the first line of the proof, $[f]\leqslant 1/2$ for each $f\in F$.   By our choice of $a_{m+1}$, $\|f|_{M(H_m, C'_m)}\| \leqslant \frac{C_m'-C_{m-1}'}{2^r}$ for each $f\in F$.   Then by Lemma \ref{so}$(iv)$, it follows that \[\max_{g\in H_{m+1}}=\max\{[h+f]: (h,f)\in H_m\times F\} \leqslant C'_m.\]  This completes the recursive process.  

If $\alpha\in [\Gamma_{\xi,\infty}].A_\Omega$ is the sequence which has $a_1, a_2, \ldots$ as initial segments, it follows from the recursive construction that  \begin{align*} \sup_m \max\Bigl\{&\Bigl\|\sum_{n=1}^m \sum_{a\in \Lambda_n(\alpha)} \mathbb{P}_{\xi,\infty}(a) |f_a|\Bigr\|: (f_a)_{a<\alpha}\in \prod_{a<\alpha}F_a\Bigr\} \\ & = 2^r\sup_m \max\Bigl\{\Bigl\|\sum_{n=1}^m \sum_{a\in \Lambda_n(\alpha)} \mathbb{P}_{\xi,\infty}(a) |f_a|/2^r\Bigr\|: (f_a)_{a<\alpha}\in \prod_{a<\alpha}F_a\Bigr\} \\ & \leqslant 2^r\sup_m \max\Bigl\{\Bigl[\sum_{n=1}^m \sum_{a\in \Lambda_n(\alpha)} \mathbb{P}_{\xi,\infty}(a) |f_a|/2^r\Bigr]: (f_a)_{a<\alpha}\in \prod_{a<\alpha}F_a\Bigr\} \\ & \leqslant 2^r\sup_m C_{m-1}'=C.\end{align*}  This contradicts the properties of $(F_a)_{a\in \Gamma_{\xi,\infty}.A_\Omega}$ and finishes the proof.

\end{proof}

We next provide the last technical piece prior to our main results. In what follows, for a compact, Hausdorff space $K$ and $\varpi\in K$, $\delta_\varpi\in C(K)^*$ denotes the Dirac measure on $K$ given by $\langle \delta_\varpi, f\rangle=f(\varpi)$.

\begin{lemma} If $K,L$ are compact, Hausdorff spaces, $F\subset K$, and $G\subset L$, then for any $u\in B_{C(K)\widehat{\otimes}_\pi C(L)}\cap \bigcap_{(\kappa, \lambda)\in F\times G} \ker(\delta_\kappa\otimes \delta_\lambda)$ and any $\ee>0$, there exist finite sets $A,C\subset B_{C(K)}$ and finite sets $B,E\subset B_{C(L)}$ and $x,y\in B_{C(K)\widehat{\otimes}_\pi C(L)}$ such that \begin{enumerate}[(i)]\item $A\subset \text{\emph{Ann}}(F)$ and $E\subset \text{\emph{Ann}}(G)$, \item $x\in \text{\emph{co}}\{f\otimes g: f\in A, g\in B\}$, \item $y\in \text{\emph{co}}\{f\otimes g: f\in C, g\in E\}$, \item $\|u-4(x+y)\|<\ee$.  \end{enumerate}

\label{technical}
\end{lemma}

\begin{proof} Since $\|u\|\leqslant 1$, there exist $m\in\nn$, positive numbers $(w_i)_{i=1}^m$ such that $\sum_{i=1}^m w_i=1$, and functions $(f_i)_{i=1}^m\subset B_{C(K)}$, $(g_i)_{i=1}^m \subset B_{C(L)}$ such that $\|u-\sum_{i=1}^m w_i f_i\otimes g_i\|<\ee/2$.    

Let $(U_\kappa)_{\kappa\in F}$ be pairwise disoint open subsets of $K$ such that for each $\kappa\in F$, $\kappa\in U_\kappa$. Let $(V_\lambda)_{\lambda\in G}$ be pairwise disjoint open subsets of $L$ such that for each $\lambda\in G$, $\lambda\in V_\lambda$. For each $\kappa\in F$, we can fix a continuous function $e_\kappa:K\to [0,1]$ such that $e_\kappa(\kappa)=1$ and $e_\kappa|_{K\setminus U_\kappa}\equiv 0$. Similarly, for each $\lambda\in G$, we ca fix a continuous function $h_\lambda:L\to [0,1]$ such that $h_\lambda(\lambda)=1$ and $h_\lambda|_{L\setminus V_\lambda}\equiv 0$.    

Define $P_1,Q_1:C(K)\to C(K)$ by $P_1f=\sum_{\kappa\in F}f(\kappa)e_\kappa$ and $Q_1 f=f-P_1f$.  Note that $\|P_1\|=1$, so that $\|Q_1\|\leqslant 2$. Note also that $Q_1$ is a projection whose range is $\an(F)$.    

Next, define $P_2,Q_2:C(L)\to C(L)$ by $P_2g=\sum_{\lambda\in G}g(\lambda)h_\lambda$ and $Q_2 g=g-P_2g$.  Note that $\|P_2\|=1$, so that $\|Q_2\|\leqslant 2$. Note also that $Q_2$ is a projection whose range is $\an(G)$.    

Define $R,S,T:C(K)\widehat{\otimes}_\pi C(L)\to C(K)\widehat{\otimes}_\pi C(L)$ by letting $R=P_1\otimes P_2$, $S=Q_1\otimes P_2$, and $T=I\otimes Q_2$, where $I=I_{C(K)}$.    Note also that $\|R\|=\|P_1\|\|P_2\|=1$, $\|S\|=\|Q_1\|\|P_2\|\leqslant 2$, and $\|T\|=\|I\|\|Q_2\|\leqslant 2$. Note also that for any $v\in C(K)\widehat{\otimes}_\pi C(L)$, $Rv+Sv+Tv=v$. Indeed, if $J$ denotes the identity on $C(L)$, \[I\otimes J=I\otimes(P_2+Q_2)= I\otimes P_2+I\otimes Q_2 = (P_1+Q_1)\otimes P_2 +T = R+S+T.\] 

Note that for an elementary tensor $f\otimes g$, \[Rf\otimes g = \Bigl(\sum_{\kappa\in F} f(\kappa)e_\kappa\Bigr)\otimes \Bigl(\sum_{\lambda\in G}g(\lambda)h_\lambda\Bigr) = \sum_{(\kappa,\lambda)\in F\times G}\langle \delta_\kappa\otimes \delta_\lambda, f\otimes g\rangle e_\kappa\otimes h_\lambda.\]  From this we can see that $Rv=0$ for any $v\in \bigcap_{(\kappa, \lambda)\in F\times G}\ker(\delta_\kappa\otimes \delta_\lambda)$.

Let $u\in B_{C(K)\widehat{\otimes}_\pi C(L)}\cap\bigcap_{(\kappa, \lambda)\in F\times G}\ker(\delta_\kappa\otimes \delta_\lambda)$ be as in the statement of the lemma. For $\ee>0$, we can fix $m\in\nn$, positive numbers $(w_i)_{i=1}^m$ such that $1=\sum_{i=1}^m w_i$, and $(f_i)_{i=1}^m \subset B_{C(K)}$, $(g_i)_{i=1}^m\subset B_{C(L)}$ such that $\|u-\sum_{i=1}^n w_i f_i\otimes g_i\|<\ee/2$.   Let $v=\sum_{i=1}^n w_i f_i\otimes g_i$.  Since $\|R\|=1$ and $Ru=0$, $\|Rv\|\leqslant \|Ru\|+\|R\|\|v-u\|<\ee/2$.   Let $x=\frac{1}{2}Sv$ and $y=\frac{1}{2}Tv$.   Then \[\|u-2(x+y)\| = \|u-Sv-Tv\| \leqslant \|u-v\|+\|Rv\|<\ee/2+\ee/2=\ee.\]

Let \[A=\{Q_1 f_i/2: 1\leqslant i\leqslant n\}\subset B_{C(K)}\cap \an(F)\] and \[B=\{P_2 g_i: 1\leqslant i\leqslant n\}\subset B_{C(L)}\] and note that \[x=\frac{1}{2}S\sum_{i=1}^n w_i f_i\otimes g_i= \sum_{i=1}^n w_i (Q_1 f_i/2)\otimes P_2 g_i \in \text{co}\{f\otimes g: f\in A,g\in B\}.\]  

Let \[C=\{f_i: 1\leqslant i\leqslant n\}\subset B_{C(K)}\] and \[E=\{Q_2 g_i/2: 1\leqslant i\leqslant n\}\subset B_{C(L)}\cap \an(G)\] and note that \[y=\frac{1}{2}T\sum_{i=1}^n w_if_i\otimes g_i = \sum_{i=1}^n w_i f_i\otimes (Q_2 g_i/2)\in \text{co}\{f\otimes g: f\in C, g\in E\}.\] 

\end{proof}

\begin{theorem} Assume $\xi$ is an ordinal and $K,L$ are compact, Hausdorff spaces such that $\omega^\xi\leqslant \max\{CB(K), CB(L)\}<\omega^{\xi+1}$.   Then $C(K)\widehat{\otimes}_\pi C(L)$ is $\xi$-$2$-AUS-renormable.

\label{main thing}
\end{theorem}

\begin{proof} Without loss of generality, we can assume $CB(K)\geqslant CB(L)$. We can assume $K,L$ are disjoint.   Fix $k\in\nn$ such that $K^{\omega^\xi(k-1)}\neq \varnothing$ and $K^{\omega^\xi k}=\varnothing$. Fix $c>k_G 2^{k+2}$, where $k_G$ denotes Grothendieck's constant.   We claim that Player $S$ has a winning strategy in the $(\mathcal{E}^2_c, \mathfrak{s}_{C(K)\widehat{\otimes}_\pi C(L)}, \Gamma_{\xi,\infty})$ game as defined in Theorem \ref{renorm1}.  We prove this by contradiction, which means we assume Player $V$ has a winning strategy in this game.  Since we are playing the game with singleton sets $\mathfrak{s}_{C(K)\widehat{\otimes}_\pi C(L)}$, this means we may assume there exist $(u_a)_{a\in \Gamma_{\xi,\infty}.D}\subset B_{C(K)\widehat{\otimes}_\pi C(L)}$ normally weakly null such that for every $\alpha=(\zeta_a, u_a)_{a<\alpha}\in [\Gamma_{\xi,\infty}].D$, $c<\|\bigl(\sum_{a\in \Lambda_n(\alpha)} \mathbb{P}_{\xi,\infty}(a) u_a\bigr)_{n=1}^\infty\|_2^w$. Here, $D= CD(C(K)\widehat{\otimes}_\pi C(L))$, the set of finite codimensional subspaces of $C(K)\widehat{\otimes}_\pi C(L)$.

Let $\Omega=K\oplus L$ be the topological disjoint sum of $K$ and $L$. Note that $CB(\Omega)=CB(K) \in [\omega^\xi(k-1), \omega^\xi k)$.     Fix $2^k<c_1$ such that $4k_Gc_1<c$.   By Lemma \ref{breathing}, Player $S$ has a winning strategy in the $(\mathcal{I}_{c_1}, \mathfrak{s}_{C(\Omega)}, \Gamma_{\xi,\infty})$ game.  Let $\chi$ be such a strategy.    Fix $(\ee_n)_{n=1}^\infty$ such that $4k_Gc_1+\sum_{n=1}^\infty \ee_n<c$.

In the proof, for $f\in C(K)$, let $f+0\in C(\Omega)$ be the function such that $(f+0)|_K\equiv f$ and $(f+0)|_L\equiv 0$. Similarly, for $g\in C(L)$, let $0+g\in C(\Omega)$ be given by $(0+g)|_K\equiv 0$ and $(0+g)|_L\equiv g$.    For a finite subset $M$ of $\Omega$, we will denote $M=(M\cap K)\oplus (M\cap L)$. 

Let $\chi(\varnothing)=(\zeta_1, F_1\oplus G_1)$. By enlarging $F_1, G_1$ if necessary, we may assume these sets are each non-empty.  Let $Z_1=\bigcap_{(\kappa, \lambda)\in F_1\times G_1} \ker(\delta_\kappa\otimes \delta_\lambda)$ and $b_1=(\zeta, Z_1)\in \Gamma_{\xi,1}.D$.   Since $u_{b_1}\in B_{Z_1}$, by Lemma \ref{technical}, we can fix $A_1, C_1\subset B_{C(K)}$ be such that $A_1\subset \an(F_1)$, $B_1, E_1\subset B_{C(L)}$ be such that $E_1\subset \an(G_1)$, and \[x_1\in \text{co}\{f\otimes g: f\in A_1, g\in B_1\}, \] \[y_1\in \text{co}\{f\otimes g: f\in C_1, g\in E_1\},\] and $\|u_{b_1}-2(x_1+y_1)\|<\ee_1$.   Let \[N_1=\{f+0:f\in A_1\}\cup \{0+g: g\in E_1\},\] which is a finite subset of $B_{C(\Omega)}\cap \an(F_1\oplus G_1)$. 

Now assume that $b_n=(\zeta_i, B_{Z_i})_{i=1}^n\in \Gamma_{\xi,\infty}.D$ has been chosen.  Assume also that we have sets $A_1, \ldots, A_n, C_1, \ldots, C_n\subset B_{C(K)}$, $B_1, \ldots, B_n, E_1, \ldots, E_n\subset B_{C(L)}$, $F_1\oplus G_1, \ldots, F_n\oplus G_n\subset \Omega$, $x_1, \ldots, x_n, y_1, \ldots, y_n \in B_{C(K)\widehat{\otimes}_\pi C(L)}$ such that for all $1\leqslant i\leqslant n$,  \begin{enumerate}[(i)]\item $A_i\subset \an(F_i)$, \item $E_i\subset \an(G_i)$, \item $x_i\in \text{co}\{f\otimes g:f\in A_i, g\in B_i\}$, \item $y_i\in \text{co}\{f\otimes g : f\in C_i, g\in E_i\}$, and \item $\|u_{b_i}-2(x_i+y_i)\|<\ee_i$. \end{enumerate}  Let $\chi((\zeta_i, N_i)_{i=1}^n=(\zeta_{n+1}, F_{n+1}\oplus G_{n+1})$.  By enlarging $F_{n+1}, G_{n+1}$ if necessary, we can assume these sets are non-empty. Let  \[Z_{n+1}=\bigcap_{(\kappa, \lambda)\in F_{n+1}\times G_{n+1}}\ker(\delta_\kappa\otimes \delta_\lambda)\] and let $b_{n+1}=(\zeta_i, B_{Z_i})_{i=1}^{n+1}$.  By Lemma \ref{technical}, we can find $A_{n+1}, C_{n+1}\subset B_{C(K)}$ such that $A_{n+1}\subset \an(F_{n+1})$, $B_{n+1}, E_{n+1}\subset B_{C(L)}$ such that $E_{n+1}\subset \an(G_{n+1})$, $x_{n+1}\in \text{co}\{f\otimes g: f\in A_{n+1}, g\in B_{n+1}\}$, and $y_{n+1}\in \text{co}\{f\otimes g: f\in C_{n+1}, g\in E_{n+1}\}$ such that  $\|u_{b_{n+1}}-2(x_{n+1}+y_{n+1})\|<\ee_{n+1}$. Let \[N_{n+1}=\{f+0: f\in A_{n+1}\}\cup \{0+g: g\in E_{n+1}\},\] which is a finite subset of $\an(F_{n+1}\oplus G_{n+1})$.     

The end result of this process is a sequence $b_1<b_2<\ldots$ such that $b_{n+1}^-=b_n$ for all $n\in\nn$.  Let $\alpha\in [\Gamma_{\xi,\infty}].D$ be the sequence such that $\alpha|n=b_n$ for all $n\in\nn$.  Let $\varnothing=a_0<a_1<\ldots$ be such that $a_n<\alpha$ and $a_n\in MAX(\Lambda_{\xi,\infty,n}.D)$ for all $n\in\nn$.  Let $r_n=|a_n|$ for $n=0,1,2,\ldots$.   Since $\chi$ is a winning strategy for Player $S$ in the $(\mathcal{I}_{c_1}, \mathfrak{s}_{C(\Omega)}, \Gamma_{\xi,\infty})$ game, it follows that for any $(h_a)_{a<\alpha}\in \prod_{a<\alpha}N_{|a|}$, \[\sup_m \Bigl\|\sum_{n=1}^m \sum_{a\in \Lambda_n(\alpha)} \mathbb{P}_{\xi,\infty}(a) |h_a|\Bigr\|_{C(\Omega)} \leqslant c_1.\]  Therefore for any $(f_a)_{a<\alpha}\in \prod_{a<\alpha}A_{|a|}$, $f_a+0\in N_{|a|}$ for all $a<\alpha$, and  \[\sup_m \Bigl\|\sum_{n=1}^m \sum_{a\in \Lambda_n(\alpha)} \mathbb{P}_{\xi,\infty}(a) |f_a|\Bigr\|_{C(K)} = \sup_m \Bigl\|\sum_{n=1}^m \sum_{a\in \Lambda_n(\alpha)} \mathbb{P}_{\xi,\infty}(a) |f_a+0|\Bigr\|_{C(\Omega)} \leqslant c_1.\]   By Corollary \ref{kane}$(ii)$, it follows that \[\Bigl\|\Bigl(\sum_{b\in \Lambda_n(\alpha)} \mathbb{P}_{\xi,\infty}(b) x_{|b|}\Bigr)_{n=1}^\infty \Bigr\|_2^w\leqslant k_G c_1^{1/2} \leqslant k_Gc_1.\]  Similarly, using the properties of the sets $(E_n)_{n=1}^\infty$, \[\Bigl\|\Bigl(\sum_{b\in \Lambda_n(\alpha)} \mathbb{P}_{\xi,\infty}(b) y_{|b|}\Bigr)_{n=1}^\infty \Bigr\|_2^w\leqslant k_G c_1^{1/2} \leqslant k_Gc_1.\]  Therefore \begin{align*} \Bigl\|\Bigl(\sum_{b\in \Lambda_n(\alpha)} \mathbb{P}_{\xi,\infty}(b) u_{|b|}\Bigr)_{n=1}^\infty \Bigr\|_2^w & \leqslant \sum_{b<\alpha} \|u_{|b|}-2(x_{|b|}+y_{|b|})\| \\ & + 2\Bigl\|\Bigl(\sum_{b\in \Lambda_n(\alpha)} \mathbb{P}_{\xi,\infty}(b) x_{|b|}\Bigr)_{n=1}^\infty \Bigr\|_2^w +2\Bigl\|\Bigl(\sum_{b\in \Lambda_n(\alpha)} \mathbb{P}_{\xi,\infty}(b) y_{|b|}\Bigr)_{n=1}^\infty \Bigr\|_2^w \\ & <\sum_{n=1}^\infty \ee_n + 4\cdot k_G c_1<c.  \end{align*} This contradicts the properties of $(u_b)_{b\in \Gamma_{\xi,\infty}.D}$, and this contradiction finishes the proof.

\end{proof}

Here we recall the convention that if is a Banach space which fails to be Asplund, then we write $Sz(X)=\infty$. In what follows, for Banach spaces $X,Y$, we agree to the convention that  $Sz(Y)\leqslant Sz(X)$ is true if $Sz(X)=\infty$. We also agree to the convention that $\max\{Sz(X),Sz(Y)\}=\infty$ if either $Sz(X)=\infty$ or $Sz(Y)=\infty$. 

\begin{corollary} Let $K,L$ be compact, Hausdorff topological spaces.   Then \[Sz(C(K)\widehat{\otimes}_\pi C(L))=\max\{Sz(C(K)), Sz(C(L))\}.\]   \label{ms}
\end{corollary}

\begin{proof}  Since $C(K)$, $C(L)$ are each isomorphic to subspaces of $C(K)\widehat{\otimes}_\pi C(L)$, \[Sz(C(K)\widehat{\otimes}_\pi C(L))\geqslant \max\{Sz(C(K)), Sz(C(L))\}.\] 

If either $K$ or $L$ fails to be scattered, then $\max\{Sz(C(K)), Sz(C(L))\}=\infty$, and \[\infty=Sz(C(K)\widehat{\otimes}_\pi C(L))=\max\{Sz(C(K)), Sz(C(L))\}\] holds. Therefore \[Sz(C(K)\widehat{\otimes}_\pi C(L))\leqslant \max\{Sz(C(K)), Sz(C(L))\}\] by the conventions established prior to the corollary. 

Assume $K,L$ are both scattered. If both $K$ and $L$ are finite, then so is $C(K)\widehat{\otimes}_\pi C(L)$, and \[1=Sz(C(K)\widehat{\otimes}_\pi C(L))=\max\{Sz(C(K)), Sz(C(L))\}.\]

Assume $K,L$ are both scattered and at least one of $K,L$ is infinite. Without loss of generality, assume $CB(L)\leqslant CB(K)\in(\omega^\xi, \omega^{\xi+1})$.  Then \[\max\{Sz(C(K)), Sz(C(L))\}=Sz(C(K))=\omega^{\xi+1}.\]  By Theorem \ref{main thing}, $C(K)\widehat{\otimes}_\pi C(L)$ is $\xi$-$2$-AUS renormable, from which it follows that \[Sz(C(K)\widehat{\otimes}_\pi C(L))\leqslant \omega^{\xi+1}=Sz(C(K)).\]

\end{proof}

\begin{rem}\upshape
We recall that the Schreier family $\mathcal{S}_1$ is given by \[\mathcal{S}_1=\{\varnothing\}\cup \{E\subset \nn: \varnothing\neq E, |E|\leqslant \min E\}.\]   We endow the power set $2^\nn$ of $\nn$ with the Cantor topology, which is the topology induced by identifying $E$ with its indicator function $1_E\in \{0,1\}^\nn$ and endowing $\{0,1\}^\nn$ with the product topology.  It is known that $\mathcal{S}_1$ is a compact subset of $2^\nn$ whose Cantor-Bendixson index is $\omega+1$ and such that $\mathcal{S}_1^\omega=\{\varnothing\}$.  Therefore $\mathcal{S}_1$ is homeomorphic to $\omega^\omega+$, and $C(\mathcal{S}_1)$ is isometrically isomorphic to $C(\omega^\omega+)$.  The Schreier space $X_1$ is the completion of $c_{00}$ with respect to the norm \[\Bigl\|\sum_{i=1}^\infty a_ie_i\Bigr\|_{X_1} = \sup_{E\in \mathcal{S}_1} \Bigl|\sum_{i\in E} a_i\Bigr|.\]  Of course, this space is isometrically embeddable into $C(\mathcal{S}_1)$ via the map that takes $x=\sum_{i=1}^\infty a_ie_i $ to the function $f_x$ given by $f_x(E)=\sum_{i\in E}a_i$.  We note that the canonical basis of $X_1$ is unconditional and dominates the canonical dual basis in $X^*_1$.   From this it follows that $X_1\widehat{\otimes}_\pi X_1$ contains an isomorphic copy of $\ell_1$, and therefore $X_1\widehat{\otimes}_\pi X_1$ is non-Asplund.  Therefore by Corollary \ref{ms}, $Sz(C(\omega^\omega+)\widehat{\otimes}_\pi C(\omega^\omega+))=\omega^2$, $C(\omega^\omega+)$ has a subspace $X_1$ such that $X_1\widehat{\otimes}_\pi X_1$ is non-Asplund. Since $X_1$ has all of the same asymptotic smoothness properties of $C(\omega^\omega+)$, this is yet another example which illustrates the intricacies of the preservation of asymptotic smoothness properties during the formation of projective tensor products.

\end{rem}

\section{Applications}

In this section, for an ordinal $\alpha$, we let $\alpha+=[0, \alpha]$. 

\begin{theorem} If $K,L,M$ are countable, compact, Hausdorff spaces, then $C(M)$ is isomorphic to a quotient of a subspace of  $C(K)\widehat{\otimes}_\pi C(L)$ if and only if $C(M)$ embeds isomorphically into either $C(K)$ or $C(L)$. In particular, $C(\omega^\omega)\not\hookrightarrow c_0\widehat{\otimes}_\pi c_0$. 
\label{app1}
\end{theorem}

\begin{proof}There exist countable ordinals $\alpha, \beta, \gamma$ such that $C(K), C(L), C(M)$ are isomorphic to $C(\omega^{\omega^\alpha}+)$, $C(\omega^{\omega^\beta}+)$, and $C(\omega^{\omega^\gamma}+)$, respectively. Without loss of generality, assume that $\alpha \leqslant \beta$.    Moreover, $Sz(C(K))=\omega^{\alpha+1}$, $Sz(C(L))=\omega^{\beta+1}$, and $Sz(C(M))=\omega^{\gamma+1}$ \cite{Samuel}.    By Corollary \ref{ms} and our assumption that $\alpha \leqslant \beta$, \[Sz(C(K)\widehat{\otimes}_\pi Sz(L))=\max\{\omega^{\alpha+1}, \omega^{\beta+1}\}.\]

If $\gamma>\beta$, then $Sz(C(M))>Sz(C(K)\widehat{\otimes}_\pi C(L))$, and $C(M)$ does not isomorphically embed into $C(K), C(L)$, or $C(K)\widehat{\otimes}_\pi C(L)$. Here we are using the fact that the Szlenk index is an isomorphic invariant, and the Szlenk index of a Banach space cannot be less than the Szlenk index of any quotient of any of its subspaces.

If $\gamma\leqslant \beta$, then $\omega^{\omega^\gamma}+$ is a clopen subset of $\omega^{\omega^\beta}+$, and $C(M)\approx C(\omega^{\omega^\gamma}+)$ embeds isomorphically into $C(\omega^{\omega^\beta})\approx C(L)$, which embeds into $C(K)\widehat{\otimes}_\pi C(L)$.

\end{proof}

\begin{rem}\upshape The preceding result does not extend to uncountable sets $K,L,M$.  Indeed, it is known that $C(\omega_1\cdot 2+)$ does not embed into $C(\omega_1 +)$ \cite{Semadeni}. However, by \cite{Brooker}, $Sz(\omega_1+)=Sz(\omega_1\cdot 2+)=\omega^{\omega_1+1}$.  Moreover, $C(\omega_1\cdot 2+)$ embeds isomorphically into $\ell_\infty^2\widehat{\otimes}_\pi C(\omega_1+)$. Therefore we have an example with $K=\{0,1\}$, $L=\omega_1+$, and $M=\omega_1\cdot 2+$ in which $C(M)$ embeds into $C(K)\widehat{\otimes}_\pi C(L)$ but not into $C(K)$ or $C(L)$. 

However, in this example, $\omega_1+$ and $\omega_1\cdot 2+$ have the same Cantor-Bendixson index. For general compact, Hausdorff $K,L,M$, we have the following. 

\end{rem}

\begin{corollary} If $K,L,M$ are compact, Hausdorff spaces such that \[\max \{CB(K), CB(L)\}\omega < CB(M),\]  then $C(M)$ is not isomorphic to any subspace of any quotient of $C(K)\widehat{\otimes}_\pi C(L)$.   

\end{corollary}

\begin{proof} Assume $\max\{CB(K), CB(L)\}\omega<CB(M)$.    By  our conventions on the Cantor-Bendixson index, if $K$ or $L$ fails to be scattered, $\max\{CB(K), CB(L)\}=\infty$. By our conventions, $\max\{CB(K), CB(L)\}\omega< CB(M)$ implies that $K,L$ are scattered.  Therefore there exists a minimum ordinal $\xi$ such that $\max\{CB(K), CB(L)\}<\omega^{\xi+1}$.  This means $\omega^\xi \leqslant \max\{CB(K), CB(L)\}$, and $\omega^{\xi+1}<CB(M)$.  Therefore $Sz(C(M))\geqslant \omega^{\xi+2}>\omega^{\xi+1}=Sz(C(K)\widehat{\otimes}_\pi C(L))$, and $C(M)$ is no isomorphic to any subspace of any quotient of $C(K)\widehat{\otimes}_\pi C(L)$.      \end{proof}

We last show the sharpness of the exponent $2$ in Theorem \ref{main thing}.  

\begin{theorem} Let $K$ be an infinite, compact, Hausdorff, scattered topological space. Let $\xi$ be such that $\omega^\xi<CB(K)<\omega^{\xi+1}$.   Then $c_0\widehat{\otimes}_\pi C(K)$ is not $\xi$-$p$-AUS-renormable for an $2<p<\infty$.    

\end{theorem}

\begin{proof} Let $\xi$ be such that $\omega^\xi<CB(K)<\omega^{\xi+1}$.  Then $Sz(C(K))=\omega^{\xi+1}>\omega^\xi$. This means there exists $\ee>0$ such that, if $D$ is a fixed weak neighborhood basis at $0$ in $X$, there exist $\ee>0$ and a normally weakly null collection $(g_a)_{a\in \Gamma_{\xi,1}.D}\subset B_{C(K)}$ such that \[\underset{a\in MAX(\Gamma_{\xi,1}.D)}{\inf} \Bigl\|\sum_{b\leqslant a}\mathbb{P}_{\xi,1}(b) g_b\Bigr\| =\ee>0.\]   We note that $\ee$ can be taken to be $1$, but this is not necessary for the proof.  

Fix $n\in\nn$.  We next extend the collection from the previous paragraph to a normally weakly null collection $(h_a)_{a\in \Gamma_{\xi,n}.D}\subset B_{C(K)}$ by letting \[ h_{(\omega^\xi(n-1)+a_1)\smallfrown (\omega^\xi(n-2)+a_2) \smallfrown \ldots \smallfrown (\omega^\xi(n-i)+a_i)}= g_{a_i}.\] That is, $(g_a)_{a\in \Gamma_{\xi,1}.D}$ is simply repeated on each level of the tree.  Now by a standard pruning argument, we may extract from the branches of this collection some new normally weakly null collection $(f_a)_{a\in \Gamma_{\xi,n}.D}\subset B_{C(K)}$ such that for each $a\in MAX(\Gamma_{\xi,n}.D)$, $(f_b)_{b\leqslant a}$ is basic with basis constant not more than $2$. Moreover, by the property of $(g_a)_{a\in \Gamma_{\xi,1}.D}$ from previous paragraph, for any $a\in MAX(\Gamma_{\xi,n}.D)$ and $1\leqslant i\leqslant n$, \[\Bigl\|\sum_{\lambda_{i-1}(a)<b\leqslant \lambda_i(a)} \mathbb{P}_{\xi,n}(b) f_b\Bigr\|\geqslant \ee.\]

Fix $n\in\nn$.  Let $2^n=\{\pm 1\}^n$ and for $1\leqslant i\leqslant n$, let $\ee_i:2^n\to \rr$ be given by $\ee_i(\varpi_1, \ldots, \varpi_n)=\varpi_i$.  Endow $2^n$ with the uniform probability measure.    Note that $(\ee_i)_{i=1}^n$ is an orthonormal system in $L_2(2^n)$, so by Jensen's inequality, for any scalars $(a_i)_{i=1}^n$, \[\Bigl\|\sum_{i=1}^n a_i \ee_i \Bigr\|_{L_1(2^n)} \leqslant \Bigl\|\sum_{i=1}^n a_i\ee_i\Bigr\|_{L_2(2^n)} = \Bigl(\sum_{i=1}^n |a_i|^2\Bigr)^{1/2}.\]     Now define the weakly null collection $(u_a)_{a\in \Gamma_{\xi,n}.D}\subset B_{L_\infty(2^n)\widehat{\otimes}_\pi C(K)}$ by letting $u_a= \ee_i\otimes f_a$, where $1\leqslant i\leqslant n$ is such that $a\in \Lambda_{\xi,n,i}.D$.   Note that this collection is weakly null but not normally weakly null. However, this collection could be made normally weakly null by another pruning.

Fix $a\in MAX(\Gamma_{\xi,n}.D)$ and let $(\mu^*_i)_{i=1}^n$ be the Hahn-Banach extensions to the biorthogonal functionals of the basic sequence $\Bigl(\sum_{\lambda_{i-1}(a)<b\leqslant \lambda_i(a)} \mathbb{P}_{\xi,n}(b) f_b\Bigr)_{i=1}^n$.  Note that $\|\mu^*_i\|\leqslant 6/\ee$ for all $1\leqslant i\leqslant n$.  Fix a scalar sequence $(b_i)_{i=1}^n$ and let \[u=\sum_{i=1}^n b_i\sum_{\lambda_{i-1}(a)<b\leqslant \lambda_i(a)} \mathbb{P}_{\xi,n}(b)u_b= \sum_{i=1}^n b_i\ee_i\otimes \Bigl(\sum_{\lambda_{i-1}(a)<b\leqslant \lambda_i(a)} \mathbb{P}_{\xi,n}(b)f_b\Bigr).\]  Let $(a_i)_{i=1}^n$  sequence such that $1=\sum_{i=1}^n |a_i|^2=1$ and $\sum_{i=1}^n a_ib_i=\Bigl(\sum_{i=1}^n |b_i|^2\Bigr)^{1/2}$.    Let \[T=\sum_{i=1}^n b_i \ee_i\otimes \mu_i\in L_1(2^n)\widehat{\otimes}_\ee C(K)^* \subset (L_\infty(2^n)\widehat{\otimes}_\pi C(K))^*=\mathfrak{L}(C(K), L_1(2^n)).\]    We claim that $\|u^*\|\leqslant 6/\ee$. To see this, fix $f\in B_{C(K)}$.   Then \begin{align*} \|Tf\| & =\Bigl\|\sum_{i=1}^n a_i\langle \mu_i, f\rangle  \ee_i\Bigr\|_{L_1(2^n)} \leqslant \Bigl(\sum_{i=1}^n |a_i\langle \mu_i, f\rangle|^2\Bigr)^{1/2} \leqslant (6/\ee)\Bigl(\sum_{i=1}^n |a_i|^2\Bigr)^{1/2}.  \end{align*}

Note that \begin{align*} |\langle T,u\rangle| & = \sum_{i=1}^n a_ib_i \Bigl\langle \mu_i, \sum_{\lambda_{i-1}(a)<b\leqslant \lambda_i(a)} \mathbb{P}_{\xi,n}(b) f_b\Bigr\rangle \langle \ee_i, \ee_i\rangle = \sum_{i=1}^n a_ib_i = \Bigl(\sum_{i=1}^n |b_i|^2\Bigr)^{1/2}. \end{align*}   From this it follows that $\|u\| \geqslant (\ee/6)\Bigl(\sum_{i=1}^n |b_i|^2\Bigr)^{1/2}$.   Therefore we have shown that for any $a\in MAX(\Gamma_{\xi,n}.D)$, $\Bigl(\sum_{\lambda_{i-1}(a)<b\leqslant \lambda_i(a)} \mathbb{P}_{\xi,n}(b) u_b\Bigr)_{i=1}^n$ satisfies a lower $\ell_2^n$ estimate with constant $\ee/6$.   Note that the constant $\ee/6$ does not depend on $n$. 

Let $0=s_0$ and let $s_n=s_{n-1}+n$ for $n\in\nn$.  Since $L_\infty(2^n)$ is isometric to a $1$-complemented subspace of $c_0$, we can build a full normally weakly null collection $(v_b)_{b\in \Gamma_{\xi,\infty}.D}\subset B_{c_0\widehat{\otimes}_\pi C(K)}$ by letting the first $n$ level be built as in the previous paragraph for $n=1$, the next two levels built as in the previous paragraphs for $n=2$, etc.  Then for any $\alpha\in [\Gamma_{\xi,\infty}].D$, the sequence $\Bigl(\sum_{\lambda_{i-1}(\alpha)<b\leqslant \lambda_i(\alpha)} \mathbb{P}_{\xi,\infty}(b) v_b\Bigr)_{i=1}^\infty$ is not weakly $q$-summing for any $1<q<2$, since for each $n\in\nn$, $\Bigl(\sum_{\lambda_{i-1}(\alpha)<b\leqslant \lambda_i(\alpha)} \mathbb{P}_{\xi,\infty}(b) v_b\Bigr)_{i=s_{n-1}+1}^{s_n}$ is a sequence of length $n$ satisfying an $\ee/6$ lower $\ell_2^n$ estimate.

\end{proof}

If $K$ is as in the previous theorem, $C(K)$ is $\xi$-AUF renormable, which implies that it is $\xi$-$p$-AUS-renormable for all $2<p<\infty$. Since these properties are isomorphic invariants and pass to subspaces and quotients, we deduce the first part of the following corollary. The second part of the following corollary follows from the fact that if $L$ is an infinite, compact, Hausdorff space, then $c_0$ is isometrically isomorphic to a $1$-complemented subspace of $C(L)$, so $c_0\widehat{\otimes}_\pi C(K)$ embeds isomorphically into $C(L)\widehat{\otimes}_\pi C(K)$.

\begin{corollary} Let $K,L$ be compact, Hausdorff topological spaces such that $L$ is infinite and $K$ is scattered.  Then neither $c_0\widehat{\otimes}_\pi C(K)$ nor $C(L)\widehat{\otimes}_\pi C(K)$ is isomorphic to any subspace of any quotient of $C(K)$. 
\label{cake}
\end{corollary}

\begin{rem}\upshape We conclude by stating known results analogous to Corollary \ref{cake} for projective tensor products with more than two factors, where the picture is far from complete.  As shown in \cite{CD2}, for any $n\in\nn$ and any scattered, compact, Hausdorff spaces $K_1, \ldots, K_n$, \[\Bigl(\widehat{\otimes}_{\pi,i=1}^n C(K_i)\Bigr)^* = \widehat{\otimes}_{\ee,i=1}^n \ell_1(K_i)\Bigr.\]  There it was also shown that for integers $m,n$ with $m<n$, $\widehat{\otimes}_\ee^n \ell_1$ is not isomorphic to any subspace of $\widehat{\otimes}_\ee^m \ell_1$. From this it follows that if $m,n$ are not equal and if $K_1 ,\ldots, K_n$, $L_1, \ldots, L_m$ are infinite, scattered, compact, Hausdorff spaces, then $\widehat{\otimes}_{\pi,i=1}^n C(K_i)$ is not isomorphic to any quotient of $\widehat{\otimes}_{\pi,i=1}^m C(L_i)$. However, this does not yield that $\widehat{\otimes}_{\pi,i=1}^n C(K_i)$ cannot be isomorphic to a subspace of $\widehat{\otimes}_{\pi,i=1}^m C(L_i)$. 

It was also shown in \cite{CD2} that for any integer $n$, $\widehat{\otimes}_\pi^{2n}c_0$, the $2n$-fold projective tensor product of $c_0$, admits an equivalent $(n-1)$-$2$-AUS norm, and $\widehat{\otimes}_\pi^{2n-1}c_0$ admits an equivalent $(n-1)$-AUF norm.  Moreover, both of these are sharp, which implies that $\widehat{\otimes}_\pi^n c_0$ is not isomorphic to any subspace of any quotient of $\widehat{\otimes}_\pi^m c_0$ for integers $m,n$ with $m<n$.

\end{rem}


\begin{thebibliography}{HD}

\normalsize
\baselineskip=17pt

\bibitem{BP} C. Bessaga, A. Pe\l czy\'{n}ski, \emph{Spaces of continuous functions (IV)}, Studia Math., 19(1) (1960), 53-62. 

\bibitem{Brooker}  P.A.H. Brooker, \emph{Szlenk and w$^*$-dentability indices of the Banach spaces $C([0,\alpha])$,} J. Math. Anal. Appl., 399 (2013), 559-564.

\bibitem{Causey.5} R.M. Causey, \emph{The Szlenk index of convex hulls and injective tensor products}, J. Funct. Anal., 272 (2) (2017), 3375-3409.  


\bibitem{Causey1} R.M. Causey, \emph{Power type $\xi$-asymptotically uniformly smooth norms}, Trans. Amer. Math. Soc. 371 (2019), 1509-1546.

\bibitem{Causey} R.M. Causey, \emph{Power type $\xi$-Asymptotically uniformly smooth and $\xi$-asymptotically uniformly flat norms}, J. Math. Anal. Appl., 462(2), (2018), 1482-1518.  

\bibitem{beta} R.M. Causey, S.J. Dilworth, \emph{$\xi$ asymptotically uniformly smooth, $\xi$ asymptotically uniformly convex, and $(\beta)$ operators}, (2016), J. Funct. Anal., 274(10), 2906-2954.


\bibitem{CD2} R.M. Causey, S.J. Dilworth, \emph{Higher projective tensor products of $c_0$}, to appear in Studia Mathematica. 

\bibitem{DGZ} R. Deville, G. Godefroy, V. Zizler, \emph{Smoothness and renormings in
Banach spaces}, Pitman Monographs and Surveys in Pure and Applied Mathematics, vol. 64, Longman Scientific \& Technical, Harlow 1993.

 \bibitem{DK}  S. J. Dilworth, D. Kutzarova, \emph{ Kadec-Klee properties for $L(\ell_p, \ell_q)$,} Function spaces (Edwardsville, IL, 1994), 71-83, Lecture Notes in Pure and Appl. Math., 172, Dekker, New York, 1995

\bibitem{GKL} G. Godefroy, N. J. Kalton,  G. Lancien, \emph{Subspaces of $c_0(\nn)$ and Lipschitz isomorphisms}, Geom. Funct. Anal. 10(4)(2000), 798-820.


\bibitem{Gro} A. Grothendieck, \emph{R\'{e}sum\'{e} de la th\'{e}orie m\'{e}trique des produits tensoriels topologiques}, Boll.
Soc. Mat. S\~{a}o-Paulo 8 (1953), 1–79. Reprinted in Resenhas 2(4),  (1996), 401-480.

\bibitem{Grothendieck} A. Grothendieck, \emph{Produits tensoriels topologiques et espaces nucl\'{e}aires}, Mem. Amer. Math. Soc. 16 (1955). 

\bibitem{Lancien} G. Lancien, \emph{A survey on the Szlenk index and some of its applications}, RACSAM. Rev. R. Acad. Cienc. Exactas F\'{i}s. Nat. Ser. A Mat. 100(1-2) (2006), 209–235.

\bibitem{Martin} D. Martin, \emph{Borel Determinacy}, Ann. of Math., vol. 102(2)(1975),  363-371.

\bibitem{Pisier2}  G. Pisier, \emph{Martingales with values in uniformly convex Banach spaces}, Israel J. Math., 20(3)(1975), 326-350.

\bibitem{Pisier} G. Pisier, \emph{Factorization of Linear Operators and Geometry of Banach spaces}, CBMS Regional Conference Series in Mathematics, Volume 60 (1986). 

\bibitem{Ramsey} R.M. Causey, C. Doebele, \emph{A Ramsey theorem for pairs in trees,} Fund. Math., 248 (2020), 147-193

\bibitem{Rudin} W. Rudin, \emph{Continuous Functions on Compact Spaces Without Perfect Subsets}, Proc.
Amer. Math. Soc., 8(1) (1957), 39-42.

\bibitem{Samuel} C. Samuel, \emph{Indice de Szlenk des $C(K)$, S\'{e}minaire de G\'{e}om\'{e}trie des espaces de Banach,}
Vol. I-II, Publications Math\'{e}matiques de l'Universit\'{e} Paris VII, Paris (1983), 81-91.

\bibitem{Semadeni} Z. Semadeni, \emph{Banach spaces non isomorphic to their cartesian squares II}, Bulletin de
l'Acad\'{e}mie Polonaise des Sciences, S\'{e}rie Math. Astron. et Phys., 8 (1960), 81-84.

\bibitem{Unexpected} F. S\'{a}nchez, D. Pérez-García, I. Villanueva, \emph{Unexpected subspaces of tensor products,} J. Lond. Math. Soc. 74(02) (2006), 512-526.

\end{thebibliography}
\end{document}